\documentclass[11pt, leqno]{amsart} 
\usepackage{amssymb,amscd,amsfonts,amsbsy}
\usepackage{latexsym}
\usepackage{exscale}
\usepackage{amsmath,amsthm,amsfonts,amsxtra}
\usepackage{mathrsfs}
\usepackage{xcolor}
\usepackage[colorlinks=true, linkcolor=blue, citecolor=red, urlcolor=red, 
]{hyperref}
\usepackage{esint} 
\usepackage{stmaryrd}
\usepackage{pifont}
\usepackage{bm}

\usepackage[utf8]{inputenc}

\calclayout

\allowdisplaybreaks

\makeatletter
\@namedef{subjclassname@2020}{%
	\textup{2020} Mathematics Subject Classification}
\makeatother

\theoremstyle{plain}
\newtheorem{theorem}[equation]{Theorem}

\newtheorem{lemma}[equation]{Lemma}

\theoremstyle{definition}
\newtheorem{definition}[equation]{Definition}

\numberwithin{equation}{section} 

\def\A{\mathcal{A}}
\def\B{\mathcal{B}}
\def\C{\mathbb{C}}
\def\D{\mathcal{D}}

\def\F{\mathcal{F}}

\def\H{\mathcal{H}}
\def\K{\mathcal{K}}
\def\L{\mathcal{L}}
\def\bL{\mathbb{L}}
\def\M{\mathcal{M}}
\def\bM{\mathbb{M}}
\def\N{\mathbb{N}}
\def\Q{\mathcal{Q}}
\def\R{\mathbb{R}}
\def\Rd{\mathbb{R}^d}
\def\cR{\mathcal{R}}
\def\S{\mathcal{S}}
\def\bS{\mathbb{S}}

\def\V{\mathcal{V}}

\def\X{\mathbb{X}}
\def\Y{\mathbb{Y}}
\def\Z{\mathbb{Z}}
\def\c{\boldsymbol{c}}
\def\m{\boldsymbol{m}}
\def\p{\mathfrak{p}}
\def\q{\mathfrak{q}}
\def\s{\mathfrak{s}}
\def\t{\mathfrak{t}}
\def\w{\omega}
\def\a{\bm{\alpha}}
\def\b{\bm{b}}

\def\Rn{\mathbb{R}^n}

\def\dist{\operatorname{dist}}

\def\inter{\operatorname{int}}
\def\conv{\operatorname{conv}}

\def\supp{\operatorname{supp}}

\def\BMO{\operatorname{BMO}}
\def\CMO{\operatorname{CMO}}
\def\Stop{\operatorname{Stop}}
\def\LH{\operatorname{LH}}

\newcommand{\norm}[1]{\left\lVert#1\right\rVert}

\DeclareMathOperator*{\esssup}{ess\,sup}
\DeclareMathOperator*{\essinf}{ess\,inf}

\begin{document}

\title{Multilinear extrapolation on Banach function spaces}

\author{Mingming Cao}
\address{Mingming Cao\\
Institute for Advanced Study in Mathematics\\ 
Harbin Institute of Technology\\ 
Harbin 150001\\
People's Republic of China}
\email{mingming.cao@hit.edu.cn}

\thanks{The first author was supported by the National Natural Science Foundation of China (No. 12571101, 12526417, and 12671117).}

\year=2026 \month=11 \day=06

\date{\today}




\begin{abstract}
We solve an open problem on the multilinear Rubio de Francia extrapolation on Banach function spaces. 
The main result states that, under the assumption of the boundedness of the multisublinear Hardy--Littlewood maximal operator, the multivariable boundedness on Banach function spaces can be extrapolated from some weighted Lebesgue space, including the cases where some exponents may be infinite. Based on this result, we also establish 
multilinear extrapolation of boundedness on rearrangement invariant Banach function spaces, and multilinear extrapolation of compactness on general Banach function spaces. Beyond that, we give characterizations of the boundedness of the multisublinear Hardy--Littlewood maximal operator on both general Banach function spaces and certain specific function spaces. All results do not require any convexity or concavity conditions on Banach function spaces. The main ingredients in the proof are a multilinear version of the Rubio de Francia algorithm, a stopping-time argument, convex analysis, and a compactness criterion on Banach function spaces. As applications of extrapolation theorems, we obtain multilinear extrapolation for the boundedness and compactness of commutators on Banach function spaces. These results can be further applied to specific operators, which cover both the full range case including multilinear Calder\'{o}n--Zygmund operators, multilinear pseudo-differential operators, and bilinear paraproducts, and the limited range case including the bilinear Hilbert transform, multilinear rough singular integrals, and bilinear oscillatory Fourier multipliers. 
\end{abstract}

\maketitle

\section{Introduction}

\subsection{Motivation}
The Rubio de Francia extrapolation theorem \cite{RdF} is a cornerstone of modern harmonic analysis and the theory of weighted norm inequalities. The classical result states that if an operator $T$ is bounded on a weighted Lebesgue space $L^{p_0}(w_0)$ for a single exponent $p_0 \in (1, \infty)$ and all weights $w \in A_{p_0}$, then $T$ is automatically bounded on $L^p(w)$ for all $p \in (1, \infty)$ and $w \in A_p$. This profound principle implies that one just needs to verify a weighted bound at a single, convenient exponent to obtain a full range of weighted estimates. Over the decades, this result has been extensively generalized to the off-diagonal and limited range cases (see \cite{CMP-book}), and beyond Lebesgue spaces (see \cite{CMM, CH, CW, Nie23}), and found a wealth of applications in harmonic analysis, PDE, and geometric measure theory (see \cite{CMM, MMMMM, MMMM16}).
It is worth mentioning that the fundamental insight in \cite{CMM, Nie23} is that extrapolation on Banach function spaces is deeply tied to the boundedness of the Hardy--Littlewood maximal operator.

In the multilinear setting, a breakthrough was made by Lerner, Ombrosi, P\'erez, Torres, and Trujillo-Gonz\'alez \cite{LOPTT}, who introduced the natural multilinear Muckenhoupt classes $A_{\vec{p}}$, which captures the intrinsic multivariable structure of multilinear Calder\'on--Zygmund operators and the multisublinear Hardy--Littlewood maximal operator. Recently, Li, Martell, and Ombrosi \cite{LMO} successfully established Rubio de Francia extrapolation for the more general classes $A_{\vec{p}, (\vec{r}, s)}$. These classes not only characterize the boundedness of the multisublinear Hardy--Littlewood maximal operator and multilinear Calder\'on--Zygmund operators, but also naturally capture the behavior of operators with restricted ranges of boundedness, such as the bilinear Hilbert transform (see \cite{CDO, LMO}). A quantitative version of \cite{LMO} was obtained by Nieraeth \cite{Nie19} using a different method.

Beyond boundedness, the compactness of an operator is of great importance, since it determines the spectral properties and enables the application of Fredholm theory to boundary value problems (see \cite{HMT, MMMMM}). The recent work of Hyt\"onen and Lappas \cite{HL} revealed a compact counterpart of Rubio de Francia's theorem: if a linear operator $T$ is bounded on $L^{p_0}(w_0)$ for some $p_0 \in (1, \infty)$ and all weights $w \in A_{p_0}$, then the unweighted compactness of $T$ on $L^{p_1}(\Rd)$ for some $p_1 \in (1, \infty)$ implies its compactness on $L^p(w)$ for all $p \in (1, \infty)$ and $w \in A_p$. This result was extended to the multilinear setting by Cao, Olivo, and Yabuta \cite{COY}, who developed a weighted Kolmogorov--Riesz theorem and multilinear interpolation techniques to handle the quasi-Banach range $p \in (0, 1)$. Parallel to this, Lorist and Nieraeth \cite{LN} obtained the extrapolation of compactness on Banach function spaces $\X$ with certain convexity and concavity, by proving that the compactness of $T$ on some unweighted Lebesgue space yields its compactness on $\X$ whenever the Hardy--Littlewood maximal operator is bounded on $\X$ and $\X'$.

Despite these advances, it remains a significant open problem to establish multilinear extrapolation on general Banach function spaces. As explicitly pointed out by Lorist and Nieraeth \cite[Remark 1.5]{LN24} and Nieraeth \cite[Conjecture 6.5]{Nie23}, extending multilinear extrapolation (both boundedness and compactness) to the general Banach function spaces setting has not yet been achieved in its full expected generality. The main obstructions are twofold: (i) the inherent structural complexity of the multilinear Muckenhoupt classes $A_{\vec{p}, (\vec{r}, s)}$ interacting with the geometry of Banach function spaces; (ii) the severe technical difficulties in handling the quasi-Banach range $p \in (0, 1)$ and multilinear compactness, where classical tools like the Rubio de Francia algorithm and multilinear complex interpolation fail. 

Solving this open problem is deeply motivated by both theoretical necessity and practical applications: (1) The resolution of this open problem will complete the multilinear extrapolation framework. By bridging multilinear analysis with the intrinsic geometric properties of Banach function spaces, we provide a definitive answer to the questions left open in \cite{LN24} and \cite{Nie23}; (2) Resolving this open problem will allow us to automatically deduce the compactness for a broad variety of multilinear operators/commutators on weighted Lorentz spaces, variable Lebesgue spaces, and Morrey spaces, directly from their unweighted compactness.

\subsection{Main theorems}
Now we state our main theorems. The relevant definitions and notions are postponed until Section \ref{sec:pre}.

Our first result is multilinear extrapolation of boundedness on Banach function spaces. 
\begin{theorem}\label{thm:BFS}
Let $\vec{r} = (r_1, \ldots, r_m) \in [1, \infty)^m$ and $s \in (1, \infty]$. Let $\F$ be a family of $(m+1)$-tuples of nonnegative measurable functions. Let $\X_1, \ldots, \X_m$ be Banach function spaces over $\Rd$ such that
\begin{align*}
\mathcal{N}_{\vec{\X}} 
:= \|\M_{m \vec{r}, (ms)'}\|_{\X_1^{\frac1m} \times \cdots \times \X_m^{\frac1m}\times \Y' \to L^1}<\infty,
\end{align*}
where $\X = \prod_{j=1}^m \X_j$ and $\Y = \X^{\frac1m}$. If there exists some $\vec{p} = (p_1, \ldots, p_m)$ with $(\vec{r}, s) \preceq \vec{p}$ such that for all $\vec{w} \in A_{\vec{p}, (\vec{r}, s)}$,
\begin{align*}
\|f\|_{L^p_w}
\le \Phi([\vec{w}]_{\vec{p}, (\vec{r}, s)}) \prod_{j=1}^m \|f_j\|_{L^{p_j}_{w_j}}, 
\qquad (f, f_1, \ldots, f_m) \in \F,
\end{align*}
where $\Phi: [1, \infty) \to [1, \infty)$ is a non-decreasing function, then 
\begin{align*}
\|f\|_{\X} 
\le 2^{2m+1} \, \Phi \big(C_0 \mathcal{N}_{\vec{\X}}^m \big) 
\prod_{j=1}^m \|f_j\|_{\X_j}, \qquad (f, f_1, \ldots, f_m) \in \F,
\end{align*}
where $C_0 = C_0(d, m, \vec{p}, \vec{r}, s)$.
\end{theorem}

In terms of rearrangement invariant Banach function spaces, multilinear extrapolation of boundedness can be formulated as follows.

\begin{theorem}\label{thm:RIBFS}
Let $\vec{r} = (r_1, \ldots, r_m) \in [1, \infty)^m$ and $s \in (1, \infty]$. Let $\F$ be a family of $(m+1)$-tuples of nonnegative measurable functions. Let $\X_1, \ldots, \X_m$ be rearrangement invariant Banach function spaces over $\R^d$ such that $r_j < p_{\X_j}^- \le p_{\X_j}^+ < \infty$ and $\sum_{j=1}^m \frac1{p_{\X_j}^+} > \frac1s$. If there exists some $\vec{p} = (p_1, \ldots, p_m)$ with $(\vec{r}, s) \preceq \vec{p}$ such that for all $\vec{w} \in A_{\vec{p}, (\vec{r}, s)}$,
\begin{align*}
\|f\|_{L^p_w}
\le \Phi([\vec{w}]_{\vec{p}, (\vec{r}, s)}) \prod_{j=1}^m \|f_j\|_{L^{p_j}_{w_j}}, 
\quad (f, f_1, \ldots, f_m) \in \F,
\end{align*}
where $\Phi: [1, \infty) \to [1, \infty)$ is a non-decreasing function, then
\begin{align*}
\|fu\|_{\X} \lesssim \prod_{j=1}^m \|f_j u_j\|_{\X_j}, \qquad (f, f_1, \ldots, f_m) \in \F.
\end{align*}
for all $\vec{u} \in \bigcap_{\vec{q} \in \V_{\vec{\X}}} A_{\vec{q}, (\vec{r}, s)}$, where $\X = \prod_{j=1}^m \X_j$, $u = \prod_{j=1}^m u_j$, and $\V_{\vec{\X}} = \prod_{j=1}^m \{p_{\X_j}^-, p_{\X_j}^+\}$. 
\end{theorem}

On the other hand, we would like to investigate multilinear extrapolation of compactness on Banach function spaces. 
To be explicit, let us recall the compactness of multilinear operators. Given quasi-normed spaces $\mathscr{X}_1, \ldots, \mathscr{X}_m, \mathscr{X}$, an $m$-linear operator $T: \mathscr{X}_1 \times \cdots \times \mathscr{X}_m \to \mathscr{X}$ is said to be {\tt compact} if $T(B_1 \times \cdots \times B_m)$ is {\tt precompact} in $\mathscr{X}$ for all bounded sets $B_j \subset \mathscr{X}_j$, $j=1, \ldots, m$, i.e., $\overline{T(B_1 \times \cdots \times B_m)}$ is a compact subset of $\mathscr{X}$.

We establish multilinear extrapolation of compactness on Banach function spaces below. 
\begin{theorem}\label{thm:BFScpt}
Let $\vec{r} = (r_1, \ldots, r_m) \in [1, \infty)^m$ and $s \in (1, \infty]$. Let $\X_1, \ldots, \X_m$ be Banach function spaces over $\Rd$ such that 
\begin{align*}
\M_{m \vec{r}, (ms)'}: \X_1^{\frac1m} \times \cdots \times \X_m^{\frac1m}\times \Y' \to L^1(\Rd) \text{ is bounded},
\end{align*} 
where $\Y = \X^{\frac1m}$ and $\X = \prod_{j=1}^m \X_j$. If $T$ is an $m$-linear operator such that  
\begin{enumerate}
\item $T$ is bounded from $L^{p_1}_{u_1} \times \cdots \times L^{p_m}_{u_m}$ to $L^p_u$ for some $\vec{p} = (p_1, \ldots, p_m)$ with $(\vec{r}, s) \preceq \vec{p}$, and for all $\vec{u} \in A_{\vec{p}, (\vec{r}, s)}$, where $\frac1p = \sum_{j=1}^m \frac{1}{p_j} > 0$ and $u = \prod_{j=1}^m u_j$; 

\item $T$ is compact from $L^{q_1}_{v_1} \times \cdots \times L^{q_m}_{v_m}$ to $L^q_v$ for some $\vec{q} = (q_1, \ldots, q_m)$ with $(\vec{r}, s) \preceq \vec{q}$, and for some $\vec{v} \in A_{\vec{q}, (\vec{r}, s)}$, where $\frac1q = \sum_{j=1}^m \frac{1}{q_j}$ and $v = \prod_{j=1}^m v_j$;
\end{enumerate}
then $T: \X_1 \times \cdots \times \X_m \longrightarrow \X$ is compact.
\end{theorem}

As observed in Theorems \ref{thm:BFS} and \ref{thm:BFScpt}, the boundedness of multilinear maximal operators $\M_{m \vec{r}, (ms)'}$ plays a crucial role in multilinear extrapolation on Banach function spaces. For this reason, we give some characterizations of $\M_{m \vec{r}, (ms)'}: \X_1^{\frac1m} \times \cdots \times \X_m^{\frac1m} \times \Y' \to L^1(\Rd)$.

\begin{theorem}\label{thm:MM}
Let $\vec{r} \in [1, \infty)^m$ and $s \in (1, \infty]$ with $\frac1s < \sum_{j=1}^m \frac{1}{r_j}$. Let $\X_1, \ldots, \X_m$ be Banach function spaces over $\Rd$, $\X = \prod_{j=1}^m \X_j$, and $\Y = \X^{\frac1m}$. Then 
\begin{align*}
& \|\M_{m \vec{r}, (ms)'}\|_{\X_1^{\frac1m} \times \cdots \times \X_m^{\frac1m} \times \Y' \to L^1}
\\
\simeq &
\|\Lambda^*_{m\vec{r}, (ms)'}\|_{\X_1^{\frac1m} \times \cdots \times \X_m^{\frac1m} \times \Y' \to \R}
\\
\simeq &
\sup\limits_{\S: \gamma \text{-sparse}} 
\|\A_{\S, m\vec{r}, (ms)'}\|_{\X_1^{\frac1m} \times \cdots \times \X_m^{\frac1m} \times \Y' \to L^1}.
\end{align*}
Moreover, in the case $\vec{r} = (1, \dots, 1)$ and $s = \infty$, the finiteness of the above quantities is also equivalent to one of the following: 
\begin{enumerate}\setlength{\itemsep}{0.2cm}

\item[(a)] $T: \X_1 \times \cdots \times \X_m \to \X$ is bounded for all $m$-linear Calder\'{o}n--Zygmund operators $T$. 

\item[(b)] $\cR_j^k: \X_1 \times \cdots \times \X_m \to \X$ is bounded for all $m$-linear Riesz transforms $\cR_j^k$, $j=1, \ldots, m$ and $k = 1, \ldots, d$. 
\end{enumerate}
\end{theorem}

Beyond above, we examine what $\M_{m \vec{r}, (ms)'}: \X_1^{\frac1m} \times \cdots \times \X_m^{\frac1m} \times \Y' \to L^1(\Rd)$ actually means in some specific Banach function spaces.

\begin{theorem}\label{thm:example}
Let $\vec{r} \in [1, \infty)^m$ and $s \in (1, \infty]$. Let $\X_1, \ldots, \X_m$ be Banach function spaces over $\Rd$, $\X = \prod_{j=1} \X_j$, and $\Y = \X^{\frac1m}$. Denote $\mathcal{N}_{\vec{\X}} := \|\M_{m \vec{r}, (ms)'}\|_{\X_1^{\frac1m} \times \cdots \times \X_m^{\frac1m}\times \Y' \to L^1}$. The following statements hold:

\begin{itemize}\setlength{\itemsep}{0.2cm}
\item Weighted Lebesgue spaces: $\X_j = L^{p_j}_{w_j}$ with $r_j < p_j < \infty$ and $\frac1s < \frac1p = \sum_{j=1}^m \frac{1}{p_j}$. Then $\mathcal{N}_{\vec{\X}} < \infty$ if and only if $\vec{w} \in A_{\vec{p}, (\vec{r}, s)}$.

\item Weighted Lorentz spaces: $\X_j = L^{p_j, q_j}_{w_j}$ with $r_j <p_j, q_j < \infty$ and $\frac1s < \frac1p = \sum_{j=1}^m \frac{1}{p_j}$. Then $\mathcal{N}_{\vec{\X}} < \infty$ if and only if $\vec{w} \in A_{\vec{p}, (\vec{r}, s)}$.

\item Weighted variable Lebesgue spaces: $\X_j = L^{p_j(\cdot)}_{w_j}$ with $p_j(\cdot) \in \LH$, $r_j < (p_j)_- \le (p_j)_+ < \infty$, $\frac{1}{p(\cdot)} = \sum_{j=1}^m \frac{1}{p_j(\cdot)}$, and $p_+ < s$. Then $\mathcal{N}_{\vec{\X}} < \infty$ if and only if $\vec{w} \in A_{\vec{p}(\cdot), (\vec{r}, s)}$.

\item Weighted Morrey spaces: $\X_j = \L^{p_j, q_j}_{w_j}$ with $r_j < p_j \le q_j < \infty$ and $\frac1s < \frac1q = \sum_{j=1}^m \frac{1}{q_j}$. If $w_j(x) = |x|^{d \alpha_j}$ with $\beta_j \in \R$, then $\mathcal{N}_{\vec{\X}} < \infty$ if and only if $\vec{w} \in A_{\vec{q}, (\vec{r}, s)}$, and both are equivalent to $\sum_{j=1}^m \alpha_j > \frac1s - \frac1q$ and $\alpha_j < \frac{1}{r_j} - \frac{1}{q_j}$, $j=1, \ldots, m$.
\end{itemize}
\end{theorem}

\subsection{Applications of extrapolation theorems}
Next, let us present some applications of Theorems \ref{thm:BFS} and \ref{thm:BFScpt}. We begin with multilinear extrapolation of the boundedness of commutators on Banach function spaces. 

\begin{theorem}\label{thm:Tbbdd}
Let $\vec{r} \in [1, \infty)^m$ and $s \in (1, \infty]$. Let $\X_1, \ldots, \X_m$ be Banach function spaces over $\Rd$  such that
\begin{align*}
\M_{m \vec{r}, (ms)'}: \X_1^{\frac1m} \times \cdots \times \X_m^{\frac1m}\times \Y' \to L^1(\Rd) \text{ is bounded},
\end{align*} 
where $\Y = \X^{\frac1m}$ and $\X = \prod_{j=1}^m \X_j$. If there exists some $\vec{p} = (p_1, \ldots, p_m)$ with $(\vec{r}, s) \preceq \vec{p}$ such that for all $\vec{w} \in A_{\vec{p}, (\vec{r}, s)}$,
\begin{align*}
\|Tf\|_{L^p_w}
\lesssim \prod_{j=1}^m \|f_j\|_{L^{p_j}_{w_j}}, 
\end{align*}
where $\frac1p = \sum_{j=1}^m \frac{1}{p_j}$ and $w = \prod_{j=1}^m w_j$, then for every $\a \in \N^m$ and $\b \in \BMO^m$, 
\begin{align*}
\|[T, \b]_{\a}(\vec{f})\|_{\X} 
\lesssim \prod_{j=1}^m \|b_j\|_{\BMO}^{\alpha_j} \|f_j\|_{\X_j}.
\end{align*}
\end{theorem}

The second application concerns multilinear extrapolation of the compactness of commutators on Banach function spaces. 

\begin{theorem}\label{thm:Tbcpt}
Let $\vec{r} \in [1, \infty)^m$ and $s \in (1, \infty]$. Let $\a \in \N^m$ and $\b \in \CMO^m$. Let $\X_1, \ldots, \X_m$ be Banach function spaces over $\Rd$ such that 
\begin{align*}
\M_{m \vec{r}, (ms)'}: \X_1^{\frac1m} \times \cdots \times \X_m^{\frac1m}\times \Y' \to L^1(\Rd) \text{ is bounded},
\end{align*} 
where $\Y = \X^{\frac1m}$ and $\X = \prod_{j=1}^m \X_j$. If $T$ is an $m$-linear operator such that  
\begin{enumerate}
\item $T$ is bounded from $L^{p_1}_{u_1} \times \cdots \times L^{p_m}_{u_m}$ to $L^p_u$ for some $\vec{p} = (p_1, \ldots, p_m)$ with $(\vec{r}, s) \preceq \vec{p}$, and for all $\vec{u} \in A_{\vec{p}, (\vec{r}, s)}$, where $\frac1p = \sum_{j=1}^m \frac{1}{p_j} > 0$ and $u = \prod_{j=1}^m u_j$; 

\item $[T, \b]_{\a}$ is compact from $L^{q_1}(\Rd) \times \cdots \times L^{q_m}(\Rd)$ to $L^q(\Rd)$ for some $\vec{q} = (q_1, \ldots, q_m)$ with $(\vec{r}, s) \preceq \vec{q}$, where $\frac1q = \sum_{j=1}^m \frac{1}{q_j}$;
\end{enumerate}
then $[T, \b]_{\a}: \X_1 \times \cdots \times \X_m \longrightarrow \X$ is compact.
\end{theorem}

The next application  is devoted to multilinear operators with sparse structure, which encompasses a very wide range of practical applications. 

\begin{theorem}\label{thm:spBFS}
Let $\vec{r} \in [1, \infty)^m$ and $s \in (1, \infty]$ with $\frac1s < \sum_{j=1}^m \frac{1}{r_j}$. Let $\X_1, \ldots, \X_m$ be Banach function spaces over $\Rd$ such that 
\begin{align*}
\M_{m \vec{r}, (ms)'}: \X_1^{\frac1m} \times \cdots \times \X_m^{\frac1m}\times \Y' \to L^1(\Rd) \text{ is bounded},
\end{align*} 
where $\Y = \X^{\frac1m}$ and $\X = \prod_{j=1}^m \X_j$. If an $m$-linear operator $T$ satisfies 
\begin{align*}
|\langle T(\vec{f}), \, f_{m+1} \rangle|
\lesssim \Lambda_{\vec{r}, s'}^*(\vec{f}, f_{m+1}),
\end{align*}
for all $f_1, \ldots, f_m \in L_c^{\infty}(\Rd)$, then both $T$ and $[T, \b]_{\a}$ are bounded from $\X_1 \times \cdots \times \X_m$ to $\X$ for all $\a \in \N^m$ and $\b \in \BMO^m$.

In particular, the above conclusion holds whenever $T$ is one of the following operators:
\begin{itemize}

\item The full-range case $\vec{r} = (1, \ldots, 1)$ and $s = \infty$:
\begin{enumerate}
\item[---] Multilinear $\w$-Calder\'{o}n--Zygmund operators $T$ (cf. Theorem \ref{thm:CZO});

\item[---] Multilinear pseudo-differential operators $T_{\sigma}$ (cf. Theorem \ref{thm:pse}).
\end{enumerate}

\item The limited-range case $\vec{r} \in [1, \infty)^m$ and $s \in (1, \infty]$:
\begin{enumerate}
\item[---] Bilinear Hilbert transform $BH$ (cf. Theorem \ref{thm:BH});

\item[---] Multilinear multi-scale operators:
\begin{enumerate}
\item[---] Multilinear rough singular integrals $T_{\Omega}$ (cf. Theorem \ref{thm:rough}); 

\item[---] Bilinear pseudo-differential operators $T_{\sigma}$ (cf. Theorem \ref{thm:psee});

\item[---] Bilinear oscillatory Fourier multiplier $T_{\m_{a, b}^c}$ (cf. Theorem \ref{thm:osc}).
\end{enumerate}

\end{enumerate}
\end{itemize}
\end{theorem}

Let us turn to the compactness on Banach function spaces. We first focuses on the compactness of commutators. 

\begin{theorem}\label{thm:TbBTS}
Let $\vec{r} \in [1, \infty)^m$ and $s \in (1, \infty]$. Let $\X_1, \ldots, \X_m$ be Banach function spaces over $\Rd$  such that 
\begin{align*}
\M_{m \vec{r}, (ms)'}: \X_1^{\frac1m} \times \cdots \times \X_m^{\frac1m}\times \Y' \to L^1(\Rd) \text{ is bounded},
\end{align*} 
where $\Y = \X^{\frac1m}$ and $\X = \prod_{j=1}^m \X_j$. Then for every $\a \in \N^m \setminus \{0\}^m$ and $\b \in \CMO^m$, $[T, \b]_{\a}: \X_1 \times \cdots \times \X_m \longrightarrow \X$ is compact whenever $T$ is one of the following operators:
\begin{itemize}

\item The full-range case $\vec{r} = (1, \ldots, 1)$ and $s = \infty$:
\begin{enumerate}
\item[---] Multilinear $\w$-Calder\'{o}n--Zygmund operators $T$ (cf. Theorem \ref{thm:CZOcpt});

\item[---] Multilinear pseudo-differential operator $T_{\sigma}$ (cf. Theorem \ref{thm:CZOcpt});

\item[---] Bilinear Bochner-Riesz means $\B^{\delta}$ (cf. Theorem \ref{thm:CZOcpt});

\end{enumerate}

\item The limited-range case $\vec{r} \in [1, \infty)^m$ and $s \in (1, \infty]$:
\begin{enumerate}
\item[---] Multilinear rough singular integrals $T_{\Omega}$ (cf. Theorem \ref{thm:roughcpt}); 

\item[---] Multilinear Fourier multiplier $T_{\m}$ (cf. Theorem \ref{thm:FIO}).
\end{enumerate}
\end{itemize}
\end{theorem}

Finally, we deal with the compactness of a class of multilinear Calder\'{o}n--Zygmund operators on Banach function spaces. 

\begin{theorem}\label{thm:cccpt}
Let $\X_1, \ldots, \X_m$ be Banach function spaces over $\Rd$ such that 
\begin{align*}
\M_{(m, \ldots, m), 1}: \X_1^{\frac1m} \times \cdots \times \X_m^{\frac1m}\times \Y' \to L^1(\Rd) \text{ is bounded},
\end{align*} 
where $\Y = \X^{\frac1m}$ and $\X = \prod_{j=1}^m \X_j$. If $T$ is an $m$-linear Calder\'{o}n--Zygmund operators satisfying the hypotheses $\mathrm{\eqref{H1}}$--$\mathrm{\eqref{H3}}$, then $T: \X_1 \times \cdots \times \X_m \longrightarrow \X$ is compact.

In particular, the above conclusion holds whenever $T$ is one of the following operators:
\begin{itemize}
\item Bilinear continuous paraproducts $\Pi_b$ (cf. Theorem \ref{thm:CZOcpt}); 

\item Bilinear dyadic paraproducts $\Pi_{\D}$ (cf. Theorem \ref{thm:CZOcpt});

\item Bilinear pseudo-differential operator $T_{\sigma}$ (cf. Theorem \ref{thm:CZOcpt});

\item Bilinear commutators $[b, T]_{e_j}$ (cf. Theorem \ref{thm:CZOcpt}).
\end{itemize}

\end{theorem}

The rest of the paper is organized as follows. Section \ref{sec:pre} contains some preliminaries including notation, definitions, and some basic estimates. Section \ref{sec:bdd} is devoted to showing Theorems \ref{thm:BFS} and \ref{thm:RIBFS}. Then  based on Theorem \ref{thm:BFS}, the proof of Theorem \ref{thm:BFScpt} is presented in Section \ref{sec:cpt}. After that, we prove Theorems \ref{thm:MM} and \ref{thm:example} in Section \ref{sec:MM}. Finally, we sequentially give the proof of Theorems \ref{thm:Tbbdd}--\ref{thm:cccpt} in Section \ref{sec:app}.

\section{Preliminaries}\label{sec:pre}

\subsection{Notation}\label{sec:notation}
Let us introduce some notation that will be used throughout this article.
\begin{itemize}

\item Let $\N :=\{0, 1, 2, \ldots\}$ be the set of natural numbers. 

\item Let $\lfloor a \rfloor$ denote the largest integer not greater than $a$.

\item Let $\vec{v} = (v_1, \ldots, v_m)$ denote the vector of components $v_1, \ldots, v_m$.

\item Let $\|f\|_{L_w^p} := \|fw\|_{L^p}$ and $\|f\|_{L^p(w)} := \|f w^{\frac1p}\|_{L^p}$ for each $p \in (0, \infty]$ and weight $w$.

\item Let $\mathscr{S}(\Rd)$ denote the class of Schwartz functions on $\Rd$. 

\item Let $\mathscr{C}_c^{\infty}(\Rd)$ be the collection of smooth functions with compact support on $\Rd$. 

\item By a cube in $\Rd$ of sidelength $\ell$, we mean a half-open cube $Q = \prod_{i=1}^d [a_i, a_i + \ell)$, $\ell > 0$.

\item Let $\Q$ denote the collection of all cubes in $\Rd$.

\item Given a cube $Q \subset \Rd$, we denote its center by $c_Q$ and its sidelength by $\ell(Q)$. For any $\lambda>0$, we denote by $\lambda Q$ the cube with the center $c_Q$ and sidelength $\lambda \ell(Q)$. The measure of $Q$ is simply denoted by $|Q|$ no matter what dimension we are in.

\item For a bounded set $E \subset \Rd$ with $|E| > 0$, we write $\langle f \rangle_E := \fint_E f \, dx = \frac{1}{|E|} \int_E f \, dx$, and $\langle f \rangle_{E, p} := \langle |f|^p \rangle_E^{1/p}$ for any $1 \le p < \infty$.

\item We shall use $a\lesssim b$ and $a \simeq b$ to mean, respectively, that $a\leq C b$ and $0 < c \leq a/b \leq C$, where the constants $c$ and $C$ are harmless positive constants, not necessarily the same at each occurrence, which depend only on dimension and the constants appearing in the hypotheses of theorems. Moreover, if $c$ and $C$ depend on some given parameter $\eta$, which is somehow relevant, we write $a \lesssim_\eta b$ and $a \simeq_\eta b$.
\end{itemize}

\subsection{Dyadic lattices}
A {\tt dyadic child} of a cube $Q \subset \Rd$ is any of the $2^d$ congruent subcubes obtained by bisecting $Q$ in each coordinate direction. Let $\D(Q)$ consist of $Q$ and all its dyadic descendants. A collection $\mathscr{D}$ of cubes is said to be a {\tt dyadic lattice} if it satisfies
\begin{enumerate}
\item[$(1)$] For every $Q \in \mathscr{D}$, each child of $Q$ is in $\mathscr{D}$ as well;

\item[$(2)$] For every cubes $Q_1, Q_2 \in \mathscr{D}$, there exists $Q \in \mathscr{D}$ such that $Q_1, Q_2 \in \D(Q)$;

\item[$(3)$] For every compact set $K \subset \Rd$, there exists a cube $Q \in \mathscr{D}$ containing $K$.
\end{enumerate}

Let us recall the three lattice theorem (see \cite[Remark 3.2]{LN}) below. 
\begin{lemma}\label{lem:dyadic}
There exist $3^d$ dyadic lattices $\{\mathscr{D}^k\}_{k=1}^{3^d}$ such that for every cube $Q \subset \Rd$, there exists some $k \in \{1, \ldots, 3^d\}$ and $P_Q \in \mathscr{D}^k$ such that $Q \subset P_Q$ and $\ell(P_Q) \le 3 \ell(Q)$. 
\end{lemma}

Given $\gamma \in (0, 1)$, a subset $\S$ of a dyadic lattice is called {\tt $\gamma$-sparse} if for every $Q \in \S$ there exists a measurable set $E_Q \subset Q$ such that $|E_Q| \ge \gamma |Q|$ and the family $\{E_Q: Q \in \S\}$ is pairwise disjoint. 

Throughout this paper, let $\gamma \in (0, 1)$ be a fixed sparse constant. Given a $\gamma$-sparse family $\S$, the {\tt multilinear sparse operator} is defined by
\begin{align*}
\A_{\S, \vec{r}, r_{m+1}}(\vec{f}, f_{m+1})(x)
:= \sum_{Q \in \S} \prod_{j=1}^{m+1} \langle f_j \rangle_{Q, r_j} \mathbf{1}_Q(x),
\end{align*}
whenever $r_j \in [1, \infty)$, $j=1, \ldots, m+1$. We then define the {\tt multilinear maximal sparse form} by
\begin{align*}
\Lambda_{\vec{r}, r_{m+1}}^*(\vec{f}, f_{m+1})
:= \sup_{\S: \gamma \text{-sparse}} 
\sum_{Q \in \S} |Q| \prod_{j=1}^{m+1} \langle f_j \rangle_{Q, r_j},
\end{align*}
and the {\tt multilinear Hardy}--{\tt Littlewood maximal operator} by
\begin{align*}
\M_{\vec{r}, r_{m+1}}(\vec{f}, f_{m+1})(x)
:= \sup_{Q \in \Q} \prod_{j=1}^{m+1} \langle f_j \rangle_{Q, r_j} \, \mathbf{1}_Q(x).
\end{align*}

\subsection{Muckenhoupt weights}
A measurable function $w$ on $\Rd$ is called a {\tt weight} if $0 < w(x) < \infty$ for a.e.~$x \in \Rd$. Given $p \in [1, \infty)$, we define the {\tt Muckenhoupt class} $A_p$ as the collection of all weights $w$ on $\Rn$ satisfying 
\begin{equation*}
[w]_{A_p} 
:= \sup_{Q \in \Q} \langle w \rangle_Q \,
\big\langle w^{-\frac{1}{p-1}} \big\rangle_Q^{p-1} < \infty,
\end{equation*} 
where in the case $p=1$, $\langle w^{-\frac{1}{p-1}} \rangle_Q^{p-1}$ is understood as $\esssup_Q w^{-1} $. The case $A_{\infty}$ is defined as the collection of all weights $w$ on $\Rn$ satisfying 
\begin{align*}
[w]_{A_{\infty}} 
:= \sup_{Q \in \Q} \langle w \rangle_Q \, \exp \big\langle \log w^{-1} \big\rangle_Q < \infty.
\end{align*}
Then $A_{\infty} = \bigcup_{p \ge 1}A_p$ (see \cite[Corollary 7.3.4]{Gra1}). 

We are going to give the definition of the {\tt multilinear Muckenhoupt classes} $A_{\vec{p}, (\vec{r}, s)}$ introduced in \cite{LMO}. Let $m \ge 1$. Given $\vec{p} = (p_1, \ldots, p_m) \in [1, \infty]^m$, $\vec{r} = (r_1, \ldots, r_m) \in [1, \infty)^m$, and $s \in (1, \infty]$, we say that $(\vec{r}, s) \preceq \vec{p}$ whenever
\begin{align*}
r_j \le p_j, \, j=1, \ldots, m, \text{ and } p \le s, \text{ where } \frac1p := \sum_{j=1}^m \frac{1}{p_j}.
\end{align*}
Analogously, we say that $(\vec{r}, s) \prec \vec{p}$ if $r_j < p_j$ for each $j=1, \ldots, m$, and $p<s$.

\begin{definition}\label{def:Apr}
Let $\vec{p} = (p_1, \ldots, p_m) \in [1, \infty]^m$, $\vec{r} = (r_1, \ldots, r_m) \in [1, \infty)^m$ and $s \in (1, \infty]$ such that $(\vec{r}, s) \preceq \vec{p}$. Suppose that $\vec{w} = (w_1, \ldots, w_m)$ and each $w_i$ is a weight on $\Rd$. We say that $\vec{w} \in A_{\vec{p}, (\vec{r}, s)}$ if 
\begin{align*}
[\vec{w}]_{A_{\vec{p}, (\vec{r}, s)}} 
:= \sup_{Q \in \Q} \Big\langle w^{\frac{sp}{s - p}} \Big\rangle_Q^{\frac1p - \frac1s} 
\prod_{j=1}^m \Big \langle w_j^{- \frac{p_j r_j}{p_j - r_j}} \Big \rangle_Q^{\frac{1}{r_j} - \frac{1}{p_j}} < \infty,
\end{align*} 
where the supremum is taken over all cubes $Q \subset \Rd$, $w = \prod_{j=1}^m w_j$, and $\frac1p = \sum_{j=1}^m \frac{1}{p_j}$. When $p=s$, the term corresponding to $w$ needs to be replaced by $\esssup_Q w$ and, analogously, when $p_j = r_j$, the term corresponding to $w_i$ should be $\esssup_Q w_i^{-1}$. When $s=\infty$, the term corresponding to $w$ needs to be replaced by $\langle w^p \rangle_Q^{1/p}$.
\end{definition}

The following result gives a characterization of $A_{\vec{p}, (\vec{r}, s)}$ (see \cite[Lemma 5.3]{LMO}). 
\begin{lemma}\label{lem:Aprs}
Let $\vec{p} \in [1, \infty]^m$, $\vec{r} \in [1, \infty)^m$ and $s \in (1, \infty]$ such that $(\vec{r}, s) \preceq \vec{p}$. Then 
\begin{align*}
\vec{w}_{A_{\vec{p}, (\vec{r}, s)}} \quad \iff \quad
w^{\frac{1}{\frac1p - \frac1s}} 
\in A_{\frac{\frac1r - \frac1s}{\frac1p - \frac1s}}
\quad \text{and} \quad
w_j^{- \frac{1}{\frac{1}{r_j} - \frac{1}{p_j}}} 
\in A_{\frac{\frac1r - \frac1s}{\frac{1}{r_j} - \frac{1}{p_j}}}, \quad j=1, \ldots, m,
\end{align*}
where $w = \prod_{j=1}^m w_j$, $\frac1p = \sum_{j=1}^m \frac{1}{p_j}$, and $\frac1r = \sum_{j=1}^m \frac{1}{r_j}$.
\end{lemma}

Together with multilinear extrapolation, sparse domination immediately implies quantitative weighted estimates (see \cite[Corollary 2.15]{LMO} and \cite[Corollary 4.2]{Nie19}). 

\begin{theorem}\label{thm:form}
Let $r_1, \ldots, r_m \in [1, \infty)$ and $s \in (1, \infty]$. If an $m$-linear or a positive valued $m$-sublinear operator $T$ satisfies 
\begin{align*}
|\langle T(\vec{f}), \, f_{m+1} \rangle|
\lesssim \Lambda_{\vec{r}, s'}^*(\vec{f}, f_{m+1}),
\end{align*}
for all $f_1, \ldots, f_m \in L_c^{\infty}(\Rd)$, then 
\begin{align*}
\|T\|_{L^{p_1}(w_1^{p_1}) \times \cdots \times L^{p_m}(w_m^{p_m}) \to L^p(w^p)}
\lesssim [\vec{w}]_{A_{\vec{p}, (\vec{r}, s)}}^{\max\Big\{\frac{\frac{1}{r_1}}{\frac{1}{r_1} - \frac{1}{p_1}}, \ldots, \frac{\frac{1}{r_m}}{\frac{1}{r_m} - \frac{1}{p_m}}, \frac{\frac{1}{s'}}{\frac1p - \frac1s} \Big\}},
\end{align*}
for all $\vec{p} \in (1, \infty]$ with $(\vec{r}, s) \prec \vec{p}$, and for all $\vec{w} \in A_{\vec{p}, (\vec{r}, s)}$.
\end{theorem}

The following result implies the openness of $A_{\vec{p}, (\vec{r}, s)}$, which should be compared with \cite[Lemma 2.6]{COY}.

\begin{lemma}\label{lem:open}
Let $w_1, \ldots, w_{m+1}$ be weights such that $\prod_{j=1}^{m+1} w_j =1$, and $p_1, \ldots, p_{m+1} \in (0, \infty)$ with $\sum_{j=1}^{m+1} \frac{1}{p_j} = \alpha > 0$. If $\sup_{Q \in \Q} \prod_{j=1}^{m+1} \langle w_j^{-1} \rangle_{Q, p_j} < \infty$, then there exists some $\lambda>1$ such that
\begin{align*}
\sup_{Q \in \Q} \prod_{j=1}^{m+1} \langle w_j^{-1} \rangle_{Q, \lambda p_j}
\lesssim \sup_{Q \in \Q} \prod_{j=1}^{m+1} \langle w_j^{-1} \rangle_{Q, p_j}. 
\end{align*}
\end{lemma}

\begin{proof}
Let $q_j = \alpha p_j > 1$, $j=1, \ldots, m+1$. By H\"{o}lder's inequality and the fact $\frac{q_j-1}{p_j} = \sum_{i \ne j} \frac{1}{p_i}$, 
\begin{align*}
\bigg[\fint_Q (w_j^{-p_j})^{-\frac{1}{q_j -1}} \, dx \bigg]^{q_j -1}
= \bigg[\fint_Q \bigg(\prod_{i \ne j} w_i^{-1} \bigg)^{\frac{p_j}{q_j -1}} dx \bigg]^{\frac{q_j -1}{p_j} p_j}
\le \prod_{i \ne j} \bigg( \fint_Q w_i^{-p_i} \, dx \bigg)^{\frac{p_j}{p_i}},
\end{align*}
which gives 
\begin{align*}
\bigg[\fint_Q w_j^{-p_j} \, dx \bigg]^{\frac{1}{p_j}}
\bigg[\fint_Q (w_j^{-p_j})^{-\frac{1}{q_j -1}} \, dx \bigg]^{\frac{q_j -1}{p_j}}
\le \prod_{i=1}^{m+1} \bigg( \fint_Q w_i^{-p_i} \, dx \bigg)^{\frac{1}{p_i}} 
< \infty.
\end{align*}
This means $w_j^{-p_j} \in A_{q_j-1}$. Then by the reverser H\"{o}lder's inequality (cf. \cite[Theorem 7.2.2]{Gra1}), there exist some $\lambda_j > 1$ such that $\langle w_j^{-p_j} \rangle_{Q, \lambda_j} \lesssim \langle w_j^{-p_j} \rangle_{Q}$ for all cubes $Q$. Choosing $\lambda = \min\{\lambda_1, \ldots, \lambda_{m+1}\}$, we obtain
\begin{align*}
\prod_{j=1}^{m+1} \langle w_j^{-1} \rangle_{Q, \lambda p_j}
\le \prod_{j=1}^{m+1} \langle w_j^{-p_j} \rangle_{Q, \lambda_j}^{\frac{1}{p_j}}
\lesssim \prod_{j=1}^{m+1} \langle w_j^{-p_j} \rangle_Q^{\frac{1}{p_j}}
= \prod_{j=1}^{m+1} \langle w_j^{-1} \rangle_{Q, p_j}.
\end{align*}
This completes the proof. 
\end{proof}

\subsection{$\BMO$ and $\CMO$ spaces}\label{sec:BMO}
We say that a locally integrable function $b \in \BMO$ if 
\begin{equation*}
\|b\|_{\BMO} := \sup_{Q \in \Q} \fint_{Q} |b(x) - \langle b \rangle_Q| \, dx < \infty,
\end{equation*}
Let $\CMO$ denote the closure of $\mathscr{C}_c^{\infty}(\Rd)$ in $\BMO$. Additionally, the space $\CMO$ is endowed with the norm of $\BMO$. 

It was show in \cite[p. 253]{CR} that
\begin{align}\label{eq:AiBMO}
w \in A_{\infty} \quad \Longrightarrow \quad 
\log w \in \BMO.
\end{align}
A consequence of the John--Nirenberg inequality is the following (cf. \cite[Lemma 3.5]{BMMST}):
\begin{equation}\label{eq:JN}
\begin{aligned}
& \forall b \in \BMO \text{ and } p \in (1, \infty),   
\exists \lambda_0 \in (0, 1) \text{ such that } 
\\
& \langle e^{|\lambda| |b - \langle b \rangle_Q|} \rangle_Q \lesssim 1 
\, \text{ and } \, 
e^{\lambda b} \in A_p \text{ for all } |\lambda| \le \lambda_0.
\end{aligned}
\end{equation}

Let $T$ be an $m$-linear operator from $\mathscr{X}_1  \times \cdots \times \mathscr{X}_m$ into $\mathscr{X}$, where $\mathscr{X}_1, \ldots, \mathscr{X}_m, \mathscr{X}$ are some quasi-normed spaces. For $(f_1, \ldots, f_m) \in \mathscr{X}_1 \times \cdots \times \mathscr{X}_m$ and for a measurable vector $\b = (b_1, \ldots, b_m)$, and $1 \leq j \leq m$, we define, whenever it makes sense, the first order commutators 
\begin{align*}
[T, \b]_{e_j} (f_1, \ldots, f_m) 
:= b_j T(f_1, \ldots, f_j, \ldots, f_m) - T(f_1, \ldots, b_j f_j, \ldots, f_m);
\end{align*}
we denoted by $e_j$ the basis element taking the value $1$ at component $j$ and $0$ in every other component, therefore expressing the fact that the commutator acts as a linear one in the $j$-th variable and leaving the rest of the entries of 
$(f_1, \ldots, f_m)$ untouched. Then, if $k \in \N_+$, we define 
\begin{align*}
[T, \b]_{ke_j}=[\cdots[[T, \b]_{e_j}, \b]_{e_j} \cdots, \b]_{e_j}, 
\end{align*}
where the commutator is performed $k$ times. Finally, if $\a = (\alpha_1, \ldots, \alpha_m) \in \N^m$ is a multi-index, we define 
\begin{align*}
[T, \b]_{\a} = [\cdots[[T, \b]_{\alpha_1 e_1}, \b]_{\alpha_2 e_2} \cdots, \b]_{\alpha_m e_m}.
\end{align*}

\subsection{Banach function spaces}
Let $(\Sigma, \mu)$ be a measure space, and $\bM_{\mu}$ be the collection of measurable functions on $(\Sigma, \mu)$. 

\begin{definition}\label{def:BFS}
We say that a mapping $\|\cdot\| : \bM_{\mu} \to [0, \infty]$ is a {\tt function norm} provided following properties are satisfied for all $f, g \in \bM_{\mu}$: 
\begin{list}{\textup{(\theenumi)}}{\usecounter{enumi}\leftmargin=1cm \labelwidth=1cm \itemsep=0.2cm 
		\topsep=0.2cm \renewcommand{\theenumi}{\arabic{enumi}}}
\item\label{bfs-1} $\|f\|=\big\| |f|\big\|$ and $\|f\|=0$ if and only if $f=0$ $\mu$-a.e.;

\item\label{bfs-2} $\|f+g\| \leq\|f\| + \|g\|$;

\item\label{bfs-3} $\|\lambda f\| = |\lambda| \|f\|$ for every $\lambda \in \R$;

\item\label{bfs-4} If $|f|\leq|g|$ $\mu$-a.e., then $\|f\| \leq \|g\|$;

\item\label{bfs-5} If a sequence $\{f_k\}_{k \ge 1} \subset \bM_{\mu}$ satisfies $0 \le f_k \uparrow f \in \bM_{\mu}$ a.e. and $\sup_{k \ge 1} \|f_k\|_{\X} < \infty$, then $f \in \X$ with $\|f\|_{\X} = \sup_{k \ge 1} \|f_k\|_{\X}$;

\item\label{bfs-6} For every measurable set $E \subset \Sigma$ with $\mu(E)>0$, there exists a measurable $F \subset E$ such that $\mu(F) > 0$ and $\mathbf{1}_F \in \X$. 
\end{list}
\end{definition}

Given a function norm $\|\cdot\|$, the set 
\begin{equation*}
\X = \{f \in \bM_{\mu}: \|f\| < \infty\} 
\end{equation*} 
is called a {\tt Banach function space} over $(\Sigma, \mu)$. In such a scenario, we shall write $\|\cdot\|_{\X}$ in place of $\|\cdot\|$ in order to emphasize the connection between the function norm $\|\cdot\|$ and its associated function space $\X$. Then $(\X, \|\cdot\|_{\X})$ is a Banach space. Moreover, $(\X, \|\cdot\|_{\X})$ is called a {\tt quasi-Banach function space} over $(\Sigma, \mu)$ if the property \eqref{bfs-2} is replaced by 
\begin{align*}
\|f + g\|_{\X} \leq C_0 (\|f\|_{\X} + \|g\|_{\X}), \quad \text{ for some } C_0 \ge 1.
\end{align*}

Given a Banach function space $\X$ over $(\Sigma, \mu)$, its associate space $\X'$ is defined by means of the function form 
\begin{equation*}
\|f\|_{\X'} 
:= \sup \bigg\{\int_{\Sigma} |f(x) g(x)| \, d\mu(x): \,g \in \X, \, \, \|g\|_{\X} \leq 1 \bigg\}.
\end{equation*}
With this definition, $\X'$ is also a Banach function space. Then it follows from the definition of $\X'$ that the following generalized Hölder's inequality holds: 
\begin{equation}\label{eq:Holder}
\int_{\Sigma}|f(x)g(x)| \, d\mu(x) 
\leq \|f\|_{\X} \|g\|_{\X'},
\quad f \in \X \, \text{ and } \, g \in \X'.
\end{equation}
It turns out that $\X = (\X')' =: \X''$ (cf. \cite[Theorem 71.1]{Zaa}). Therefore, one has 
\begin{equation}\label{eq:fX-norm}
\|f\|_{\X}
= \sup_{\|g\|_{\X'} \leq 1} \int_{\Sigma} |f(x)g(x)| \, d\mu(x).
\end{equation}

Given $r \in (0, \infty)$ and a Banach function space $\X$ over $(\Sigma, \mu)$, we define the spaces $\X^r$ by
\begin{equation}\label{Xr}
\X^r := \{f \in \bM_{\mu}: |f|^{\frac1r} \in\X\} 
\quad \text{and} \quad 
\|f\|_{\X^r} := \| |f|^{\frac1r}\|_{\X}^r. 
\end{equation}
Given Banach function spaces $\X_1, \ldots, \X_m$, their pointwise product $\X := \prod_{j=1}^m \X_j$ is equipped with the function norm
\begin{equation*}
\|f\|_{\X}
:= \inf \bigg\{\prod_{j=1}^m \|f_j\|_{\X_j}: \,
|f| \le \prod_{j=1}^m f_j, \, 0 \le f_j \in \X_j, \, j=1, \ldots, m \bigg\}.
\end{equation*}
If $\alpha_1, \ldots, \alpha_m >0$ with $\sum_{j=1}^m \alpha_j \le 1$, then $\prod_{j=1}^m \X_j^{\alpha_j}$ is a Banach function space, see \cite{Schep}. Consequently, if $\X_1, \ldots, \X_m$ are Banach function spaces, then
\begin{equation*}
\X := \prod_{j=1}^m \X_j 
\quad \text{ and } \quad
\Y := \X^{\frac1m} = \prod_{j=1}^m \X_j^{\frac1m}
\end{equation*}
are respectively a quasi-Banach function space and a Banach function space.

\begin{lemma}\label{lem:vv}
Let $\X$ be a quasi-Banach function space such that $\X^{\frac1m}$ is a Banach function space. Then 
\begin{align*}
\bigg\|\sum_{k=1}^{\infty} |f_k| \bigg\|_{\X} 
\le \bigg(\sum_{k=1}^{\infty} \|f_k\|_{\X}^{\frac1m} \bigg)^m
\quad \text{ and } \quad
\bigg\|\prod_{k=1}^m |g_k|^{\frac1m} \bigg\|_{\X} 
\le \prod_{k=1}^m \|g_k\|_{\X}^{\frac1m}.
\end{align*}
\end{lemma}

\begin{proof}
First, we may assume that $\sum_{k=1}^{\infty} \|f_k\|_{\X}^{\frac1m} < \infty$. By the ideal property, 
\begin{align*}
\bigg\|\sum_{k=1}^N |f_k| \bigg\|_{\X}^{\frac1m} 
= \bigg\| \bigg[\sum_{k=1}^N |f_k| \bigg]^{\frac1m} \bigg\|_{\X^{\frac1m}} 
\le \bigg\| \sum_{k=1}^N |f_k|^{\frac1m} \bigg\|_{\X^{\frac1m}} 
\le \sum_{k=1}^N \| |f_k|^{\frac1m}\|_{\X^{\frac1m}} 
\le \sum_{k=1}^{\infty} \|f_k\|_{\X}^{\frac1m},
\end{align*}
which together with the Fatou property gives the first estimate. To show the second one, we may assume that $\|g_k\|_{\X} > 0$ for each $k=1, \ldots, m$, otherwise it is trivial. Letting $h_k = (|g_k|/\|g_k\|_{\X})^{\frac1m}$, we have  
\begin{align*}
\bigg\|\prod_{k=1}^m |h_k|^{\frac1m} \bigg\|_{\X^{\frac1m}} 
\le \bigg\| \frac1m \sum_{k=1}^m h_k \bigg\|_{\X^{\frac1m}} 
\le \frac1m \sum_{k=1}^m \| h_k\|_{\X^{\frac1m}} = 1,
\end{align*}
which implies the second estimate. 
\end{proof}

To define a rearrangement invariant Banach function space, let $\mu_{f}$ denote the distribution function of $f \in \bM_{\mu}$: 
\begin{equation*}
\mu_f(\lambda) := \mu(\{x \in \Sigma: |f(x)| > \lambda\}), \qquad \lambda \in [0, \infty). 
\end{equation*}
A Banach function space $\X$ over $(\Sigma, \mu)$ is {\tt rearrangement invariant} if $\|f\|_{\X} = \|g\|_{\X}$ for every pair of functions $f, g\in \X$ such that $\mu_f = \mu_g$. This means that the function norm of $f$ in $\X$  depends only on its distribution function. Note that if $(\Sigma, \mu)$ is nonatomic, then $\X$ is rearrangement invariant if and only if $\X'$ is (cf. \cite[Corollary 4.4, p. 60]{BS}). 

For each $f \in \bM_{\mu}$, the decreasing rearrangement of $f$ with respect to $\mu$ is defined as 
\begin{equation*}
f^{*}_{\mu}(t) 
:= \inf\{\lambda \geq 0 : \mu_f(\lambda) \leq t\}, \quad t \in[0, \mu(\Sigma)].
\end{equation*} 
Note that the functions $f$ and $f^*_{\mu}$ have the same distribution function. One remarkable consequence of this is Luxemburg representation theorem: if $\X$ is rearrangement invariant Banach function space over $(\Sigma, \mu)$, then there exists a rearrangement invariant Banach function space $\overline{\X}$ over $[0, \mu(\Sigma))$ such that $f \in \X$ if and only if  $f^*_{\mu} \in \overline{\X}$ and $\|f^{*}_{\mu}\|_{\overline{\X}} = \|f\|_{\X}$ (cf.  \cite[Theorem 4.10, p.\,62]{BS}). This allows us to define the {\tt Boyd indices} by 
\begin{equation*}
p_{\X}^- := \lim_{t \to \infty} \frac{\log t}{\log\norm{D_t}_{\overline{\X} \to \overline{\X}}}
\quad \text{ and } \quad
p_{\X}^+ := \lim_{t\to 0^{+}}\frac{\log t}{\log\norm{D_t}_{\overline{\X}\to\overline{\X}}}, 
\end{equation*}
where $D_t: \overline{\X} \to \overline{\X}$ is the dilation operator defined by
\begin{equation*}
D_t f(s) := \left\lbrace
\begin{array}{lcl}
f(s/t), &\text{ if}& s\in[0, \mu(\Sigma) \min\{1, t\}],
\\[4pt]
0, &\text{ if}& s\in(\mu(\Sigma)\min\{1,t\}, \, \mu(\Sigma)].
\end{array}
\right. 
\end{equation*}
It is not hard to see that
\begin{align*}
1 \leq p_{\X}^- \leq p_{\X}^+ \leq \infty, \quad 
(p_{\X}^-)' = p_{\X'}^+, \quad
(p_{\X}^+)' = p_{\X'}^-, \quad
p_{\X^r}^- = p_{\X}^-/r,
\quad\text{and}\quad  
p_{\X^r}^+ = p_{\X}^+/r.
\end{align*}

We now present examples of Banach function spaces. From now on, we restrict our attention to Banach function spaces over $\Rd$. Let $\bM$ be the collection of all measurable functions on $\Rd$.

\subsection{Weighted Lorentz spaces} 
Given $p, q \in (0, \infty]$ and a weight $w$ on $\Rd$, the {\tt weighted Lorentz space} $L^{p, q}_w = L^{p, q}(w^p)$ is defined as the collection of all functions $f \in \bM$ with $\|f\|_{L^{p, q}_w} < \infty$, where
\begin{equation*} 
\|f\|_{L^{p, q}_w} := 
\begin{cases}
\displaystyle
\bigg(\int_0^{\infty} \big( t^{\frac1p} f_{w^p}^*(t) \big)^q \frac{dt}{t} \bigg)^{\frac1q}, & \text{if } q < \infty,
\vspace{0.3cm}
\\ 
\displaystyle
\sup_{t > 0} t^{\frac1p} f_{w^p}^*(t), & \text{if } q = \infty.
\end{cases}
\end{equation*}
Denote $L^{p, q}(\Rd) = L^{p, q}_w$ when $w \equiv 1$. When $p=q$, $L^{p, q}_w$ is the weighted Lebesgue spaces $L^p(w^p)$. The space $L^{p, q}(\Rd)$ is a rearrangement invariant Banach function space if and only if $p \in (1, \infty)$ and $q \in [1, \infty]$ or $(p, q) \in \{(1, 1), (\infty, \infty)\}$ (possibly after switching to an equivalent norm); in this case $\X = L^{p, q}(\Rd)$, the Boyd indices are $p_{\X}^- = p_{\X}^+ = p$, see \cite[Theorem 4.6, p.\,219]{BS}. 

It was shown in \cite[Theorem 4]{CHK} that for all $p, q \in (1, \infty)$,
\begin{align}\label{eq:Apq}
w \in A_p \quad \iff \quad
\sup_{Q \in \Q} |Q|^{-1} \|\mathbf{1}_Q\|_{L^{p, q}(w)} 
\|w^{-1} \mathbf{1}_Q\|_{L^{p', q'}(w)} < \infty.
\end{align}
Define the space
\begin{align*}
\bL^{p, q}(u; v)
:= \{h u^{\frac1p}: h \in L^{p, q}(uv)\}
\quad \text{and} \quad
\|f\|_{\bL^{p, q}(u; v)}
:= \|fu^{-\frac1p}\|_{L^{p, q}(uv)},
\end{align*}
and the weighted maximal operator
\begin{align*}
M_r^v f(x) 
:= \sup_{x \in Q \in \Q} \bigg(\frac{1}{v(Q)} \int_Q |f|^r \, dv \bigg)^{\frac1r}.
\end{align*}
Denote $M^v := M_1^v$.

\begin{lemma}\label{lem:Lorentz}
Let $p, q \in (1, \infty)$, $u$ and $v$ be weights such that $\log u \in \BMO$ and $v \in A_{\infty}$. Then $M_r^v$ is bounded on $\bL^{p, q}(u; v)$ for all $r \in [1, p)$.
\end{lemma}

\begin{proof}
We borrow the idea in \cite[Example 2.46]{CMM}. Fix $r \in [1, p)$. By \eqref{eq:JN}, for every $s \in [1, \infty)$, there exists some $\lambda_0 \in (0, 1)$ such that $u^{\lambda} \in A_s$ for all $|\lambda| \le \lambda_0$. Then choose $r < p_0 < p < p_1 < \infty$ with $|1-p_i/p| < \lambda_0$ so that $u^{1-p_i/p} \in A_{p_i/r}$. By the condition $v \in A_{\infty}$ and \cite[Chapter I, Theorem 9]{ST}, the operator $T f := u^{-\frac1p} M_r^v(f u^{\frac1p})$ satisfies
\begin{align*}
\|Tf\|_{L^{p_i}(uv)}
= \|M^v(|f u^{\frac1p}|^r)\|_{L^{p_i/r}(u^{1-p_i/r} dv)}^{\frac1r} 
\lesssim \||f u^{\frac1p}|^r\|_{L^{p_i/r}(u^{1-p_i/r} dv)}^{\frac1r} 
= \|f\|_{L^{p_i}(uv)},
\end{align*}
for each $i=0, 1$. Then the interpolation (cf. \cite[Theorem 1.4.19]{Gra1}) gives $T$ is bounded on $L^{p, q}(uv)$, which is equivalent to the boundedness of $M_r^v$ on $\bL^{p, q}(u; v)$.
\end{proof}

We next provide some elementary estimates on weighted Lorentz spaces.  

\begin{lemma}\label{lem:uupq}
Let $u$ be a weight on $\Rd$, $p, q \in (0, \infty)$, $0 < r < \min\{p, q\}$, $\alpha \ne 0$, and $\beta \in \R$ with $p \alpha + \beta \ne 0$. Then for every cube $Q \subset \Rd$,
\begin{align}
\label{eq:uupq-1}
\|u^{\alpha} \mathbf{1}_Q\|_{L^{p, q}(u^{\beta})}
& \simeq \|u^{p \alpha + \beta} \mathbf{1}_Q\|_{L^{1, q/p}}^{\frac1p}, 
\\
\label{eq:uupq-2}
\sup_{\|f\|_{L^{p, q}(u)} \le 1} \|f \mathbf{1}_Q\|_{L^r}
& \simeq \|u^{-\frac{r}{p-r}} \mathbf{1}_Q\|_{L^{1, \frac{(q/r)'}{(p/r)'}}}^{\frac1r - \frac{1}{p}}.
\end{align} 
Moreover, if $\log u \in \BMO$ and $0 < p_0 < p$, then 
\begin{align}\label{eq:uupq-3}
\|f u^{\frac1p} \mathbf{1}_Q\|_{L^{p_0}}
\lesssim |Q|^{\frac{1}{p_0} - \frac1p} \|f \mathbf{1}_Q\|_{L^{p, q}(u)}.
\end{align}
\end{lemma}

\begin{proof}
Note that the inequality \eqref{eq:uupq-2} follows from \eqref{eq:uupq-1} and that
\begin{align*}
\sup_{\|f\|_{L^{p, q}(u)} \le 1} \|f \mathbf{1}_Q\|_{L^r}^r
= \sup_{\|g\|_{L^{p/r, q/r}(u)} \le 1} \|g u^{-1} \mathbf{1}_Q\|_{L^1(u)}
= \|u^{-1} \mathbf{1}_Q\|_{L^{(p/r)', (q/r)'}(u)}.
\end{align*}
Fix a cube $Q \subset \Rd$. Letting $Q_j :=\{x \in Q: 2^j < u(x) \le 2^{j+1}\}$, we have
\begin{align*}
u^{\alpha} \mathbf{1}_Q \simeq \sum_{j \in \Z} 2^{j \alpha} \mathbf{1}_{Q_j}
\quad \text{ and } \quad
u^{\beta}(Q_j) \simeq 2^{j \beta} |Q_j|,
\end{align*}
which along with \cite[Example 1.4.8]{Gra1} gives
\begin{align*}
\|u^{\alpha} \mathbf{1}_Q\|_{L^{p, q}(u^{\beta})}^q 
\simeq \sum_{j \in \Z} 2^{j \alpha q} 
\bigg(\sum_{k: \alpha k \ge \alpha j} u^{\beta}(Q_k) \bigg)^{\frac{q}{p}}
\simeq \sum_{j \in \Z} 2^{j \alpha q} 
\bigg(\sum_{k: \alpha k \ge \alpha j} 2^{k \beta} |Q_k| \bigg)^{\frac{q}{p}}.
\end{align*}
Denote $\gamma_k := 2^{k(p \alpha + \beta)} |Q_k|$. If $\alpha > 0$, then 
\begin{align*}
\|u^{\alpha} \mathbf{1}_Q\|_{L^{p, q}(u^{\beta})}^q 
\simeq \sum_{j \in \Z}  
\bigg(\sum_{k \ge j} 2^{-(k-j) p \alpha} \gamma_k \bigg)^{\frac{q}{p}}
= \sum_{j \in \Z}  
\bigg(\sum_{k \ge 0} 2^{-k p |\alpha|} \gamma_{k+j} \bigg)^{\frac{q}{p}}
\simeq \sum_{j \in \Z} \gamma_j^{\frac{q}{p}}.
\end{align*}
If $\alpha < 0$, then 
\begin{align*}
\|u^{\alpha} \mathbf{1}_Q\|_{L^{p, q}(u^{\beta})}^q 
\simeq \sum_{j \in \Z}  
\bigg(\sum_{k \le j} 2^{(j-k) p \alpha} \gamma_k \bigg)^{\frac{q}{p}}
= \sum_{j \in \Z}  
\bigg(\sum_{k \ge 0} 2^{-k p |\alpha|} \gamma_{k-j} \bigg)^{\frac{q}{p}}
\simeq \sum_{j \in \Z} \gamma_j^{\frac{q}{p}}.
\end{align*}
By the condition $p \alpha + \beta \ne 0$, the above estimates applied to $(1, q/p, p\alpha + \beta, 0)$ in place of $(p, q, \alpha, \beta)$ yield 
\begin{align*}
\|u^{p\alpha + \beta} \mathbf{1}_Q\|_{L^{1, q/p}}^{q/p} 
\simeq \sum_{j \in \Z} \gamma_j^{\frac{q}{p}},
\end{align*}
which justifies \eqref{eq:uupq-1}. 

To show \eqref{eq:uupq-3}, let $w := \exp(\log u - \langle \log u \rangle_Q)$. In view of \eqref{eq:JN}, there exists some $\lambda_0 \in (0, 1)$ such that 
\begin{align}\label{eq:vam}
\langle w^{\lambda} \rangle_Q \lesssim 1, \quad \forall |\lambda| \le \lambda_0.
\end{align} 
Fix $|\lambda| \le \lambda_0$ and take $s>1$ so that $\frac1s + \frac{1}{p/p_0} = 1$. If $w(x) < t^{-s}$, we use the identity $w = w^{\lambda} w^{1- \lambda}= w^{-\lambda} w^{1+\lambda}$ to deduce 
\begin{align*}
w(x) < w(x)^{\lambda} t^{-s(1-\lambda)}
\quad \text{ and } \quad
w(x) < w(x)^{-\lambda} t^{-s(1+\lambda)},
\end{align*}
which together with \eqref{eq:vam} implies
\begin{align*}
w(\{x \in Q: w(x)^{-1/s} > t\})
\lesssim \min\{t^{-s(1-\lambda)}, t^{-s(1+\lambda)} \} |Q|.
\end{align*}
 and hence, 
\begin{align*}
\|w^{-\frac1s} \mathbf{1}_Q\|_{L^{s, 1}(w)}
\simeq \int_0^{\infty} w(\{x \in Q: w(x)^{-1/s} > t\})^{\frac1s} \, dt
\lesssim |Q|^{\frac1s}.
\end{align*}
This further indicates
\begin{align*}
\|f w^{\frac1p} \mathbf{1}_Q\|_{L^{p_0}}^{p_0}
& = \||f|^{p_0} w^{-\frac1s} \mathbf{1}_Q\|_{L^1(w)}
\lesssim \||f|^{p_0} \mathbf{1}_Q\|_{L^{p/p_0, \infty}(w)}
\|w^{-\frac1s} \mathbf{1}_Q\|_{L^{s, 1}(w)}
\\
& \lesssim \|f \mathbf{1}_Q\|_{L^{p, \infty}(w)}^{p_0} |Q|^{\frac1s}
= \|f \mathbf{1}_Q\|_{L^{p, \infty}(w)}^{p_0} |Q|^{p_0(\frac{1}{p_0} - \frac1p)},
\end{align*}
which gives \eqref{eq:uupq-3} as desired.
\end{proof}

\begin{lemma}\label{lem:ssth}
Let $k \ge 1$, $\sigma_j$ be weights on $\Rd$, and $q_j, \theta_j \in(0, \infty)$ with $\prod_{j=1}^k \sigma_j^{\theta_j} =1$. If 
\begin{align*}
\sup_{Q \in \Q} \prod_{j=1}^k 
\big(|Q|^{-1} \|\sigma_j \mathbf{1}_Q\|_{L^{1, q_j}} \big)^{\theta_j} < \infty,
\end{align*} 
then $\sigma_j \in A_{\infty}$ for each $j=1, \ldots, k$, hence, $\sup_{Q \in \Q} \prod_{j=1}^k \langle \sigma_j \rangle_Q^{\theta_j} < \infty$. 
\end{lemma}

\begin{proof}
Let $\beta_i(Q) := |Q|^{-1} \|\sigma_j \mathbf{1}_Q\|_{L^{1, q_j}}$. Then 
\begin{align*}
|\{x \in Q: \sigma_j(x) > \lambda \beta_j(Q)\}|
\lesssim \lambda^{-1} |Q|,
\end{align*}
which gives
\begin{align*}
& \langle \log \sigma_j \rangle_Q - \log \beta_j(Q)
\le \fint_Q \bigg( \log \frac{\sigma_j}{\beta_j(Q)} \bigg)_+ \, dx
\\
& = |Q|^{-1} \int_0^{\infty} |\{x \in Q: \sigma_j(x) > e^t \beta_j(Q)\}| \, dt
\lesssim \int_0^{\infty} e^{-t} \, dt 
= 1.
\end{align*}
Thus, $\beta_j(Q) \gtrsim \exp(\langle \log \sigma_j \rangle_Q)$, which along with $\prod_{j=1}^k \sigma_j^{\theta_j} =1$ implies
\begin{align*}
\big[\beta_j(Q) \exp(- \langle \log \sigma_j \rangle_Q) \big]^{\theta_j}
= \beta_j(Q)^{\theta_j} \prod_{i \ne j} [ \exp (\langle \log \sigma_j \rangle_Q )]^{\theta_i}
\lesssim \prod_{j=1}^k \beta_j(Q)^{\theta_j} 
< \infty.
\end{align*}
Consequently, 
\begin{align}\label{eq:betj}
\beta_j(Q) \simeq \exp(\langle \log \sigma_j \rangle_Q)
\end{align}
and 
\begin{align*}
\fint_Q \big( \log \sigma_j - \langle \log \sigma_j \rangle_Q \big)_+ \, dx
& = |Q|^{-1} \int_0^{\infty} |\{x \in Q: ( \log \sigma_j - \langle \log \sigma_j \rangle_Q \big)_+ > t\}| \, dt
\\
& \le |Q|^{-1} \int_0^{\infty} |\{x \in Q: \sigma_j > \exp(t + \langle \log \sigma_j \rangle_Q)\}| \, dt
\\
& \lesssim |Q|^{-1} \int_0^{\infty} \exp(-t - \langle \log \sigma_j \rangle_Q) 
\|\sigma_j \mathbf{1}_Q\|_{L^{1, q_j}} \, dt
\\
& \simeq \int_0^{\infty} e^{-t} \, dt 
= 1.
\end{align*}
Since $\fint_Q \big( \log \sigma_j - \langle \log \sigma_j \rangle_Q \big) \, dx = 0$, we obtain
\begin{align*}
\fint_Q \big| \log \sigma_j - \langle \log \sigma_j \rangle_Q \big| \, dx 
\lesssim 1.
\end{align*}
That is, $\log \sigma_j \in \BMO$. Pick $0 < \varepsilon_j < \min\{1, q_j, \lambda_0\}$, where $\lambda_0$ is given in \eqref{eq:JN}. Then,
\begin{align*}
\sigma_j^{-\varepsilon_j}(Q)
= \exp(- \varepsilon_j \langle \log \sigma_j \rangle_Q) 
\int_Q \exp [- \varepsilon_j (\log \sigma_j - \langle \log \sigma_j \rangle_Q) ] dx 
\lesssim \exp(- \varepsilon_j \langle \log \sigma_j \rangle_Q) |Q|,
\end{align*}
which together with \eqref{eq:uupq-1} and \eqref{eq:betj} yields
\begin{align*}
& |Q|^{-1} \|\mathbf{1}_Q\|_{L^{1+\varepsilon_j, (q_j(1 + \varepsilon_j)')'}(\sigma_j^{-\varepsilon_j})}
\|\sigma_j^{\varepsilon_j} \mathbf{1}_Q\|_{L^{(1+\varepsilon_j)', q_j(1 + \varepsilon_j)'}(\sigma_j^{-\varepsilon_j})}
\\
& \simeq |Q|^{-1} \sigma_j^{-\varepsilon_j}(Q)^{\frac{1}{1+\varepsilon_j}} 
\|\sigma_j \mathbf{1}_Q\|_{L^{1, q_j}}^{\frac{1}{(1+\varepsilon_j)'}}
\\
& \lesssim \exp \bigg[\bigg(- \frac{\varepsilon_j}{1 + \varepsilon_j} 
+ \frac{1}{(1 + \varepsilon_j)'} \bigg)
\langle \log \sigma_j \rangle_Q \bigg]
 = 1.
\end{align*}
In view of \eqref{eq:Apq}, this implies $\sigma_j^{- \varepsilon_j} \in A_{1 + \varepsilon_j}$, which is equivalent to $\sigma_j \in A_{1 + 1/\varepsilon_j}$. 
\end{proof}

\subsection{Weighted Variable Lebesgue spaces}
Given a measurable function $p(\cdot): \Rd \to [1, \infty)$, set 
\begin{equation*}
p_{-} := \essinf_{x \in \Rd} p(x) 
\quad\text{ and } \quad 
p_{+} := \esssup_{x \in \Rd} p(x).
\end{equation*}
The {\tt variable Lebesgue space} $L^{p(\cdot)}(\Rd)$ consists of all functions in $\bM$ satisfying
\begin{align*}
\|f\|_{L^{p(\cdot)}}
:= \inf \bigg\{\lambda > 0 :
\int_{\Rd} \bigg(\frac{|f(x)|}{\lambda} \bigg)^{p(x)} \, dx \leq 1 \bigg\} < \infty.
\end{align*}
Given a weight $w$ on $\Rd$, the {\tt weighted variable Lebesgue space} $L^{p(\cdot)}_w$ is the collection of all functions in $\bM$ for which
\begin{align*}
\|f\|_{L^{p(\cdot)}_w} := \|f w\|_{L^{p(\cdot)}} < \infty. 
\end{align*}
The space $L^{p(\cdot)}(\Rd)$ is a Banach function space, but not generally rearrangement invariant,  
\begin{align*}
(L^{p(\cdot)}(\Rd))' = L^{p'(\cdot)}(\Rd) 
\quad \text{ and } \quad
(L^{p(\cdot)}_w)' = L^{p'(\cdot)}_{w^{-1}},
\end{align*}
where $\frac{1}{p(x)} + \frac{1}{p'(x)} = 1$.

Furthermore, we say that $p(\cdot)\in{\rm LH}$ if there are constants $C_0,C_{\infty}\in(0,\infty)$ and $p_{\infty}\in[0,\infty)$ such that
\begin{align*}
|p(x)-p(y)| 
& \leq \frac{C_0}{-\log(|x-y|)}, \quad |x-y|<1/2, 
\\
|p(x) - p_{\infty}| 
& \leq \frac{C_{\infty}}{\log(e+|x|)}, \quad\quad x \in \Rn.
\end{align*}
For further details on these spaces, the reader is referred to \cite{CF}. The class of $A_{p(\cdot)}$ weights consists of all weights on $\Rd$ satisfying 
\begin{equation*}
[w]_{A_{p(\cdot)}} 
:= \sup_{Q \in \Q} |Q|^{-1}\,
\|w \mathbf{1}_Q\|_{L^{p(\cdot)}}
\|w^{-1} \mathbf{1}_Q\|_{L^{p'(\cdot)}}
< \infty. 
\end{equation*}
It was proved in \cite{CDH, CFN} that whenever $p(\cdot) \in \mathrm{LH}$ with $1 < p_- \leq p_+ < \infty$, 
\begin{align*}
w \in A_{p(\cdot)} 
\quad \Longleftrightarrow \quad \|M\|_{L^{p(\cdot)}_w \to L^{p(\cdot)}_w} < \infty 
\quad \Longleftrightarrow \quad \|M\|_{L^{p'(\cdot)}_{w^{-1}} \to L^{p'(\cdot)}_{w^{-1}}} < \infty.
\end{align*}
In particular, by \cite[Theorem~2.4]{CW} it is easy to see that 
\begin{align*}
w \equiv 1 \in A_{p(\cdot)},\quad \forall\,p(\cdot) \in{\rm LH} \text{ with } 1<p_{-}\leq p_{+}<\infty. 
\end{align*}

\begin{definition}
Let $\vec{r} \in [1, \infty)^m$, $s \in (1, \infty]$, $r_j < (p_j)_- \le (p_j)_+ < \infty$, and $\frac{1}{p(\cdot)} = \sum_{j=1}^m \frac{1}{p_j(\cdot)}$ with $p_+ < s$. Given weights $w_1, \ldots, w_m$ on $\Rd$, we say that $\vec{w} = (w_1, \ldots, w_m) \in A_{\vec{p}(\cdot), (\vec{r}, s)}$ if 
\begin{align*}
[\vec{w}]_{A_{\vec{p}(\cdot), (\vec{r}, s)}}
:= \sup_{Q \in \Q} |Q|^{-(\frac1r - \frac1s)} 
\|w \mathbf{1}_Q\|_{L^{\frac{1}{\frac{1}{p(\cdot)} - \frac1s}}}
\prod_{j=1}^m \|w_j^{-1} \mathbf{1}_Q\|_{L^{\frac{1}{\frac{1}{r_j} - \frac{1}{p_j(\cdot)}}}} < \infty.
\end{align*}
\end{definition}

The characterization of $A_{\vec{p}(\cdot), (\vec{r}, s)}$ is included in Theorem \ref{thm:example}, whose proof will be given in Section \ref{sec:exa}.

\subsection{Weighted Morrey spaces}
Given $1 \le p \le q \le \infty$, the {\tt Morrey space} $\L^{p, q}(\Rd)$ is defined as the collection of all functions $f \in \bM$ for which 
\begin{align*}
\|f\|_{\L^{p, q}}
:= \sup_{Q \in \Q} |Q|^{\frac1q} \langle f \rangle_{Q, p} < \infty.
\end{align*}
Then $\L^{p, q}(\Rd)$ is a Banach function space, $\L^{p, p}(\Rd) = L^p(\Rd)$, $\L^{p, \infty}(\Rd) = L^{\infty}(\Rd)$ by the Lebesgue differentiation theorem, and $L^q(\Rd) \subset \L^{p, q}(\Rd)$ by H\"{o}lder's inequality.

The {\tt block space} $\B^{q, p}(\Rd)$ is defined as the collection of measurable functions $g$ for which there exists a sequence $\{\lambda_n\}_{n \ge 1} \in \ell^1$ and a sequence of functions $\{a_n\}_{n \ge 1}$ with the property that for each $n \ge 1$ there is a cube $Q_n$ such that 
\begin{align*}
\supp a_n \subset Q_n, \quad
\|a_n\|_{L^q} \le |Q_n|^{\frac1q - \frac1p},
\quad \text{ and } \quad 
g = \sum_{n=1}^{\infty} \lambda_n a_n.
\end{align*}
The norm $\|g\|_{\B^{q, p}}$ is defined as the smallest possible value of $\|\{\lambda_n\}_{n \ge 1}\|_{\ell^1}$ for which such a representation exists. The space $\B^{q, p}(\Rd)$ is a Banach function space. 

Given a weight $w$ on $\Rd$, we then define weighted variants of Morrey and block spaces by
\begin{align*}
\|f\|_{\L^{p, q}_w} := \|fw\|_{\L^{p, q}}
\quad \text{ and } \quad 
\|g\|_{\B^{q, p}_w} := \|fw\|_{\B^{q, p}},
\end{align*}
which satisfy 
\begin{align*}
(\L^{p, q}_w)' = \B^{p', q'}_{w^{-1}}, \quad 1 \le p \le q \le \infty.
\end{align*}

\section{Extrapolation of boundedness}\label{sec:bdd}
The goal of this section is to show Theorem \ref{thm:BFS}. The core of its proof is the following result. 

\begin{theorem}\label{thm:fac}
Let $(\vec{r}, r_{m+1}) \in [1, \infty)^{m+1}$ be such that $\sum_{j=1}^{m+1} \frac{1}{r_j} > 1$. Let $\Y_1, \ldots, \Y_{m+1}$ be Banach function spaces such that 
\begin{align*}
\prod_{j=1}^{m+1} \Y_j = L^1 
\quad \text{ and } \quad
\mathscr{M} := \|\M_{\vec{r}, r_{m+1}}\|_{\Y_1 \times \cdots \times \Y_{m+1} \to L^1} < \infty.
\end{align*}
Then every $h_j \in \Y_j$ and $p_j \in [r_j, \infty)$ with $\sum_{j=1}^{m+1} \frac{1}{p_j} = 1$, there exist weights $v_1, \ldots, v_{m+1}$ such that 
\begin{align*}
\prod_{j=1}^{m+1} v_j = 1, \quad 
\sup_{Q \in \Q} \prod_{j=1}^{m+1} \langle v_j^{-1} \rangle_{\frac{1}{1/r_j - 1/p_j}, Q} \lesssim \mathscr{M},
\quad
\prod_{j=1}^{m+1} \|h_j v_j\|_{L^{p_j}}
\le 2^{2(m+1)} \prod_{j=1}^{m+1} \|h_j\|_{\Y_j}.
\end{align*}
\end{theorem}

\begin{proof}[{\bf Proof of Theorem~\ref{thm:BFS}}]
Set $\Y_{m+1} = \Y'$ and $\Y_j = \X_j^{\frac1m}$, $j = 1, \ldots, m$. Then $\prod_{j=1}^{m+1} \Y_j = L^1$. Moreover,
\begin{align*}
\sum_{j=1}^m \frac{1}{m r_j} + \frac{1}{(ms)'} 
= 1 + \frac1m \bigg(\frac1r - \frac1s\bigg) > 1
\quad \text{ and } \quad
\sum_{j=1}^m \frac{1}{m p_j} + \frac{1}{(mp)'} = 1.
\end{align*}
Fix $0 \le h \in \Y'$. Let $h_{m+1} = h$ and $h_j = |f_j|^{\frac1m}$, $j=1, \ldots, m$. Recall that 
\begin{align*}
\mathcal{N} 
:= \|\M_{m \vec{r}, (ms)'}\|_{\X_1^{\frac1m} \times \cdots \times \X_m^{\frac1m}\times \Y' \to L^1} < \infty.
\end{align*}
Then by Theorem \ref{thm:fac}, one can find weights $v_1, \ldots, v_{m+1}$ such that $\prod_{j=1}^{m+1} v_j = 1$, 
\begin{align}
\label{eq:vmp-1}
& \sup_{Q \in \Q} \prod_{j=1}^m \langle v_j^{-1} \rangle_{\frac{m}{1/r_j - 1/p_j}, Q} \, 
\langle v_{m+1}^{-1} \rangle_{\frac{1}{1/(ms)' - 1/(mp)'}, Q} 
\lesssim \mathcal{N},
\\
\label{eq:vmp-2}
& \prod_{j=1}^m \||f_j|^{\frac1m} v_j\|_{L^{m p_i}}
\|h v_{m+1}\|_{L^{(m p)'}}
\le 2^{m+1} \prod_{j=1}^m \| |f_j|^{\frac1m}\|_{\X_j^{\frac1m}} \, \|h\|_{\Y'}.
\end{align}
Now take $w = \prod_{j=1}^m w_j$ and $w_j = v_j^m$ for $j =1, \ldots, m$. Then \eqref{eq:vmp-1} and \eqref{eq:vmp-2} can be rewritten as
\begin{align}\label{eq:prsn}
[\vec{w}]_{A_{\vec{p}, (\vec{r}, s)}}^{\frac1m} \lesssim \mathcal{N}
\quad \text{and} \quad
\prod_{j=1}^m \|f_j\|_{L^{p_i}_{w_j}}^{\frac1m}
\|h w^{-\frac1m}\|_{L^{(m p)'}}
\lesssim \prod_{j=1}^m \|f_j\|_{\X_j}^{\frac1m} \, \|h\|_{\Y'}.
\end{align}
Hence, by H\"older's inequality, \eqref{eq:prsn}, and the hypothesis, 
\begin{align*}
\langle |f|^{\frac1m}, h \rangle
& \le \||f|^{\frac1m} w^{\frac1m}\|_{L^{mp}} \|h w^{-\frac1m} \|_{L^{(mp)'}} 
\\
& = \|f\|_{L^p_w}^{\frac1m} \|h w^{-\frac1m} \|_{L^{(mp)'}}
\\
& \le \Phi(C_0 \mathcal{N}^m)^{\frac1m} 
\prod_{j=1}^m \|f_j\|_{L^{p_i}_{w_j}}^{\frac1m} 
\|h w^{-\frac1m} \|_{L^{(mp)'}} 
\\
& \le 2^{2(m+1)} \, \Phi(C_0 \mathcal{N}^m)^{\frac1m} 
\prod_{j=1}^m \|f_j\|_{\X_j}^{\frac1m} \, \|h\|_{\Y'},
\end{align*}
where $C_0 = C_0(d, m, \vec{p}, \vec{r}, s)$. Note that 
\begin{align*}
\Y = \Y'' 
\quad \text{ and } \quad
\|f\|_{\X}^{\frac1m} 
= \| |f|^{\frac1m} \|_{\Y}
= \sup_{\|h\|_{\Y'} \le 1} \langle |f|^{\frac1m}, |h| \rangle,
\end{align*}
which implies
\begin{align*}
\|f\|_{\X}  
\le 2^{2(m+1)} \, \Phi(C_0 \mathcal{N}^m)
\prod_{j=1}^m \|f_j\|_{\X_j}.
\end{align*}
This shows Theorem~\ref{thm:BFS}.
\end{proof}

\subsection{Multilinear Rubio de Francia algorithm}
We present a multilinear generalization of the Rubio de Francia algorithm on Banach function spaces, which is the main ingredient in proof of Theorem \ref{thm:fac}. Throughout this section, let $\mathscr{D} := \bigcup_{k=1}^{3^d} \mathscr{D}^k$, where $\mathscr{D}^k$ is the dyadic lattice obtained from Lemma \ref{lem:dyadic}.

\begin{lemma}\label{lem:maj}
Let $1 \le r_j \le p_j < \infty$ be such that $\sum_{j=1}^{m+1} \frac{1}{p_j} = 1 < \sum_{j=1}^{m+1} \frac{1}{r_j}$. Let $\Y_1, \ldots, \Y_{m+1}$ be Banach function spaces such that 
\begin{align*}
\prod_{j=1}^{m+1} \Y_j = L^1 
\quad \text{ and } \quad
\mathscr{M} := \|\M_{\vec{r}, r_{m+1}}\|_{\Y_1 \times \cdots \times \Y_{m+1} \to L^1} < \infty.
\end{align*}
Then for every $h_j \in \Y_j$ with $h_j > 0$ a.e., there exists $H_j \in \Y_j$ such that $|h_j| \le H_j$, $\|H_j\|_{\Y_j} \le 2 \|h_j\|_{\Y_j}$, and 
\begin{align*}
\prod_{j=1}^{m+1} \langle H_j \rangle_{r_j, Q}^{1 - \frac{r_j}{p_j}}
\lesssim \mathscr{M} \, \prod_{j: p_j \ne \infty} 
\Big(\essinf_Q H H_j^{-r_j} \Big)^{\frac{1}{p_j}}, \quad \forall Q \in \mathscr{D},
\end{align*}
where $H = \prod_{j=1}^{m+1} H_j$.
\end{lemma}

\begin{proof}[{\bf Proof of  Theorem \ref{thm:fac}}]
First, we consider the case $h_j > 0$ a.e., $j=1, \ldots, m+1$. Let $H_j \in \Y_j$ be functions obtained from Lemma \ref{lem:maj}. Pick $v_j := H^{\frac{1}{p_j}} H_j^{-1}$ so that $\prod_{j=1}^{m+1} v_j = 1$. Moreover, we have $|h_j| v_j \le H_j v_j = H^{\frac{1}{p_j}}$, which along with H\"{o}lder's inequality yields 
\begin{align*}
\prod_{j=1}^{m+1} \|h_j v_j\|_{L^{p_j}} 
\le \|H\|_{L^1}
\le \prod_{j=1}^{m+1} \|H_j\|_{\Y_j} 
\le 2^{m+1} \prod_{j=1}^{m+1} \|h_j\|_{\Y_j}.
\end{align*}
Fix a cube $Q \subset \Rd$. Then there exists some $P_Q \in \mathscr{D}$ such that $Q \subset P_Q$ and $\ell(P_Q) \le 3 \ell(Q)$. This and Lemma \ref{lem:maj} imply 
\begin{align*}
& \prod_{j=1}^{m+1} \langle v_j^{-1} \rangle_{\frac{1}{1/r_j - 1/p_j}, Q} 
\lesssim \prod_{j=1}^{m+1} \langle v_j^{-1} \rangle_{\frac{1}{1/r_j - 1/p_j}, P_Q} 
\\
& = \prod_{j=1}^{m+1} \bigg(\fint_{P_Q}  
\Big[(H H_j^{-r_j})^{- \frac{1}{p_j}} H_j^{r_j(\frac{1}{r_j} - \frac{1}{p_j})} \Big]^{\frac{1}{\frac{1}{r_j} - \frac{1}{p_j}}} \bigg)^{\frac{1}{r_j} - \frac{1}{p_j}}
\\
& \le \prod_{j=1}^{m+1} \bigg(\fint_{P_Q}  
H_j^{r_j} \bigg)^{\frac{1}{r_j} - \frac{1}{p_j}}
\prod_{j: p_j \ne \infty} 
\Big(\esssup_{P_Q} H H_j^{-r_j} \Big)^{-\frac{1}{p_j}}
\lesssim \mathscr{M}.
\end{align*}

Next, let us treat the general case $h_j \in \Y_j$. By the saturation property of $\Y_j$, there exists some $h'_j \in \Y_j$ so that $h'_j > 0$ a.e. and $\|h'_j\|_{\Y_j} \le 1$. In the scenario $\|h_j\|_{\Y_j} > 0$, choose $\widetilde{h}_j := |h_j| + \|h_j\|_{\Y_j} h'_j$ so that $\widetilde{h}_j > 0$ a.e. and $\|\widetilde{h}_j\|_{\Y_j} \le 2 \|h_j\|_{\Y_j}$. Applying the preceding result to $\widetilde{h}_j$, we obtain the desired weights $v_1, \ldots, v_{m+1}$ satisfying 
\begin{align*}
\prod_{j=1}^{m+1} \|h_j v_j\|_{L^{p_j}}
\le \prod_{j=1}^{m+1} \|\widetilde{h}_j v_j\|_{L^{p_j}}
\le 2^{m+1} \prod_{j=1}^{m+1} \|\widetilde{h}_j\|_{\Y_j}
\le 2^{2(m+1)} \prod_{j=1}^{m+1} \|h_j\|_{\Y_j}.
\end{align*}
In the scenario $\|h_j\|_{\Y_j} = 0$, we have $h_j = 0$ a.e. and $\widetilde{h}_j := |h_j| + h'_j > 0$ a.e. We then use  the preceding result to $\widetilde{h}_j$ to achieve weights $v_1, \ldots, v_{m+1}$ as desired. Additionally, the fact $h_j = 0$ a.e. gives
\begin{align*}
\prod_{j=1}^{m+1} \|h_j v_j\|_{L^{p_j}}
= 0 = \prod_{j=1}^{m+1} \|h_j\|_{\Y_j}.
\end{align*}
This completes the proof of Theorem \ref{thm:fac}.
\end{proof}

\subsection{Stopping time argument}
To prove Lemma \ref{lem:maj}, we establish a useful estimate as follows. Fix $h_j \in \Y_j$ with $h_j > 0$ a.e. Given $f_j \ge h_j$ and $g_j \ge h_j$ for each $j=1, \ldots, m+1$, define
\begin{align*}
\mathscr{P}_Q (\bm{f}) 
: = \prod_{j=1}^{m+1} \langle f_j \rangle_{r_j, Q}^{1 - \frac{r_j}{p_j}}
\quad \, \text{ and } \, \quad
\mathscr{Q}_j (\bm{f}, \bm{g}) 
:= f_j^{1-r_j} \prod_{i \ne j} (g_i - h_i).
\end{align*}

\begin{lemma}\label{lem:suc}
Let $r_j$ and $\Y_j$ satisfy the hypotheses of Lemma $\ref{lem:maj}$. Let $h_j \in \Y_j$ with $h_j > 0$ a.e., $j=1, \ldots, m+1$. Then for all $f_j \ge h_j$ with $\|f_j\|_{\Y_j} \le 2 \|h_j\|_{\Y_j}$, there exist $g_j \ge h_j$ with $\|g_j\|_{\Y_j} \le 2 \|h_j\|_{\Y_j}$ such that
\begin{align*}
\mathscr{P}_Q (\bm{f}) 
\lesssim \mathscr{M} \prod_{j: p_j \ne \infty}
\big(\essinf_Q \mathscr{Q}_j (\bm{f}, \bm{g}) \big)^{\frac{1}{p_j}}, 
\quad \forall Q \in \mathscr{D}',
\end{align*}
for every finite subset $\mathscr{D}' \subset \mathscr{D}$, where the implicit constant is independent of $\mathscr{D}'$.
\end{lemma}

\begin{proof}
Let $\Stop_k^{(0)}$ be the maximal cubes in $\mathscr{D}' \cap \mathscr{D}^k$, $k=1, \ldots, 3^d$. Given $A_0 > 1$ and $Q \in \Stop_k^{(0)}$, we subdivide dyadically $Q$ and stop the first time that one of the following $(m+1)$ conditions occurs
\begin{align}\label{eq:stopcon}
\langle f_j \rangle_{Q', r_j} > A_0 \langle f_j \rangle_{Q, r_j}, \quad j=1, \ldots, m+1. 
\end{align}
Let $\Stop(Q)$ be the stopping time family. Since 
\begin{align*}
\sum_{\substack{Q' \in \Stop(Q) \\ \langle f_j \rangle_{Q', r_j} > A_0 \langle f_j \rangle_{Q, r_j}}} |Q'|
\le (A_0 \langle f_j \rangle_{Q, r_j})^{-r_j}
\sum_{Q' \in \Stop(Q)} \int_{Q'} |f_j|^{r_j}
\le A_0^{-1} |Q|,
\end{align*}
one can choose $A_0 = A_0(\gamma)$ sufficiently large so that 
\begin{align*}
\sum_{Q' \in \Stop(Q)} |Q'| \le (1 - \gamma) |Q|,
\end{align*} 
Then define 
\begin{align*}
\S_k := \bigcup_{i=0}^{\infty} \Stop_k^{(i)} 
\quad \text{ and } \quad
\Stop_k^{(i+1)} := \bigcup_{Q' \in \Stop_k^{(i)}} \Stop(Q'), \quad i \ge 0.
\end{align*} 
We claim that
\begin{align}\label{eq:gas}
\text{the family $\S_k$ is $\gamma$-sparse}.
\end{align}
Indeed, given $Q \in \S_k$, if we denote $E_Q := Q \setminus \bigcup_{Q' \in \Stop} Q'$, then the construction of $\Stop(Q)$ gives $|E_Q| = |Q \setminus \Omega_Q| \ge \gamma |Q|$. Let $Q_1, Q_2 \in \S_k$ with $Q_1 \ne Q_2$. Then $Q_1 \in \Stop^{(i_1)}_k$ and $Q_2 \in \Stop^{(i_2)}_k$ for some $i_1, i_2 \ge 0$. If $i_1 = i_2$, the maximality of $\mathrm{Stop}^{(i_1)}_k$ implies $Q_1 \cap Q_2 = \varnothing$, hence, $E_{Q_1} \cap E_{Q_2} = \varnothing$. If $i_1 \ne i_2$, we may assume $i_1 < i_2$. By construction, there exists some $Q'_1 \in \Stop(Q_1)$ so that $Q_2 \in Q'_1$, which gives $E_{Q_1} \cap Q_2 = \varnothing$, hence, $E_{Q_1} \cap E_{Q_2} = \varnothing$. This means $\{E_Q\}_{Q \in \S}$ is pairwise disjoint.

Let $Q \in \mathscr{D}' \cap \mathscr{D}^k$. If $Q \in \S_k$, then let $R_Q = Q$. If $Q \notin \S_k$, then $Q \subset Q^*$ for some $Q^* \in \Stop_k^{(0)}$. By the finiteness of $\mathscr{D}'$ and the construction of $\S_k$, one can find a minimal cube $R_Q \in \S_k$ such that $Q \subset R_Q$. Thus, 
\begin{align}\label{eq:rjrq}
\langle f_j \rangle_{r_j, Q}
\le A_0 \langle f_j \rangle_{r_j, R_Q}, \quad\forall j \in \{1, \ldots, m+1\},
\end{align}
which implies
\begin{align}\label{eq:fess}
\mathscr{P}_Q (\bm{f}) 
\lesssim \prod_{j=1}^{m+1} \langle f_j \rangle_{r_j, R_Q}^{1- \frac{r_j}{p_j}}
\lesssim \prod_{j: p_j \ne \infty} \bigg[\essinf_Q \sum_{k=1}^{3^d} \sum_{R \in \S_k} \frac{|E_R|}{|R|}
\frac{\prod_{i=1}^{m+1} \langle f_i \rangle_{r_i, R}}{\langle f_j \rangle_{r_j, R}^{r_j}} 
\mathbf{1}_R \bigg]^{\frac{1}{p_j}}.
\end{align}
If we define
\begin{align*}
\A^j_{\S}(\vec{f}, f_{m+1})
: = f_j^{r_j - 1} \sum_{Q \in \S} \frac{|E_Q|}{|Q|}
\frac{\prod_{i=1}^{m+1} \langle f_i \rangle_{r_i, Q}}{\langle f_j \rangle_{r_j, Q}^{r_j}} 
\mathbf{1}_Q, \quad j = 1, \ldots, m+1,
\end{align*}
then by H\"{o}lder's inequality,
\begin{align*}
\langle f_j^{r_j -1} h \rangle_Q 
\le \langle f_j \rangle_{r_j, Q}^{r_j - 1}  \langle h \rangle_{Q, r_j},  
\end{align*}
which further leads to
\begin{align*}
& \langle \A^j_{\S}(\vec{f}, f_{m+1}), |h| \rangle
\le \sum_{Q \in \S} |E_Q| 
\prod_{i \ne j} \langle f_i \rangle_{Q, r_i} \langle h \rangle_{Q, r_j}
\\
& \le \int_{\Rd} \M_{\vec{r}, r_{m+1}}(f_1, \ldots, f_{j-1}, h, f_{j+1}, \ldots, f_{m+1})
\\
& \le \mathscr{M} \prod_{i \ne j} \|f_i\|_{\Y_i} \, \|h\|_{\Y_j}
\le 2^m \mathscr{M} \prod_{i \ne j} \|h_i\|_{\Y_i} \, \|h\|_{\Y_j}.
\end{align*}
Accordingly, 
\begin{align*}
\|\A^j_{\S}(\vec{f}, f_{m+1})\|_{\Y'_j}
\le 2^m \mathscr{M} \prod_{i \ne j} \|h_i\|_{\Y_i}.
\end{align*}
By definition and Lemma~\ref{lem:EE}, there exist nonnegative functions $g_i^j \in \Y_i$ such that 
\begin{align}\label{eq:sffm}
\A^j_{\S}(\vec{f}, f_{m+1}) \le \mathscr{M} \prod_{i \ne j} g_i^j
\quad \text{ and } \quad 
\prod_{i \ne j} \|g_i^j\|_{\Y_i} \le 3^m \prod_{i \ne j} \|h_i\|_{\Y_i}.
\end{align}
If we take 
\begin{align*}
h_i^j := g_i^j \frac{\|h_i\|_{\Y_i}}{\|g_i^j\|_{\Y_i}} 
\bigg(\prod_{i \ne j} \frac{\|g_i^j\|_{\Y_i}}{\|h_i\|_{\Y_i}} \bigg)^{\frac1m}
\quad \text{ and } \quad 
g_i := h_i + \frac{1}{3m} \sum_{j \ne i} h_i^j,
\end{align*}
then \eqref{eq:sffm} gives $g_i \ge h_i$, $\|g_i\|_{\Y_i} \le 2 \|h_i\|_{\Y_i}$, and 
\begin{align}\label{eq:smfm}
\A^j_{\S}(\vec{f}, f_{m+1}) 
\le \mathscr{M} \prod_{i \ne j} g_i^j
= \mathscr{M} \prod_{i \ne j} h_i^j
\lesssim \mathscr{M} \prod_{i \ne j} (g_i - h_i)
= \mathscr{M} \mathscr{Q}_j(\bm{f}, \bm{g}) f_j^{r_j - 1}.
\end{align}
Thus, the desired estimate follows from \eqref{eq:fess} and \eqref{eq:smfm}.
\end{proof}

\begin{lemma}\label{lem:EE}
Under the hypotheses of Lemma~$\ref{lem:maj}$, one has
\begin{align*}
\prod_{i \ne j} \Y_i = \Y'_j
\quad\text{ and } \quad 
\langle f \rangle_{r_j, Q} \le C_{j, Q} \|f\|_{\Y_j}, \quad j = 1, \ldots, m+1.
\end{align*}
\end{lemma}

\begin{proof}
Fix $j \in \{1, \ldots, m+1\}$. Note that $\Y_j, \Y'_j$, and $\prod_{i \ne j} \Y_i$ are Banach function spaces satisfying 
\begin{align*}
\Y_j \Big( \prod_{i \ne j} \Y_i \Big) 
= L^1 = \Y_j \Y'_j. 
\end{align*}
Then it follows from \cite[Corollary~2.6]{Schep} that $\prod_{i \ne j} \Y_i = \Y'_j$ isometrically. 

Fix $f \in \Y_j$. Choose $f_j = f$ and $0 < f_i \in \Y_i$ with $\|f_i\|_{\Y_i} = 1$ for every $i \ne j$. By definition and the boundedness of $\M_{\vec{r}, r_{m+1}}$, we obtain
\begin{align*}
|Q| \langle f \rangle_{r_j, Q} \prod_{i \ne j} \langle f_i \rangle_{r_i, Q}
\le \int_Q \M_{\vec{r}, r_{m+1}}(\vec{f}, f_{m+1})(x) \, dx
\le \mathscr{M} \prod_{i=1}^{m+1} \|f_i\|_{\Y_i}
= \mathscr{M} \|f\|_{\Y_j}.
\end{align*}
This proves $\langle f \rangle_{r_j, Q} \le C_{j, Q} \|f\|_{\Y_j}$.
\end{proof}

\subsection{Proof of Lemma \ref{lem:maj}}
By the fact $\prod_{j=1}^{m+1} \Y_j = L^1$, we have $\phi_j := \prod_{i \ne j} h_i \in \Y'_j$. Define a probability measure $\mu$ by 
\begin{align*}
d\mu(x) := A_0^{-1} \prod_{j=1}^{m+1} \phi_j(x) \, dx 
\quad \text{ and } \quad
0 < A_0 := \int_{\Rd} \prod_{j=1}^{m+1} \phi_j \, dy
= \int_{\Rd} h_j \phi_j \, dx < \infty.
\end{align*}
Set 
\begin{align*}
\Xi := \prod_{j=1}^{m+1} \Xi_j
\quad \text{ and } \quad
\Xi_j := \big\{a \in L^1(\mu) : a \ge 0, \ \|e^a h_j\|_{\Y_j} \le 2 \|h_j\|_{\Y_j} \big\}.
\end{align*}
Observe that $\Xi_j$ is convex since for every $a, b \in \Xi_j$ and $\alpha \in (0, 1)$,
\begin{align*}
\|e^{(1-\alpha) a + \alpha b} h_j\|_{\Y_j}
\le \|(1-\alpha) e^a h_j + \alpha e^b h_j\|_{\Y_j}
\le (1-\alpha) \|e^a h_j\|_{\Y_j} + \alpha \|e^b h_j\|_{\Y_j}
\le 2 \|h_j\|_{\Y_j}.
\end{align*}
Let $\{a_k\}_{k=1}^{\infty} \subset \Xi_j$ be such that $\lim_{k \to \infty} \|a_k - a\|_{L^1(\mu)} = 0$. Then $0 \le a \in L^1(\mu)$ and there exists a subsequence $\{a_{k_{\ell}}\}_{\ell \ge 1} \subset \{a_k\}_{k \ge 1}$ such that $a_{k_{\ell}} \to a$ $\mu$-a.e., which, along with the fact $d\mu \ll dx \ll d\mu$, yields $e^{a_{k_{\ell}}} h_j \to e^a h_j$ a.e.. Setting $H_N := \inf_{\ell \ge N} e^{a_{k_{\ell}}} h_j$ for each $N \ge 1$, we see that $H_N $ is increasing, $\lim_{N \to \infty} H_N = e^a h_j$ a.e., and $\|H_N\|_{\Y_j} \le \|e^{a_{k_N}} h_j\|_{\Y_j} \le 2 \|h_j\|_{\Y_j}$. Then by the Fatou property \eqref{bfs-5}, $e^a h_j \in \Y_j$ and $\|e^a h_j\|_{\Y_j} \le 2 \|h_j\|_{\Y_j}$. This shows $a \in \Xi_j$. Hence, $\Xi_j$ is closed.

Moreover, for all $a \in \Xi_j$,
\begin{align*}
\|e^a\|_{L^1(\mu)}
= A_0^{-1} \|e^a h_j \phi_j\|_{L^1(\mu)}
\lesssim \|e^a h_j\|_{\Y_j} \|\phi_j\|_{\Y'_j}
\lesssim \|h_j\|_{\Y_j} \|\phi_j\|_{\Y'_j},
\end{align*}
which gives for all $A>1$
\begin{align*}
\sup_{a \in \Xi_j} \|a\|_{L^1(\mu)} \lesssim 1
\quad \text{ and } \quad
\int_{\{a>A\}} a \, d\mu 
\le A e^{-A} \|e^a\|_{L^1(\mu)}
\lesssim A e^{-A} \to 0,
\end{align*}
as $A \to \infty$. Then \cite[Theorem 3.2.1]{DP} implies $\Xi_j$ is weakly compact. Hence, 
\begin{align}\label{eq:Ewcpt}
\text{$\Xi$ is convex and weakly compact, and } 0 \in \Xi.
\end{align}

Given a finite subset $\mathscr{D}' \subset \mathscr{D}$, let $\Pi_{\mathscr{D}'}$ be the collection of all pairs $(\bm{a}, \bm{b}) \in \Xi \times \Xi$ such that Lemma \ref{lem:suc} holds for $f_j = e^{a_j} h_j$ and $g_j = e^{b_j} h_j$, $j=1, \ldots, m+1$. We claim that
\begin{align}\label{eq:EDcon}
\text{$\Pi_{\mathscr{D}'}$ is convex and weakly closed.}
\end{align}
Assuming \eqref{eq:EDcon} holds momentarily, let us conclude the proof. In view of Lemma \ref{lem:suc}, for every $\bm{a} \in \Xi$, there exists some $\bm{a}' \in \Xi$ so that $(\bm{a}, \bm{a}') \in \Pi_{\mathscr{D}'}$. Thus, there exists a sequence $\{\bm{a}_k\}_{k \ge 1} \subset \Xi$ satisfying $(\bm{a}_k, \bm{a}_{k+1}) \in \Pi_{\mathscr{D}'}$. By the convexity of both $\Xi$ and $\Pi_{\mathscr{D}'}$, there holds
\begin{align}\label{eq:bck}
(\bm{b}_k, \bm{c}_k) 
:= \bigg(\frac1k \sum_{i=1}^k \bm{a}_i, \, \frac1k \sum_{i=1}^k \bm{a}_{i+1}  \bigg)
\in \Pi_{\mathscr{D}'} \subset \Xi \times \Xi.
\end{align}
Since every weakly compact subset of a Banach space is bounded (see \cite[p. 367]{DP}), we have
\begin{align}\label{eq:bkbc}
\|\bm{b}_k - \bm{c}_k\|_{\Xi}
= k^{-1} \|\bm{a}_{k+1} - \bm{a}_1\|_{\Xi} \longrightarrow 0, \quad \text{as } k \to \infty.
\end{align}
In view of \eqref{eq:Ewcpt}, $\Xi \times \Xi$ is weakly compact, which together with \eqref{eq:bck} yields a subsequence $\{(\bm{b}_{k_{\ell}}, \bm{c}_{k_{\ell}}) \}_{\ell \ge 1} \subset \{(\bm{b}_k, \bm{c}_k)\}_{k \ge 1}$ such that
\begin{align}\label{eq:wcg}
(\bm{b}_{k_{\ell}}, \bm{c}_{k_{\ell}}) \xrightarrow{\text{weakly}} (\bm{b}, \bm{c}) \in \Xi \times \Xi. 
 \end{align}
Consequently, for every $\mathscr{L} \in \Xi^*$, 
\begin{align*}
\mathscr{L}(\bm{b} - \bm{c}) 
= \lim_{\ell \to \infty} \mathscr{L}(\bm{b}_{k_{\ell}} - \bm{c}_{k_{\ell}}),
\end{align*}
and by \eqref{eq:bkbc},
\begin{align*}
\mathscr{L}(\bm{b}_{k_{\ell}} - \bm{c}_{k_{\ell}})
\le \|\mathscr{L}\|_{\Xi^*} \|\bm{b}_{k_{\ell}} - \bm{c}_{k_{\ell}}\|_{\Xi} 
\longrightarrow 0, \quad \text{as } \ell \to \infty.
\end{align*}
This means $\bm{b} = \bm{c} \in \Xi$, which along with \eqref{eq:EDcon}, \eqref{eq:bck}, and \eqref{eq:wcg} implies $(\bm{b}, \bm{b}) \in \Pi_{\mathscr{D}'}$. 

Since the diagonal mapping from $\Xi$ to $\Xi \times \Xi$ is weakly continuous, the assertion \eqref{eq:EDcon} gives
\begin{align}\label{eq:Xid-1}
\Xi_{\mathscr{D}'} 
:= \big\{\bm{a} \in \Xi: (\bm{a}, \bm{a}) \in \Pi_{\mathscr{D}'} \big\} \text{ is weakly closed and } 
\bm{b} \in \Xi_{\mathscr{D}'}.
\end{align}
Observe that for every $N \ge 1$ and finite subsets $\mathscr{D}_i \subset \mathscr{D}$, $i=1, \ldots, N$, 
\begin{align}\label{eq:Xid-2}
\bigcap_{i=1}^N \Xi_{\mathscr{D}_i} 
= \Xi_{\bigcup_{i=1}^N \mathscr{D}_i} 
\ne \varnothing.
\end{align}
As a consequence of \eqref{eq:Xid-1} and \eqref{eq:Xid-2}, there holds
\begin{align*}
\bigcap_{\substack{\mathscr{D}' \subset \mathscr{D} \\ \mathscr{D}': \text{ finite}}} \Xi_{\mathscr{D}'} 
\ne \varnothing,
\end{align*}
which gives some $\bm{a} = (a_1, \ldots, a_{m+1}) \in \Xi_{\mathscr{D}'}$ for every finite subset $\mathscr{D}' \subset  \mathscr{D}$. Then by definition of $\Xi_{\mathscr{D}'}$ and $\Pi_{\mathscr{D}'}$, we have $f_j = e^{a_j} h_j$ with $\|f_j\|_{\Y_j} \le 2 \|h_j\|_{\Y_j}$ and 
\begin{align*}
\mathscr{P}_Q (\bm{f}) 
& \lesssim \mathscr{M} \prod_{j: p_j \ne \infty}
\big(\essinf_Q \mathscr{Q}_j (\bm{f}, \bm{f}) \big)^{\frac{1}{p_j}}
\\
& \le \mathscr{M} \prod_{j: p_j \ne \infty}
\bigg(\essinf_Q f_j^{-r_j} \prod_{i=1}^{m+1} f_i \bigg)^{\frac{1}{p_j}}, 
\quad \forall Q \in \mathscr{D}',
\end{align*}
for every finite subset $\mathscr{D}' \subset  \mathscr{D}$, where $\bm{f} = (f_1, \ldots, f_{m+1})$. Taking $H_j := f_j = e^{a_j} h_j$, we complete the proof of Lemma \ref{lem:maj}.

It remains to demonstrate \eqref{eq:EDcon}. The convexity of $\bm{a} \mapsto \log \langle e^{a_j} h_j \rangle_{r_j, Q}$ implies 
\begin{align}\label{eq:pcx-1}
\bm{a} \mapsto \log \mathscr{P}_Q(\bm{f}) \, \, \text{ is convex}.
\end{align} 
In addition, note that
\begin{align*}
\frac{d^2}{dx^2} \log(e^x-1)
= - e^x (e^x-1)^{-2} < 0, \quad \forall x > 0,
\end{align*}
and 
\begin{align*}
\log \mathscr{Q}_j(\bm{f}, \bm{g})
& = \sum_{i \ne j} \big[\log h_j + \log(e^{b_j} - 1) \big] - (r_j - 1)(a_j + \log h_j),
\end{align*}
which indicates
\begin{align}\label{eq:pcx-2}
(\bm{a}, \bm{b}) \longmapsto
\sum_{j: p_j \ne \infty}\frac{1}{p_j} \log \big(\essinf_Q \mathscr{Q}_j(\bm{f}, \bm{g}) \big) 
\quad \text{is concave}.
\end{align}
Thus, the convexity of $\Pi_{\mathscr{D}'}$ follows from \eqref{eq:pcx-1} and \eqref{eq:pcx-2}.

In order to justify the weak closedness of $\Pi_{\mathscr{D}'}$, let $\{(\bm{a}_k, \bm{b}_k)\}_{k \ge 1} \subset \Pi_{\mathscr{D}'}$ and $(\bm{a}_k, \bm{b}_k) \to (\bm{a}, \bm{b})$ in $(L^1(\mu))^{2(m+1)}$ norm. Then there exists a subsequence $\{(\bm{a}_{k_{\ell}}, \bm{b}_{k_{\ell}})\}_{\ell \ge 1} \subset \{(\bm{a}_k, \bm{b}_k)\}_{k \ge 1}$ such that 
$(\bm{a}_{k_{\ell}}, \bm{b}_{k_{\ell}}) \to (\bm{a}, \bm{b})$ a.e., where we have used the fact $d\mu \ll dx \ll d\mu$. Then for ech $j=1 ,\ldots, m+1$,
\begin{align}\label{eq:fee-1}
f_j^{\ell} := e^{a_{k_{\ell}, j}} h_j \to e^{a_j} h_j =: f_j
\quad \text{ and } \quad
g_j^{\ell} := e^{b_{k_{\ell}, j}} h_j \to e^{a_j} h_j =: g_j, \quad \text{a.e.}, 
\end{align} 
which implies 
\begin{align}\label{eq:fee-2}
\mathscr{Q}_j(\bm{f}^{\ell}, \bm{g}^{\ell}) \to \mathscr{Q}_j(\bm{f}, \bm{g}) \quad \text{a.e.}. 
\end{align} 
By Lemma \ref{lem:EE}, $0 < \langle f_j \rangle_{r_j, Q} < \infty$ for each $j=1, \ldots, m+1$. Now invoking Fatou's lemma and \eqref{eq:fee-1}--\eqref{eq:fee-2}, we arrive at
\begin{align*}
\mathscr{P}_Q (\bm{f}) 
\le \liminf_{\ell} \mathscr{P}_Q (\bm{f}^{\ell})
& \lesssim \mathscr{M} \limsup_{\ell} \prod_{j: p_j \ne \infty}
\big(\essinf_Q \mathscr{Q}_j (\bm{f}^{\ell}, \bm{g}^{\ell}) \big)^{\frac{1}{p_j}}
\\
& \le \mathscr{M} \prod_{j: p_j \ne \infty}
\big(\limsup_{\ell} \essinf_Q \mathscr{Q}_j (\bm{f}^{\ell}, \bm{g}^{\ell}) \big)^{\frac{1}{p_j}}
\\
& \le \mathscr{M} \prod_{j: p_j \ne \infty}
\big(\essinf_Q \mathscr{Q}_j (\bm{f}, \bm{g}) \big)^{\frac{1}{p_j}}.
\end{align*}
Thus, $(\bm{a}, \bm{b}) \in \Pi_{\mathscr{D}'}$, which means $\Pi_{\mathscr{D}'}$ is closed. By \cite[Theorem 3.7]{Bre},
$\Pi_{\mathscr{D}'}$ is weakly closed. So far, we have proved \eqref{eq:EDcon}.

\subsection{Proof of Theorem \ref{thm:RIBFS}}
Fix $\vec{u} \in \bigcap_{\vec{q} \in \V_{\vec{\X}}} A_{\vec{q}, (\vec{r}, s)}$ and let $\Y = \X^{\frac1m}$. We claim that 
\begin{align}\label{eq:mxx-1}
\M_{m \vec{r}, (ms)'} : \X_1^{\frac1m}(u_1^{\frac1m}) \times \cdots \times 
\X_m^{\frac1m}(u_m^{\frac1m}) \times \Y'(u^{-\frac1m}) \to L^1(\Rd) \quad \text{is bounded}. 
\end{align}
Assuming \eqref{eq:mxx-1} holds momentarily, let us conclude the proof of Theorem \ref{thm:RIBFS} as follows. Given a Banach function $\Z$ and a weight $v$ on $\Rd$, define 
\begin{align*}
\Z(v) :=\{f \in \bM: fv \in \Z\}
\quad \text{ and } \quad 
\|f\|_{\Z(v)} := \|fv\|_{\Z},
\end{align*}
which satisfy 
\begin{align}\label{eq:urrr}
\Z(v)' = \Z'(v^{-1})
\quad \text{ and } \quad 
\Z(v)^r = \Z^r (v^r),
\quad \forall r > 0.
\end{align}
Denote
\begin{align*}
\widetilde{\X}_j := \X_j(u_j), \quad
\widetilde{\X} := \prod_{j=1}^m \widetilde{\X}_j, 
\quad \text{ and } \quad 
\widetilde{\Y} := \widetilde{\X}^{\frac1m}.
\end{align*}
Then, $\widetilde{\X}_j$ and $\widetilde{\Y}$ are Banach function spaces, and by \eqref{eq:urrr},
\begin{align*}
\widetilde{\X}_j^{\frac1m} = \X_j^{\frac1m}(u_j^{\frac1m}), \quad
\widetilde{\X} = \X(u), \quad
\widetilde{\Y} = \X^{\frac1m}(u^{\frac1m}) = \Y(u^{\frac1m}),
\quad \text{and} \quad 
\widetilde{\Y}' = \Y'(u^{-\frac1m}).
\end{align*}
Moreover, \eqref{eq:mxx-1} is equivalent to
\begin{align}\label{eq:mxx-2}
\M_{m \vec{r}, (ms)'} : \widetilde{\X}_1^{\frac1m} \times \cdots \times 
\widetilde{\X}_m^{\frac1m} \times \widetilde{\Y}' \to L^1(\Rd) \quad \text{is bounded}. 
\end{align}
Thus, Theorem \ref{thm:RIBFS} follows from \eqref{eq:mxx-2} and Theorem \ref{thm:BFS}.

Let us turn to the proof of \eqref{eq:mxx-1}. Recall that $\vec{u} \in \bigcap_{\vec{q} \in \V_{\vec{\X}}} A_{\vec{q}, (\vec{r}, s)}$. We are going to show that 
\begin{equation}\label{eq:eng}
\begin{aligned}
&\exists \varepsilon \in (0, 1) \text{ such that }
p_{\X_j}^- - \varepsilon > r_j, \quad
\sum_{j=1}^m \frac{1}{p_{\X_j}^+ + \varepsilon} > \frac{1}{s},
\\
&\text{and } \vec{u} \in \bigcap_{\vec{q} \in \V_{\vec{\X}}^{\varepsilon}} A_{\vec{q}, (\vec{r}, s)}, 
\text{ where } \V_{\vec{\X}}^{\varepsilon} 
:= \prod_{j=1}^m \{p_{\X_j}^- - \varepsilon, \, p_{\X_j}^+ + \varepsilon\}.
\end{aligned}
\end{equation}
Indeed, given $\vec{q} \in \V_{\vec{\X}}$, by Lemma \ref{lem:open}, there exists some $\lambda_{\vec{q}} \in (1, \infty)$ so that 
\begin{align}\label{eq:pxs-1}
\sup_{Q \in \Q} \langle u \rangle_{Q, \frac{\lambda_{\vec{q}}}{\frac1q - \frac1s}}
\prod_{j=1}^m \langle u_j^{-1} \rangle_{Q, \frac{\lambda_{\vec{q}}}{\frac{1}{r_j} - \frac{1}{q_j}}} 
\lesssim [\vec{u}]_{A_{\vec{q}, (\vec{r}, s)}} < \infty.
\end{align}
Given $\varepsilon \in (0, 1)$, define 
\begin{equation*}
q_j(\varepsilon) = 
\begin{cases}
q_j - \varepsilon & \text{if } q_j = p_{\X_j}^-,
\\
q_j + \varepsilon & \text{if } q_j = p_{\X_j}^+,
\end{cases}
\quad \text{ and } \quad
\frac{1}{q(\varepsilon)} = \sum_{j=1}^m \frac{1}{q_j(\varepsilon)}. 
\end{equation*}
Since
\begin{align*}
& \lim_{\varepsilon \to 0^+} \frac{1}{\frac{1}{q(\varepsilon)} - \frac1s} 
= \frac{1}{\frac{1}{q} - \frac1s} 
< \frac{\lambda_{\vec{q}}}{\frac1q - \frac1s},
\\
& \lim_{\varepsilon \to 0^+} \frac{1}{\frac{1}{r_j} - \frac{1}{q_j(\varepsilon)}}
= \frac{1}{\frac{1}{r_j} - \frac{1}{q_j}}
< \frac{\lambda_{\vec{q}}}{\frac{1}{r_j} - \frac{1}{q_j}},
\end{align*}
one can choose $\varepsilon \in (0, 1)$ small enough so that $p_{\X_j}^- - \varepsilon > r_j$,
\begin{align}\label{eq:pxs-2}
\sum_{j=1}^m \frac{1}{p_{\X_j}^+ + \varepsilon} > \frac{1}{s}, \quad
\frac{1}{\frac{1}{q(\varepsilon)} - \frac1s} 
< \frac{\lambda_{\vec{q}}}{\frac1q - \frac1s}, \quad \text{and} \quad
\frac{1}{\frac{1}{r_j} - \frac{1}{q_j(\varepsilon)}}
< \frac{\lambda_{\vec{q}}}{\frac{1}{r_j} - \frac{1}{q_j}}.
\end{align}
Thus, \eqref{eq:pxs-1} and \eqref{eq:pxs-2} imply
\begin{align*}
[\vec{u}]_{A_{\vec{q}(\varepsilon), (\vec{r}, s)}}
= \sup_{Q \in \Q} \langle u \rangle_{Q, \frac{1}{\frac{1}{q(\varepsilon)} - \frac1s}}
\prod_{j=1}^m \langle u_j^{-1} \rangle_{Q, \frac{1}{\frac{1}{r_j} - \frac{1}{q_j(\varepsilon)}}} 
< \infty,
\end{align*}
which justifies \eqref{eq:eng}. Note that
\begin{align}\label{eq:uum}
\vec{u} \in A_{\vec{q}, (\vec{r}, s)} 
\quad \iff \quad
(u_1^{\frac1m}, \ldots, u_m^{\frac1m}) \in A_{m\vec{q}, (m\vec{r}, ms)}.
\end{align}
Then by \eqref{eq:eng} and \eqref{eq:uum}, for all $\vec{q} \in \V_{\vec{\X}}^{\varepsilon}$, one has $(u_1^{\frac1m}, \ldots, u_m^{\frac1m}) \in A_{m\vec{q}, (m\vec{r}, ms)}$, which together with \cite[Proposition 2.7]{Nie19} yields 
\begin{align*}
\|\M_{m \vec{r}, (ms)'}(\vec{f}, f)\|_{L^1}
\lesssim \prod_{j=1}^m \|f_j u_j^{\frac1m}\|_{L^{mq_j}}
\|f u^{-\frac1m}\|_{L^{(mq)'}}.
\end{align*}
Equivalently, for all $\vec{q} \in \V_{\vec{\X}}^{\varepsilon}$, 
\begin{align}\label{eq:mxx-3}
\|\M_{m \vec{r}, (ms)'}(f_1 u_1^{-\frac1m}, \ldots, f_m u_m^{-\frac1m}, f u^{\frac1m})\|_{L^1}
\lesssim \prod_{j=1}^m \|f_j\|_{L^{mq_j}} \|f\|_{L^{(mq)'}}.
\end{align}

We would like to apply the multilinear Boyd interpolation (\cite[Corollary 3.3]{GK}) to the multisublinear operator $T(\vec{f}, f) = \M_{m \vec{r}, (ms)'}(f_1 u_1^{-\frac1m}, \ldots, f_m u_m^{-\frac1m}, f u^{\frac1m})$. For this purpose, let us introduce some notation and then verify the hypotheses in \cite[Corollary 3.3]{GK}. Since $\X_j^{\frac1m}, \Y, \Y'$ are rearrangement invariant Banach function spaces, we set 
\begin{align*}
B_{\vec{\X}} 
& := \prod_{j=1}^m \bigg[\frac{1}{p_{\X_j^{1/m}}^+}, \, \frac{1}{p_{\X_j^{1/m}}^-} \bigg] 
\times \bigg[\frac{1}{p_{\Y'}^+}, \, \frac{1}{p_{\Y'}^-} \bigg] 
= \prod_{j=1}^m \bigg[\frac{1}{m p_{\X_j}^+}, \, \frac{1}{m p_{\X_j}^-} \bigg]
\times \bigg[\frac{1}{p_{\Y'}^+}, \, \frac{1}{p_{\Y'}^-} \bigg],
\\
\Theta & := \bigg\{(\theta_1, \ldots, \theta_m, \theta): \, \theta = 1 - \sum_{j=1}^m \theta_j, \,
\theta_j = \frac{1}{mq_j}, \, \vec{q} \in \V_{\vec{\X}}^{\varepsilon} \bigg\}.
\end{align*} 
Let $\conv \Theta$ be the convex hull of $\Theta$, and $\inter (\conv \Theta)$ be the open convex hull of $\Theta$. Observe that
\begin{align}\label{eq:bxt}
\varnothing 
\ne B_{\vec{\X}} \cap \conv \Theta 
\subsetneq \inter (\conv \Theta),
\end{align}
which is a consequence of the following facts
\begin{align*}
\conv \Theta 
& = \bigg\{(\vec{\theta}, \theta): \, 
\theta_j \in \bigg[\frac{1}{m p_{\X_j}^+ + m \varepsilon}, \,
\frac{1}{m p_{\X_j}^- - m \varepsilon} \bigg], \, 
\theta = 1 - \sum_{j=1}^m \theta_j \bigg\},
\\
\inter(\conv \Theta)
& = \bigg\{(\vec{\theta}, \theta): \, 
\theta_j \in \bigg(\frac{1}{m p_{\X_j}^+ + m \varepsilon}, \,
\frac{1}{m p_{\X_j}^- - m \varepsilon} \bigg), \, 
\theta = 1 - \sum_{j=1}^m \theta_j \bigg\},
\\
& \text{and } \quad
\prod_{j=1}^m \frac{1}{p_{\X_j^{1/m}}^+} + \frac{1}{p_{\Y'}^+}
\le 1 \le 
\prod_{j=1}^m \frac{1}{p_{\X_j^{1/m}}^-} + \frac{1}{p_{\Y'}^-},
\end{align*} 
where the last one was shown in \cite[eq. (17)]{GK}. On the other hand, by the fact $\Y' \prod_{j=1}^m \X_j^{\frac1m} = L^1(\Rd)$, for all $f_j \in \X_j^{\frac1m}$ and $f \in \Y'$, 
\begin{align*} 
\bigg\| \prod_{j=1}^m f_j f \bigg\|_{L^1}
\lesssim \bigg\| \prod_{j=1}^m f_j f \bigg\|_{\Y' \prod_{j=1}^m \X_j^{\frac1m}}
\le \prod_{j=1}^m \|f_j\|_{\X_j^{\frac1m}} \|f\|_{\Y'},
\end{align*} 
which further yields 
\begin{align}\label{eq:rea} 
& \int_0^1 \prod_{j=1}^m f_j^*(t) f^*(t)\, dt
= \int_{\Rd} \prod_{j=1}^m f_j^*(\w_d |x|^d) f^*(\w_d |x|^d) \, dx
\\ \nonumber
& \lesssim \prod_{j=1}^m \|f_j^*(\w_d |\cdot|)\|_{\X_j^{\frac1m}} \|f^*(\w_d |\cdot|)\|_{\Y'}
= \prod_{j=1}^m \|f_j\|_{\X_j^{\frac1m}} \|f\|_{\Y'},
\end{align} 
provided that $\X_j^{\frac1m}$ and $\Y'$ are rearrangement invariant, and 
\begin{align*} 
& |\{x \in \Rd: h^*(\w_d |x|^d) > \lambda\}|
= \int_{\Rd} \mathbf{1}_{\{x \in \Rd: h^*(\w_d |x|^d) > \lambda\}} \, dx
\\
& = \int_0^{\infty} d \w_d \mathbf{1}_{\{r>0: h^*(\w_d r^d) > \lambda\}} r^{d-1} \, dr
= \int_0^{\infty} \mathbf{1}_{\{r>0: h^*(t) > \lambda\}} \, dt
\\
& = |\{t>0: h^*(t) > \lambda\}|
= |\{x \in \Rd: h(x) > \lambda\}|.
\end{align*} 
By \cite[Corollary 3.8]{GK}, \eqref{eq:bxt} and \eqref{eq:rea} imply the $(\Theta, 1)$ interpolation condition. In view of \cite[Theorem 3.2]{GK}, the latter is equivalent to the condition (6) in \cite{GK}. Since \eqref{eq:mxx-3} gives the local continuity of $T$, it follows from \cite[Corollary 3.3]{GK} that 
\begin{align*}
\|\M_{m \vec{r}, (ms)'}(f_1 u_1^{-\frac1m}, \ldots, f_m u_m^{-\frac1m}, f u^{\frac1m})\|_{L^1}
\lesssim \prod_{j=1}^m \|f_j\|_{\X_j^{\frac1m}} \|f\|_{\Y'}.
\end{align*}
This is equivalent to \eqref{eq:mxx-1}.

\section{Extrapolation of compactness}\label{sec:cpt}
This section is devoted to proving Theorem \ref{thm:BFScpt} by means of Theorem \ref{thm:BFS}. Although we still need to invoke \cite[Theorem 1.1]{COY}, our current proof strategy is different from that of \cite[Theorem 1.1]{COY}, because the multilinear interpolation of compactness needed there is entirely unavailable in the general Banach function spaces setting.

\subsection{A compactness criterion}
Given $r, a >0$, definde
\begin{align*}
S_r^a f(x) 
:= \bigg( \fint_{B(0, r)} |\tau_y f(x) - f(x)|^a \, dy \bigg)^{\frac1a}, \quad x \in \Rd,
\end{align*}
where $\tau_y$ is the translation operator defined by $\tau_y f (x) := f(x+y)$. 

\begin{lemma}\label{lem:ccc}
Let $0 < a \le 1$ and $\X$ be a complete quasi-Banach function lattice such that
$\X \hookrightarrow L^a(Q)$ and $\mathbf{1}_Q \in \X$ for every cube $Q$. If a subset $\K \subset \X$ satisfies the following
\begin{list}{\rm (\theenumi)}{\usecounter{enumi}\leftmargin=1cm \labelwidth=1cm \itemsep=0.2cm \topsep=0.2cm \renewcommand{\theenumi}{\alph{enumi}}}
 
\item\label{eq:cc-1} ${\displaystyle \sup_{f \in \K} \|f\|_{\X} < \infty}$, 

\item\label{eq:cc-2} ${\displaystyle \lim_{A \to \infty} \sup_{f \in \K} 
\|f \mathbf{1}_{B(0, A)^c}\|_{\X}=0}$, 

\item\label{eq:cc-3} ${\displaystyle \lim_{r \to 0} \sup_{f \in \K} \|S_r^a f\|_{\X} = 0}$. 
\end{list} 
then $\K$ is precompact in $\X$.
\end{lemma}

\begin{proof}
Given $N \ge 1$, let $\Omega_N := [-2^N, 2^N)^{-d}$ and $\D_N$ be the collection of dyadic cubes of sidelength $2^{-N}$ in $\Omega_N$. Define
\begin{align*}
\alpha_Q(f) 
:= \inf_{\alpha \in \R} \langle |f - c| \rangle_{Q, a}
\quad \text{ and } \quad
P_N f: = \sum_{Q \in \D_N} \alpha_Q(f) \mathbf{1}_Q.
\end{align*}
It follows from Lemma \ref{lem:XQ} that $\mathbf{1}_Q \in \X$ and 
\begin{align*}
|\alpha_Q(f)|
\lesssim \langle f - \alpha_Q(f) \rangle_{Q, a} + \langle f \rangle_{Q, a}
\le 2 \langle f \rangle_{Q, a}
\le C_Q \|f\|_{\X}.
\end{align*}
In light of the condition \eqref{eq:cc-1} and $\# \D_N = 2^{(2N+1)d}$, $P_N$ is a finite dimensional operator and $\{P_N f: f \in \K\}$ is a bounded subset of $\X$. Thus, $\{P_N f: f \in \K\}$ is totally bounded in $\X$. 

Let $\varepsilon > 0$ and fix $f \in \K$. Then one can find a finite number of functions $\{f_1, \ldots, f_{k_0}\} \subset \K$ such that
\begin{align}\label{eq:err-1}
\|P_N f - P_N f_k\|_{\X} \le \varepsilon, \text{ for some } k \in \{1, \ldots, k_0\}.
\end{align}
By the conditions \eqref{eq:cc-2} and \eqref{eq:cc-3}, there exists some $A_0 = A_0(\varepsilon) \ge 1$ such that 
\begin{align}\label{eq:err-2}
\|f \mathbf{1}_{B(0, A)^c}\|_{\X} + \|S_r^a f\|_{\X} 
\le \varepsilon, \quad \forall A > A_0, \, \,0 < r \le A_0^{-1}.
\end{align}
Note that $P_N f(x) = 0$ for all $x \notin \Omega_N$. If $x \in \Omega_N$, then there exists some $Q_x \in \D_N$ such that $x \in Q_x$. Then for all $y \in Q_x$,
\begin{align*}
|f(x) - \alpha_{Q_x}(f)|^a 
\le |f(x) - \alpha_{Q_x}(f)|^a + |f(x) - f(y)|^a,
\end{align*}
which leads to 
\begin{align}\label{eq:err-3}
|f(x) - P_N f(x)|
& = |f(x) - \alpha_{Q_x}(f)| 
\le 2^{\frac1a} \langle f(x) - f(\cdot) \rangle_{Q_x, a}
\\ \nonumber
& \lesssim \langle f(x) - f(\cdot) \rangle_{B(x, 2^{-N} \sqrt{d}), a}
= S_{2^{-N} \sqrt{d}}^a f(x).
\end{align}
Pick $N = N(\varepsilon) \ge 1$ large enough so that $2^N > A_0 \sqrt{d}$. Then invoking \eqref{eq:err-1}-- \eqref{eq:err-3} and the ideal property \eqref{bfs-4}, we have  
\begin{align*}
\|f - f_k\|_{\X}
& \lesssim \|(f - P_N f) \mathbf{1}_{\Omega_N^c}\|_{\X} 
+ \|(f - P_N f) \mathbf{1}_{\Omega_N}\|_{\X} 
+ \|P_N f - P_N f_k\|_{\X}
\\
& \lesssim \|f \mathbf{1}_{B(0, 2^N)^c}\|_{\X} 
+ \|S_{2^{-N} \sqrt{d}}^a f\|_{\X}
+ \|P_N f - P_N f_k\|_{\X}
\lesssim \varepsilon.
\end{align*}
This shows $\K$ is totally bounded in $\X$. In metric spaces, a set is compact if and only if it is complete and totally bounded. Therefore, $\K$ is precompact in $\X$.
\end{proof}

\begin{lemma}\label{lem:XQ}
Let $\X_1, \ldots, \X_m$ be Banach function spaces such that
\begin{align*}
\M_{m \vec{r}, (ms)'}: \X_1^{\frac1m} \times \cdots \times \X_m^{\frac1m}\times \Y' \to L^1 
\quad\text{is bounded},
\end{align*}
where $\Y = \X^{\frac1m}$ and $\X = \prod_{j=1} \X_j$. Then for every cube $Q$,
\begin{align*}
\X_j \hookrightarrow L^{r_j}(Q), \quad
\X \hookrightarrow L^r(Q),
\quad \text{and } \quad
\mathbf{1}_Q \in \X,
\end{align*}
where the implicit constants may depend on $Q$.
\end{lemma}

\begin{proof}
Fix a cube $Q$. It follows from Lemma \ref{lem:EE} that
\begin{align*}
\|f_j\|_{L^{r_j}(Q)}^{\frac1m} |Q|^{-\frac{1}{mr_j}}
= \langle |f_j|^{\frac1m} \rangle_{mr_j, Q}
\le C_{j, Q} \langle |f_j|^{\frac1m} \rangle_{\X_j^{\frac1m}} 
= C_{j, Q} \|f_j\|_{\X_j}^{\frac1m},
\end{align*}
and
\begin{align*}
\|\mathbf{1}_Q f_{m+1} \|_{L^1} |Q|^{-1}
= \langle f_{m+1} \rangle_Q 
\le \langle f_{m+1} \rangle_{Q, (ms)'} 
\le C_{m+1, Q} \|f_{m+1}\|_{\Y'}.
\end{align*}
Thus, $\X_j \hookrightarrow L^{r_j}(Q)$ and $\mathbf{1}_Q \in \Y'' = \Y$. The latter further gives $\mathbf{1}_Q \in \Y^m = \X$. In addition, given $f \in \X = \prod_{j=1}^m \X_j$, there exist $f_j \in \X_j$ such that $|f| \le \prod_{j=1}^m |f_j|$, which along with H\"older's inequality implies
\begin{align*}
\|f\|_{L^r(Q)}
\le \prod_{j=1}^m \|f_j\|_{L^{r_j}(Q)} 
\le \prod_{j=1}^m C_{j, Q}^m |Q|^{\frac{1}{r_j}} \|f_j\|_{\X_j},
\end{align*}
Thus, $\X \hookrightarrow L^r(Q)$.
\end{proof}

\subsection{Proof of Theorem \ref{thm:BFScpt}}
Using multilinear extrapolation of compactness \cite[Theorem 1.1]{COY} and the hypotheses of Theorem \ref{thm:BFScpt}, we conclude that 
\begin{align}\label{eq:wprs-1}
T: L^{q_1}_{v_1} \times \cdots \times L^{q_m}_{v_m} \longrightarrow L^q_v \, \text{ is compact}
\end{align} 
for all $\vec{\q} = (\q_1, \ldots, \q_m)$ with $(\vec{r}, s) \prec \vec{\q}$, and for all $\vec{v} \in A_{\vec{\q}, (\vec{r}, s)}$, where $\frac{1}{\q} = \sum_{j=1}^m \frac{1}{\q_j}$ and $v = \prod_{j=1}^m v_j$.

Let $\frac1r = \sum_{j=1}^m \frac{1}{r_j}$. Pick $\max\{1, r\} < \p < s$ and $\p_j = \p r_j/r$, $j=1, \ldots, m$, so that $(\vec{r}, s) \prec \vec{\p}$. Then \eqref{eq:wprs-1} applied to the exponents $\vec{\p}$ yields 
\begin{align}\label{eq:wprs-2}
T: L^{\p_1}_{w_1} \times \cdots \times L^{\p_m}_{w_m} \longrightarrow L^p_w \, \text{ is compact}
\end{align} 
for all $\vec{w} \in A_{\vec{\p}, (\vec{r}, s)}$, where $w = \prod_{j=1}^m w_j$. Furthermore, given $\vec{w} \in A_{\vec{\p}, (\vec{r}, s)}$, it was shown in \cite[Lemma 5.3]{LMO} that $w^{\frac{1}{1/\p - 1/s}} \in A_{(\frac1r - \frac1s)/(\frac{1}{\p} - \frac1s)}$, which implies 
\begin{align}\label{eq:wprs-3}
w^{\p} \in A_{(\frac1r - \frac1s)/(\frac{1}{\p} - \frac1s)} =: A_{\p_0}, 
\quad \forall \vec{w} \in A_{\vec{\p}, (\vec{r}, s)}.
\end{align}
Take $0 < a < \min\{\p/\p_0, 1\}$. Then by \eqref{eq:wprs-2}--\eqref{eq:wprs-3} and the Kolmogorov--Riesz theorem (\cite[Theorem 2.10]{COY} or \cite[Theorem 1.5]{CLSY}), given $\varepsilon > 0$, there exist $A_0 = A_0(\varepsilon) > 1$ such that 
\begin{align}
\label{eq:SN1}
\sup_{A > A_0} 
\|\mathbf{1}_{B(0, A)^c} T\|_{L^{\p_1}_{w_1} \times \cdots \times L^{\p_m}_{w_m} \to L^{\p}_w} 
& \lesssim \varepsilon, \quad \forall \vec{w} \in A_{\vec{\p}, (\vec{r}, s)},
\\
\label{eq:SN2}
\sup_{0 < \rho < A_0^{-1}} 
\|S_{\rho}^a T\|_{L^{\p_1}_{w_1} \times \cdots \times L^{\p_m}_{w_m} \to L^{\p}_w} 
& \lesssim \varepsilon, \quad \forall \vec{w} \in A_{\vec{\p}, (\vec{r}, s)},
\end{align}
where $w = \prod_{j=1}^m w_j$. Fix $A > A_0$ and $0 < \rho < A_0^{-1}$. Set
\begin{align*}
\F &:= \big\{\big(T(\vec{f}), |f_1|, \ldots, |f_m| \big): f_j \in \X_j, \, j=1, \ldots, m\big\},
\\
\F_A &:= \big\{\big(\mathbf{1}_{B(0, A)^c} T(\vec{f}), |f_1|, \ldots, |f_m| \big): f_j \in \X_j, \, j=1, \ldots, m\big\},
\\
\F_{\rho} &:= \big\{\big(S_{\rho}^a T(\vec{f}), |f_1|, \ldots, |f_m| \big): f_j \in \X_j, \, j=1, \ldots, m\big\}.
\end{align*}
In view of the boundedness hypothesis of Theorem \ref{thm:BFScpt} and \eqref{eq:SN1}--\eqref{eq:SN2}, Theorem \ref{thm:BFS} applied to $\F, \F_A, \F_{\rho}$ respectively implies
\begin{align}
\label{eq:xc-1}
\|T(\vec{f})\|_{\X} & \lesssim \prod_{j=1}^m \|f_j\|_{\X_j},
\\
\label{eq:xc-2}
\|\mathbf{1}_{B(0, A)^c} T(\vec{f})\|_{\X} 
& \lesssim \varepsilon \prod_{j=1}^m \|f_j\|_{\X_j},
\\
\label{eq:xc-3}
\|S_{\rho}^a T(\vec{f})\|_{\X} 
& \lesssim \varepsilon \prod_{j=1}^m \|f_j\|_{\X_j},
\end{align}
where the implicit constants are independent of $A$, $\rho$, $\varepsilon$, and $\vec{f}$. Hence, it follows from \eqref{eq:xc-1}--\eqref{eq:xc-3} and Lemmas \ref{lem:ccc} and \ref{lem:XQ} that $T$ is compact from $\X_1 \times \cdots \times \X_m$ to $\X$.

\section{Boundedness of multilinear maximal operators}\label{sec:MM}
In this section, we will demonstrate Theorems \ref{thm:MM} and \ref{thm:example}. 

\subsection{Self-improvement estimates}
Let us begin with a self-improvement estimate for the multisublinear Hardy--Littlewood maximal operator. 
\begin{lemma}\label{lem:Mimp}
Let $\s_{m+1} \ge 1$ and $\s_j \ge m$, $j=1, \ldots, m$. Let $\Y_1, \ldots, \Y_m$ be Banach function spaces such that $\Y = \prod_{j=1}^m \Y_j$ is also a Banach function space. If $\M_{\vec{\s}, \s_{m+1}}$ is bounded from $\Y_1 \times \cdots \times \Y_m \times \Y'$ to $L^1(\Rd)$, then there exists some $\varepsilon_0 > 0$ such that $M_{\s_{m+1}^*}$ is bounded on $\Y'$ for all $\s_{m+1}^* \in (\s_{m+1}, \s_{m+1} + \varepsilon_0)$. In particular, 
\begin{align*}
\|\M_{\vec{\s}, \s_{m+1}^*}\|_{\Y_1 \times \cdots \times \Y_m \times \Y' \to L^1} < \infty. 
\end{align*}
\end{lemma}

\begin{proof}
Let $0 \le g \in \Y$. Then there exist nonnegative function $f_j \in \Y_j$ such that $g \le \prod_{j=1}^m f_j$ and $\prod_{j=1}^m \|f_j\|_{\Y_j} \le 2 \|g\|_{\Y}$. By H\"{o}lder's and Jensen's inequalities
\begin{align}\label{eq:gqm}
\langle g \rangle_Q
\le \prod_{j=1}^m \langle f_j \rangle_{Q, m}
\le \prod_{j=1}^m \langle f_j \rangle_{Q, \s_j}.
\end{align}
As it was shown in \cite[eq. (3.4)]{Ler13}, there exist $\frac12$-sparse families $\{\S_j\}_{j=1}^{3^d}$ such that
\begin{align}\label{eq:Msp}
M f(x) 
\lesssim \sum_{k=1}^{3^d} \sum_{Q \in \S_k} \langle f \rangle_{Q, 1} \mathbf{1}_Q(x), 
\quad \text{ a.e. } x \in \Rd.
\end{align}
Note that $(\sum_j |a_j|)^{\theta} \le \sum_j |a_j|^{\theta}$ for each $\theta \in (0, 1]$. Invoking \eqref{eq:gqm} and \eqref{eq:Msp}, we deduce
\begin{align*}
\langle M_{\s_{m+1}} f, g \rangle
& \lesssim \sum_{k=1}^{3^d} \sum_{Q \in \S_k} 
|Q| \langle f \rangle_{Q, \s_{m+1}} \langle g \rangle_{Q, 1}
\le \sum_{k=1}^{3^d} \sum_{Q \in \S_k} 
|Q| \langle f \rangle_{Q, \s_{m+1}} \prod_{j=1}^m \langle f_j \rangle_{Q, \s_j}
\\
& \lesssim \|\M_{\vec{s}, \s_{m+1}}(\vec{f}, f)\|_{L^1}
\lesssim \|f\|_{\Y'} \prod_{j=1}^m \|f_j\|_{\Y_j} 
\lesssim \|f\|_{\Y'} \|g\|_{\Y},
\end{align*}
which implies
\begin{align*}
\|M_{\s_{m+1}}\|_{\Y' \to \Y'} \lesssim 1.
\end{align*}
Setting $\Z = (\Y')^{\s_{m+1}}$, we have 
\begin{align}\label{eq:MZ-1}
\|M\|_{\Z \to \Z} \lesssim 1. 
\end{align}
Then define the Rubio de Francia algorithm 
\begin{align*}
\cR h := \sum_{k=0}^{\infty} 
\frac{M^k h}{2^{k \s_{m+1}} \|M\|_{\Z \to \Z}^k},
\end{align*}
which gives
\begin{align}\label{eq:MZ-2}
|h| \le \cR h, \quad \|\cR h\|_{\Z} \le 2^{\s_{m+1}} \|h\|_{\Z}, 
\quad \text{and} \quad
\cR h \in A_1,
\end{align}
where we have used
\begin{align*}
\bigg\|\sum_k |f_k|\bigg\|_{\Z}^{\frac{1}{\s_{m+1}}}
\le \bigg\| \sum_k |f_k|^{\frac{1}{\s_{m+1}}}\bigg\|_{\Y'}
\le \sum_k \big\| |f_k|^{\frac{1}{\s_{m+1}}} \big\|_{\Y'}
= \sum_k \| |f_k|\|_{\Z}^{\frac{1}{\s_{m+1}}}.
\end{align*}
By the reverse H\"{o}lder inequality (cf. \cite[Theorem 7.2.2]{Gra1}), there exists some $\eta_0 \in (0, 1)$ such that
\begin{align}\label{eq:MZ-3}
\langle \cR h \rangle_{Q, 1+\eta_0}
\lesssim \langle \cR h \rangle_Q, \quad \forall Q \in \Q. 
\end{align}
Collecting \eqref{eq:MZ-1}--\eqref{eq:MZ-3}, we obtain that for every $0 < \eta < \eta_0$,
\begin{align*}
\|M_{1+\eta} h\|_{\Z}
\le \|M_{1+\eta_0} h\|_{\Z}
\le \|M_{1+\eta_0} (\cR h)\|_{\Z} 
\lesssim \|M (\cR h)\|_{\Z}
\lesssim \|\cR h\|_{\Z}
\lesssim \|h\|_{\Z},
\end{align*}
which leads to
\begin{align}\label{eq:MZ-4}
\|M_{(1+\eta) \s_{m+1}}\|_{\Y' \to \Y'}
\lesssim 1, \quad 0 < \eta < \eta_0.
\end{align}
In view of the fact 
\begin{align*}
\M_{\vec{\s}, (1+\eta) \s_{m+1}}(\vec{f}, f_{m+1}) 
\le \M_{\vec{\s}, \s_{m+1}}(\vec{f}, M_{(1+\eta) \s_{m+1}}f_{m+1}),
\end{align*}
the boundedness of $\M_{\vec{\s}, \s_{m+1}}$ and \eqref{eq:MZ-4} imply 
$\|\M_{\vec{\s}, (1+\eta) \s_{m+1}}\|_{\Y_1 \times \cdots \times \Y_{m+1} \to L^1} < \infty$.
\end{proof}

The nest result is a self-improvement estimate for dyadic operators. 
\begin{lemma}\label{lem:imp}
Let $\X = \prod_{j=1}^m \X_j$ with $\X_j$ being a Banach function space. Let $0 < r_1 < r_2 < \infty$, $\D$ be a dyadic lattice, $\D' \subset \D$. For each $Q \in \D'$, let $\Lambda_Q: \X_1 \times \cdots \times \X_m \to [0, \infty)$ be a mapping such that 
\begin{align*}
\bigg\| \bigg(\sum_{Q \in \D'} \Lambda_Q(\vec{f})^{r_2} \mathbf{1}_Q\bigg)^{\frac{1}{r_2}} \bigg\|_{\X}
\lesssim \prod_{j=1}^m \|f_j\|_{\X_j}
\quad \text{ and } \quad
\Lambda_Q(\vec{f}) \gtrsim \essinf_Q \prod_{j=1}^m |f_j|
\end{align*}
for all $Q \in \D'$ and $f_j \in \X_j$. Then 
\begin{align*}
\bigg\| \bigg(\sum_{Q \in \D'} a_Q^{r_1} \mathbf{1}_Q \bigg)^{\frac{1}{r_1}} \bigg\|_{\X}
\lesssim \bigg\| \bigg(\sum_{Q \in \D'} a_Q^{r_2} \mathbf{1}_Q \bigg)^{\frac{1}{r_2}} \bigg\|_{\X},
\end{align*}
whenever the right hand side is finite, where $\{a_Q\}_{Q \in \D'}$ is a sequence of nonnegative numbers.
\end{lemma}

\begin{proof}
Let $k_0 = \lfloor r_2/r_1 \rfloor$. Suppose that $\big\| \big(\sum_{Q \in \D'} a_Q^{r_2} \mathbf{1}_Q \big)^{\frac{1}{r_2}} \big\|_{\X} < \infty$. By definition, there exist nonnegative functions $f_j^1 \in \X_j$ such that
\begin{align*}
\bigg(\sum_{Q \in \D'} a_Q^{r_2} \mathbf{1}_Q \bigg)^{\frac{1}{r_2}} 
\le \prod_{j=1}^m f_j^1
\quad \text{ and } \quad 
\prod_{j=1}^m \|f_j^1\|_{\X_j} 
\le 2 \bigg\| \bigg(\sum_{Q \in \D'} a_Q^{r_2} \mathbf{1}_Q \bigg)^{\frac{1}{r_2}} \bigg\|_{\X},
\end{align*}
which together with the assumption gives 
\begin{align*}
\bigg\| \bigg(\sum_{Q \in \D'} \Lambda_Q(f_1^1, \ldots, f_m^1)^{r_2} 
\mathbf{1}_Q\bigg)^{\frac{1}{r_2}} \bigg\|_{\X}
\lesssim \prod_{j=1}^m \|f_j^1\|_{\X_j}
\lesssim \bigg\| \bigg(\sum_{Q \in \D'} a_Q^{r_2} \mathbf{1}_Q \bigg)^{\frac{1}{r_2}} \bigg\|_{\X}.
\end{align*}
Repeating this procedure, we eventually obtain a sequence $\{(f_1^k, \ldots, f_m^k)\}_{k=1}^{k_0+1}$ such that
\begin{align*}
\bigg(\sum_{Q \in \D'} |a_Q^{(k)}|^{r_2} \mathbf{1}_Q \bigg)^{\frac{1}{r_2}} 
\le \prod_{j=1}^m f_j^{k+1}
\quad \text{ and } \quad 
\prod_{j=1}^m \|f_j^{k+1}\|_{\X_j} 
\le 2 \bigg\|\bigg(\sum_{Q \in \D'} |a_Q^{(k)}|^{r_2} \mathbf{1}_Q \bigg)^{\frac{1}{r_2}} \bigg\|_{\X},
\end{align*}
for each $k=0, 1, \ldots, k_0$, where $a_Q^{(0)} := a_Q$ and $a_Q^{(k)} = \Lambda_Q(f_1^k, \ldots, f_m^k)$ for $k \ge 1$. By the choice of $a_Q^{(k+1)}$ and the assumption, there holds
\begin{align*}
\bigg\| \bigg(\sum_{Q \in \D'} |a_Q^{(k+1)}|^{r_2} \mathbf{1}_Q \bigg)^{\frac{1}{r_2}}  \bigg\|_{\X}
\lesssim \prod_{j=1}^m f_j^{k+1}
\le 2 \bigg\|\bigg(\sum_{Q \in \D'} |a_Q^{(k)}|^{r_2} \mathbf{1}_Q \bigg)^{\frac{1}{r_2}} \bigg\|_{\X},
\end{align*}
which leads to 
\begin{align}\label{eq:aqk1}
\bigg\| \bigg(\sum_{Q \in \D'} |a_Q^{(k_0)}|^{r_2} \mathbf{1}_Q \bigg)^{\frac{1}{r_2}}  \bigg\|_{\X}
\lesssim \bigg\|\bigg(\sum_{Q \in \D'} |a_Q|^{r_2} \mathbf{1}_Q \bigg)^{\frac{1}{r_2}} \bigg\|_{\X}.
\end{align}
For each $x \in Q \in \D'$, 
\begin{align*}
\sum_{Q \in \D'} |a_Q^{(k)}|^{r_2} \mathbf{1}_Q(x)
\ge \sum_{P \in \D': P \supset Q} |a_Q^{(k)}|^{r_2}, 
\end{align*}
which along with the assumption implies
\begin{align}\label{eq:aaq}
a_Q^{(k+1)} 
= \Lambda_Q(f_1^{k+1}, \ldots, f_m^{k+1})
\gtrsim \essinf_Q \prod_{j=1} |f_j|^{k+1} 
\ge \bigg( \sum_{P \in \D': P \supset Q} |a_Q^{(k)}|^{r_2} \bigg)^{\frac{1}{r_2}}.
\end{align}

Fix $x \in \bigcup_{Q \in \D'} Q$. There are finite cubes $\{Q_1, \ldots, Q_N\} \subset \S'$ such that $x \in Q_N \subsetneq \cdots \subsetneq Q_1$. Then by definition and \eqref{eq:aaq},
\begin{align}\label{eq:aqkk}
\sum_{Q \in \D'} |a_Q^{(k_0)}|^{r_2} \mathbf{1}_Q(x)
= \sum_{j=1}^N |a_{Q_j}^{(k_0)}|^{r_2}
\quad \text{ and } \quad
|a_{Q_j}^{(k)}|^{r_2}
\gtrsim \sum_{i=1}^j |a_{Q_i}^{(k-1)}|^{r_2}.
\end{align}
If $k_0 = 1$, then \eqref{eq:aqkk} gives
\begin{align*}
\sum_{Q \in \D'} |a_Q^{(k_0)}|^{r_2} \mathbf{1}_Q(x)
\gtrsim \sum_{j=1}^N \sum_{i=1}^j |a_{Q_i}^{(0)}|^{r_2}
= \sum_{i=1}^N (N-i+1) |a_{Q_i}|^{r_2}.
\end{align*}
If $k_0 = 2$, we use \eqref{eq:aqkk} twice to deduce
\begin{align*}
& \sum_{Q \in \D'} |a_Q^{(k_0)}|^{r_2} \mathbf{1}_Q(x)
\gtrsim \sum_{j_2=1}^N \sum_{j_1=1}^{j_2} |a_{Q_{j_1}}^{(1)}|^{r_2}
\gtrsim \sum_{j_2=1}^N \sum_{j_1=1}^{j_2} \sum_{i=1}^{j_1} |a_{Q_i}^{(0)}|^{r_2}
\\
& = \sum_{i=1}^N \sum_{j_1=i}^N \sum_{j_2=j_1}^N |a_{Q_i}|^{r_2}
= \sum_{i=1}^N \frac12 (N-i+1) (N-i+2) |a_{Q_i}|^{r_2}.
\end{align*}
We eventually obtain 
\begin{align*}
\sum_{Q \in \D'} |a_Q^{(k_0)}|^{r_2} \mathbf{1}_Q(x)
\gtrsim \sum_{i=1}^N \binom{N-i+k_0}{k_0} |a_{Q_i}|^{r_2}
\gtrsim \sum_{i=1}^N (N-i+1)^{k_0} |a_{Q_i}|^{r_2},
\end{align*}
where the implicit constants ate independent of $N$. Thus, by H\"{o}lder's inequality,
\begin{align*}
\bigg[ \sum_{Q \in \D'} |a_Q|^{r_1} \mathbf{1}_Q \bigg]^{\frac{1}{r_1}}
& \le \bigg[\sum_{i=1}^N (N-i+1)^{- \frac{k_0 r_1}{r_2 - r_1}} \bigg]^{\frac{1}{r_1} - \frac{1}{r_2}}
\bigg[\sum_{i=1}^N (N-i+1)^{k_0} |a_{Q_i}|^{r_2} \mathbf{1}_Q \bigg]^{\frac{1}{r_2}}
\\
& \lesssim  \bigg[\sum_{k=1}^{\infty} k^{- \frac{k_0 r_1}{r_2 - r_1}} \bigg]^{\frac{1}{r_1} - \frac{1}{r_2}}
\bigg[ \sum_{Q \in \D'} |a_Q^{(k_0)}|^{r_2} \mathbf{1}_Q \bigg]^{\frac{1}{r_1}}
\lesssim \bigg[ \sum_{Q \in \D'} |a_Q^{(k_0)}|^{r_2} \mathbf{1}_Q \bigg]^{\frac{1}{r_1}},
\end{align*}
provided that $k_0 > r_2/r_1 -1$. Combining this with \eqref{eq:aqk1}, we complete the proof. 
\end{proof}

Next, we establish vector-valued inequalities for multilinear operators on Banach function spaces. 

\begin{lemma}\label{lem:Xvec}
Let $T$ be an $m$-linear operator and $A_0 := \|T\|_{\X_1 \times \cdots \times \X_m \to \X} < \infty$, where $\X_1, \ldots, \X_m$ are Banach function spaces and $\X = \prod_{j=1}^m \X_j$. Then the following hold:
\begin{align}
\label{eq:xvec-1}
& \bigg\| \sum_{Q \in \D} \mathbf{1}_{F_Q} |T(f_1, \ldots, f_j \mathbf{1}_{G_Q}, \ldots, f_m)| \bigg\|_{\X}
\lesssim A_0 \prod_{j=1}^m \|f_j\|_{\X_j}, 
\\
\label{eq:xvec-2}
& \bigg\| \bigg(\sum_{Q \in \D} |T(h_{1, Q}, \ldots, h_{m, Q})|^2 \bigg)^{\frac12} \bigg\|_{\X}
\lesssim A_0 \prod_{j=1}^m \bigg\| \sum_{Q \in \D} |h_{j, Q}| \bigg\|_{\X_j},
\end{align}
where $\D$ is a collection of some dyadic cubes, $\{F_Q\}$, $\{G_Q\}$, and $\{\supp h_{j, Q}\}$ are disjoint families.
\end{lemma}

\begin{proof}
By the Fatou property, we may assume that $\D$ is finite. Let us recall the multiple Khintchine inequality (cf. \cite{NS}):
\begin{align}\label{eq:Khin}
\bigg(\sum_{j_1, \ldots, j_n} |a_{j_1, \ldots, j_n}|^2 \bigg)^{\frac{s}{2}}
\lesssim \mathbb{E} 
\bigg|\sum_{j_1, \ldots, j_n} a_{j_1, \ldots, j_n} \prod_{i=1}^n \varepsilon_{i, j_i} \bigg|^s, \quad s>0, 
\end{align}
where $\varepsilon_{i, j}$ are independent and take the values $1$ and $-1$ with probability $\frac12$. Then by the disjointness of $\{F_Q\}$, \eqref{eq:Khin}, and the finiteness of $\D$,
\begin{align*}
& \bigg\| \sum_{Q \in \D} \mathbf{1}_{F_Q} 
|T(f_1, \ldots, f_j \mathbf{1}_{G_Q}, \ldots, f_m)| \bigg\|_{\X}^{\frac1m}
\\
& = \bigg\| \bigg[\sum_{Q \in \D} \mathbf{1}_{F_Q} 
|T(f_1, \ldots, f_j \mathbf{1}_{G_Q}, \ldots, f_m)| \bigg]^{\frac1m} \bigg\|_{\Y}
\\
& \le \bigg\| \bigg[\sum_{Q \in \D}
|T(f_1, \ldots, f_j \mathbf{1}_{G_Q}, \ldots, f_m)|^2 \bigg]^{\frac{1}{2m}} \bigg\|_{\Y}
\\
& \lesssim \bigg\| \mathbb{E} \bigg|\sum_{Q \in \D}
T(f_1, \ldots, \varepsilon_Q f_j \mathbf{1}_{G_Q}, \ldots, f_m) \bigg|^{\frac1m} \bigg\|_{\Y}
\\
& \le \mathbb{E} \bigg\| \bigg|\sum_{Q \in \D}
T(f_1, \ldots, \varepsilon_Q f_j \mathbf{1}_{G_Q}, \ldots, f_m) \bigg|^{\frac1m} \bigg\|_{\Y}
\\
& = \mathbb{E} \bigg\| T \bigg(f_1, \ldots, 
\sum_{Q \in \D} \varepsilon_Q f_j \mathbf{1}_{G_Q}, \ldots, f_m \bigg) \bigg\|_{\X}^{\frac1m}
\\
& \lesssim \bigg( A_0 \prod_{j=1}^m \|f_j\|_{\X_j} \bigg)^{\frac1m},
\end{align*}
provided the disjointness of $\{G_Q\}$ and $\|\sum_{Q \in \D} \varepsilon_Q f_j \mathbf{1}_{G_Q}\|_{\X_j} \le \|f_j\|_{\X_j}$. Similarly, one has
\begin{align*}
& \bigg\| \bigg(\sum_{Q \in \D} |T(h_{1, Q}, \ldots, h_{m, Q})|^2 \bigg)^{\frac12} \bigg\|_{\X}^{\frac1m}
\\
& \le \bigg\| \bigg(\sum_{Q_1, \ldots, Q_m \in \D} 
|T(h_{1, Q_1}, \ldots, h_{m, Q_m})|^2 \bigg)^{\frac12} \bigg\|_{\X}^{\frac1m}
\\
& = \bigg\| \bigg(\sum_{Q_1, \ldots, Q_m \in \D} 
|T(h_{1, Q_1}, \ldots, h_{m, Q_m})|^2 \bigg)^{\frac{1}{2m}} \bigg\|_{\Y}
\\
& \lesssim \bigg\| \mathbb{E} \bigg| T \bigg(\sum_{Q_1 \in \D} \varepsilon_{1, Q_1} h_{1, Q_1}, \ldots, 
\sum_{Q_m \in \D} \varepsilon_{1, Q_m} h_{m, Q_m} \bigg) \bigg|^{\frac1m} \bigg\|_{\Y}
\\
& \le \mathbb{E} \bigg\| \bigg| T \bigg(\sum_{Q_1 \in \D} \varepsilon_{1, Q_1} h_{1, Q_1}, \ldots, 
\sum_{Q_m \in \D} \varepsilon_{1, Q_m} h_{m, Q_m} \bigg) \bigg|^{\frac1m} \bigg\|_{\Y}
\\
& = \mathbb{E} \bigg\| T \bigg(\sum_{Q_1 \in \D} \varepsilon_{1, Q_1} h_{1, Q_1}, \ldots, 
\sum_{Q_m \in \D} \varepsilon_{1, Q_m} h_{m, Q_m} \bigg) \bigg\|_{\X}^{\frac1m}
\\
& \lesssim \bigg( A_0 \prod_{j=1}^m \bigg\| \sum_{Q \in \D} |h_{j, Q}| \bigg\|_{\X_j} \bigg)^{\frac1m},
\end{align*}
where we have used the disjoint of $\{\supp h_{j, Q}\}_{Q \in \D}$ and 
\begin{align*}
\bigg\| \sum_{Q \in \D} \varepsilon_{j, Q} h_{j, Q} \bigg\|_{\X_j}
= \bigg\| \bigg|\sum_{Q \in \D} \varepsilon_{j, Q} h_{j, Q} \bigg| \bigg\|_{\X_j}
= \bigg\| \sum_{Q \in \D} |h_{j, Q}| \bigg\|_{\X_j}.
\end{align*}
This completes the proof. 
\end{proof}

Based on Lemmas \ref{lem:imp} and \ref{lem:Xvec}, we obtain an estimate for multilinear sparse operators on Banach function spaces. Recall the definition of multilinear Riesz transform 
\begin{align*}
\cR_j^k(\vec{f})(x)
:= \mathrm{p.v. } \int_{\R^{md}} 
\frac{(x - y_j)_k}{\big(\sum_{i=1}^m |x - y_i|^2 \big)^{\frac{md+1}{2}}} \, 
\prod_{i=1}^m f_i(y_i) \, d\vec{y},
\end{align*}
for each $1 \le j \le m$ and $1 \le k \le d$, where $(y_j)_k$ denotes the $k$-th coordinate of $y_j$.

\begin{lemma}\label{lem:Riesz}
Suppose that $\sup_{k, j} \|\cR_j^k\|_{\X_1 \times \cdots \times \X_m \to \X} < \infty$, where $\X_1, \ldots, \X_m$ are Banach function spaces and $\X = \prod_{j=1}^m \X_j$. Then 
\begin{align*}
\sup_{\S: \gamma \text{-sparse}} 
\bigg\| \sum_{Q \in \S} a_Q \mathbf{1}_Q \bigg\|_{\X}
\lesssim \prod_{j=1}^m \|f_j\|_{\X_j}.
\end{align*}
where $a_Q := \langle f_j \mathbf{1}_{G_Q} \rangle_{Q, 1} \prod_{i \ne j} \langle f_i \rangle_{Q, 1}$ and $\{G_Q\}_{Q \in \S}$ is a disjoint family with $G_Q \subset Q$.
\end{lemma}

\begin{proof}
Fix a $\gamma$-sparse family $\S$. Then there exists a disjoint family $\{E_Q\}_{Q \in \S}$ such that $E_Q \subset Q$ with $|E_Q| \ge \gamma |Q|$. Let $P_Q := Q + 4 \ell(Q) e_1$, where $e_1 = (1, 0, \ldots, 0)$. The $\gamma$-sparseness of $\S$ implies the family $\{P_Q: Q \in \S\}$ is $\gamma$-sparse as well, where $E_{P_Q} := E_Q + 4 \ell(Q) e_1$. Note that for all $x \in P_Q$ and $y_j \in G_Q \subset Q$, there holds $(x - y_j)_1 \ge 3 \ell(Q)$ and 
\begin{align*}
\cR_j^1(f_1 \mathbf{1}_Q, \ldots, f_j \mathbf{1}_{G_Q}, \ldots, f_m \mathbf{1}_Q)(x)
\gtrsim a_Q, \qquad x \in P_Q,
\end{align*}
which along with \eqref{eq:xvec-1} gives
\begin{align*}
\bigg\|\sum_{Q \in \S} a_Q \mathbf{1}_{E_{P_Q}} \bigg\|_{\X} 
\lesssim \bigg\|\sum_{Q \in \S'} \mathbf{1}_{E_{P_Q}}
\cR_j^1(f_1 \mathbf{1}_Q, \ldots, f_j \mathbf{1}_{G_Q}, \ldots, f_m \mathbf{1}_Q) \bigg\|_{\X}
\lesssim \prod_{j=1}^m \|f_j\|_{\X_j}. 
\end{align*}
By definition, there exist nonnegative functions $g_j \in \X_j$ such that 
\begin{align}\label{eq:epq}
\sum_{Q \in \S} a_Q \mathbf{1}_{E_{P_Q}} 
\le \prod_{j=1}^m g_j
\quad \text{ and } \quad
\prod_{j=1}^m \|g_j\|_{\X_j}
\le 2 \bigg\|\sum_{Q \in \S} a_Q \mathbf{1}_{E_{P_Q}} \bigg\|_{\X}
\lesssim \prod_{j=1}^m \|f_j\|_{\X_j}.
\end{align}
On the other hand, for all $x \in Q$ and $y_j \in E_{P_Q}$, there holds $(x - y_j)_1 < -3\ell(Q)$ and
\begin{align}\label{eq:crg}
|\cR_j^1(g_1 \mathbf{1}_{E_{P_Q}}, \ldots, g_m \mathbf{1}_{E_{P_Q}})(x)|
\gtrsim \prod_{j=1}^m \frac{|E_{P_Q}|}{|Q|} \langle g_j \rangle_{E_{P_Q}, 1}
\gtrsim \bigg\langle \prod_{j=1}^m g_j^{\frac1m} \bigg\rangle_{E_{P_Q}, 1}^m
\gtrsim a_Q,
\end{align}
provided \eqref{eq:epq} and the disjointness of $\{E_{P_Q}: Q \in \S\}$. Thus, by \eqref{eq:xvec-2} and  \eqref{eq:crg},
\begin{align}\label{eq:qqji}
\bigg\| \bigg(\sum_{Q \in \S} a_Q^2 \mathbf{1}_Q \bigg)^{\frac12} \bigg\|_{\X}
\lesssim \prod_{j=1}^m \|g_j\|_{\X_j}
\lesssim \prod_{j=1}^m \|f_j\|_{\X_j}.
\end{align}
By H\"{o}lder's inequality, the quantity $\Lambda_Q(\vec{f}) := \langle f_j \mathbf{1}_{E_Q} \rangle_{Q, 1} \prod_{i \ne j} \langle f_i \rangle_{Q, 1}$ satisfies 
\begin{align}\label{eq:qqjj}
\Lambda_Q(\vec{f}) 
\ge \prod_{j=1}^m \frac{|E_Q|}{|Q|} \langle f_j \rangle_{E_Q, 1}
\gtrsim \bigg\langle \prod_{j=1}^m |f_j|^{\frac1m} \bigg \rangle_{E_Q, 1}^m
\ge \essinf_Q \prod_{j=1}^m |f_j|.
\end{align}
Therefore, the desired estimate follows from \eqref{eq:qqji}--\eqref{eq:qqjj}, and Lemma \ref{lem:imp} applied to $r_1 = 1 < 2 = r_2$.
\end{proof}

\subsection{Proof of Theorem \ref{thm:MM}} 
It is trivial that
\begin{align*}
& \sup\limits_{\S: \gamma \text{-sparse}} 
\|\A_{\S, m\vec{r}, (ms)'}(\vec{f}, f_{m+1})\|_{L^1}
\\
& \le \Lambda^*_{m\vec{r}, (ms)'} (\vec{f}, f_{m+1})
\le \gamma^{-1} \|\M_{m\vec{r}, (ms)'} (\vec{f}, f_{m+1})\|_{L^1}.
\end{align*}
It suffices to obtain the converse. By Lemma \ref{lem:dyadic}, one can find $3^d$ dyadic lattices $\{\mathscr{D}^k\}_{k=1}^{3^d}$ such that for every cube $Q \subset \Rd$, there exists some $k \in \{1, \ldots, 3^d\}$ and $P_Q \in \mathscr{D}^k$ such that $Q \subset P_Q$ and $\ell(P_Q) \le 3 \ell(Q)$. This immediately implies 
\begin{align}\label{eq:mmrr}
\M_{m\vec{r}, (ms)'} (\vec{f}, f_{m+1})
\lesssim \sum_{k=1}^{3^d} \A_{\mathscr{D}^k, m\vec{r}, (ms)'}(\vec{f}, f_{m+1}).
\end{align}
Let $\mathscr{D}^k_0$ be a finite subset $\mathscr{D}^k$. Using the stopping time argument as in \eqref{eq:stopcon}--\eqref{eq:rjrq}, we obtain a $\gamma$-sparse family $\S_k \subset \mathscr{D}^k$ such that for every $Q \in \mathscr{D}^k_0$ there exists $R_Q \in \S_k$ such that $Q \subset R_Q$ and 
\begin{equation*}
\begin{aligned}
& \langle f_j \rangle_{mr_j, Q}
\le A_0 \langle f_j \rangle_{mr_j, R_Q}, \quad j=1, \ldots, m,
\\
& \text{and} \quad \langle f_{m+1} \rangle_{(ms)', Q}
\le A_0 \langle f_{m+1} \rangle_{(ms)', R_Q},
\end{aligned}
\end{equation*}
which gives
\begin{equation}\label{eq:fmra}
\A_{\mathscr{D}^k_0, m\vec{r}, (ms)'}(\vec{f}, f_{m+1})
\lesssim \A_{\S_k, m\vec{r}, (ms)'}(\vec{f}, f_{m+1}).
\end{equation}
Then it follows from \eqref{eq:mmrr}--\eqref{eq:fmra} and the monotone convergence theorem that 
\begin{align*}
\|\M_{m\vec{r}, (ms)'} (\vec{f}, f_{m+1})\|_{L^1}
\lesssim \sup\limits_{\S: \gamma \text{-sparse}} 
\|\A_{\S, m\vec{r}, (ms)'}(\vec{f}, f_{m+1})\|_{L^1}.
\end{align*}
This shows the first part of Theorem \ref{thm:MM}.

Let us verify the second part of Theorem \ref{thm:MM}. By the pointwise sparse domination \cite[Theorem 3.6]{DHL}, there exist $\gamma$-sparse families $\{\S_k\}_{k=1}^{3^d}$ such that
\begin{align*}
|T(\vec{f})(x)|
\lesssim \sum_{k=1}^{3^d} \A_{\S_k, \vec{1}}(\vec{f})(x), \quad \text{a.e. } x \in \Rd.
\end{align*}
Let $0 \le f_{m+1} \in \Y'$. The above implies
\begin{align*}
\langle |T(\vec{f})|^{\frac1m}, f_{m+1} \rangle
& \lesssim \sum_{k=1}^{3^d} \sum_{Q \in \S_k} |Q|
\prod_{j=1}^m \langle |f_j|^{\frac1m} \rangle_{Q, m}
\langle f_{m+1} \rangle_{Q, 1}
\\
& \lesssim \Lambda^*_{m, \ldots, m, 1}(|f_1|^{\frac1m}, \ldots, |f_m|^{\frac1m}, f_{m+1}),
\end{align*}
which together with $\|T(\vec{f})\|_{\X}^{\frac1m} = \||T(\vec{f})|^{\frac1m}\|_{\Y}$ gives
\begin{align*}
\|T\|_{\X_1 \times \cdots \times \X_m \to \X}
\lesssim \|\Lambda^*_{m, \ldots, m, 1}\|_{\X_1^{\frac1m} \times \cdots \X_m^{\frac1m} \times \Y' \to \R}.
\end{align*}
Note that every multilinear Riesz transform is a multilinear Calder\'{o}n--Zygmund operator, and \cite[Theorem 1.2]{Nie26} showed the equivalence between the boundedness of Riesz transform and that of Calder\'{o}n--Zygmund operator in the linear case. Thus, to conclude the proof, it suffices to prove
\begin{align}\label{eq:as11}
\sup_{\S'} \|\A_{\S', \vec{1}} (\vec{f})\|_{\X}
\lesssim \prod_{j=1}^m \|f_j\|_{\X_j}, \quad m \ge 2,
\end{align}
where the supremum is taken over all finite $\gamma$-sparse families $\S'$. Note that
\begin{align}\label{eq:ambq}
\Lambda_Q(\vec{f}) 
:= \prod_{j=1}^m \langle f_j \rangle_{Q, 1}
\ge \bigg\langle \prod_{j=1}^m |f_j|^{\frac1m} \bigg \rangle_{Q, 1}^m
\ge \essinf_Q \prod_{j=1}^m |f_j|.
\end{align}
In view of \eqref{eq:as11} and \eqref{eq:ambq}, Lemma \ref{lem:imp} applied to $a_Q = \Lambda_Q(\vec{f})$ and $r_1 := 1/m <1 =: r_2$ yields
\begin{align*}
\|\A_{\S', \vec{m}} (\vec{f})\|_{\Y}
& = \bigg\| \bigg(\sum_{Q \in \S'} \prod_{j=1}^m 
\langle |f_j|^m \rangle_{Q, 1}^{\frac1m} \mathbf{1}_Q \bigg)^m \bigg\|_{\X}^{\frac1m}
\\
& \lesssim \|\A_{\S', \vec{1}} (|f_1|^m, \ldots, |f_m|^m)\|_{\X}^{\frac1m}
\lesssim \prod_{j=1}^m \| |f_j|^m \|_{\X_j}^{\frac1m}
= \prod_{j=1}^m \| f_j \|_{\X_j^{\frac1m}},
\end{align*}
where the implicit constants are independent of $\S'$. By the Fatou property, we have
\begin{align*}
& \sup_{\S: \gamma{\text-sparse}} \|\A_{\S, \vec{m}} (\vec{f})\|_{\Y}
\lesssim \prod_{j=1}^m \| f_j \|_{\X_j^{\frac1m}},
\end{align*}
which is equivalent to 
\begin{align*}
\sup_{\S: \gamma{\text-sparse}} 
\|\A_{\S, \vec{m}, 1}\|_{\X_1^{\frac1m} \times \cdots \times \X_m^{\frac1m} \times \Y' \to L^1} 
< \infty.
\end{align*}
This proves the second part of Theorem \ref{thm:MM}.

Let us turn to the proof of \eqref{eq:as11}. Given $Q \in \S'$, by Lemma \ref{lem:Riesz} applied to $\S = \{Q\}$, we see that $\mathbf{1}_Q \in \X$. Then the finiteness of $\S'$ gives
\begin{align}\label{eq:fqss}
\A_{\S', \vec{1}} (\vec{f}) 
= \sum_{Q \in \S'} \prod_{j=1}^m \langle f_j \rangle_{Q, 1} \mathbf{1}_Q \in \X.
\end{align}
By definition, there exist nonnegative functions $h_j \in \X_j$ such that 
\begin{align}\label{eq:ash}
\A_{\S', \vec{1}} (\vec{f}) \le \prod_{j=1}^m h_j
\quad \text{ and } \quad 
\prod_{j=1}^m \|h_j\|_{\X_j} \le 2 \|\A_{\S', \vec{1}} (\vec{f})\|_{\X}.
\end{align}
Since $\S'$ is $\gamma$-sparse, there exists a disjoint family $\{E_Q\}_{Q \in \S'}$ such that $E_Q \subset Q$ with $|E_Q| \ge \gamma |Q|$. Then by \eqref{eq:ash} and H\"{o}lder's inequality,
\begin{align*}
\gamma^m \prod_{j=1}^m \langle f_j \rangle_{Q, 1}
& \le \bigg(\frac{1}{|Q|} \int_{E_Q} \prod_{j=1}^m \langle f_j \rangle_{Q, 1}^{\frac1m} dx \bigg)^m
\le \bigg(\frac{1}{|Q|} \int_{E_Q} \A_{\S', \vec{1}} (\vec{f})^{\frac1m} dx \bigg)^m
\\
& \le \bigg(\frac{1}{|Q|} \int_{E_Q} \prod_{j=1}^m h_j^{\frac1m} dx \bigg)^m
\le \prod_{j=1}^m \langle h_j \mathbf{1}_{E_Q} \rangle_{Q, 1},
\end{align*}
which leads to
\begin{align*}
\prod_{j=1}^m \langle f_j \rangle_{Q, 1}
\le \gamma^{-1} \prod_{j=1}^m 
\bigg(\langle h_j \mathbf{1}_{E_Q} \rangle_{Q, 1} 
\prod_{i \ne j} \langle f_j \rangle_{Q, 1} \bigg)^{\frac1m}
\end{align*}
and 
\begin{align}\label{eq:as12}
\|\A_{\S', \vec{1}} (\vec{f})\|_{\X}
& \lesssim \bigg\|\sum_{Q \in \S'} \prod_{j=1}^m 
\bigg(\langle h_j \mathbf{1}_{E_Q} \rangle_{Q, 1} 
\prod_{i \ne j} \langle f_i \rangle_{Q, 1} 
\mathbf{1}_Q \bigg)^{\frac1m} \bigg\|_{\X}
\\ \nonumber
& \le \bigg\|\prod_{j=1}^m \bigg( \sum_{Q \in \S'} 
\langle h_j \mathbf{1}_{E_Q} \rangle_{Q, 1} 
\prod_{i \ne j} \langle f_i \rangle_{Q, 1} 
\mathbf{1}_Q \bigg)^{\frac1m} \bigg\|_{\X}
\\ \nonumber
& \le \prod_{j=1}^m \bigg\|\sum_{Q \in \S'} 
\langle h_j \mathbf{1}_{E_Q} \rangle_{Q, 1} 
\prod_{i \ne j} \langle f_i \rangle_{Q, 1} \mathbf{1}_Q\bigg\|_{\X}^{\frac1m} 
\\ \nonumber
& \lesssim \prod_{j=1}^m \bigg( \|h_j\|_{\X_j}  \prod_{i \ne j} \|f_i\|_{\X_i} \bigg)^{\frac1m}
\\ \nonumber
& = \bigg( \prod_{j=1}^m \|h_j\|_{\X_j} \bigg)^{\frac1m}
\bigg(\prod_{j=1}^m \|f_j\|_{\X_j} \bigg)^{1 - \frac1m}
\\ \nonumber
& \lesssim \|\A_{\S', \vec{1}} (\vec{f})\|_{\X}^{\frac1m} 
\bigg(\prod_{j=1}^m \|f_j\|_{\X_j} \bigg)^{1 - \frac1m},
\end{align}
where we have used  H\"{o}lder's inequality, Lemmas \ref{lem:vv} and \ref{lem:Riesz}, and the estimate \eqref{eq:ash}. Consequently, \eqref{eq:as11} follows from \eqref{eq:fqss} and \eqref{eq:as12}.

\subsection{Proof of Theorem \ref{thm:example}}\label{sec:exa}
We are now ready to show Theorem \ref{thm:example}.
We omit the case of weighted Lebesgue spaces because $L^{p_j, q_j}_{w_j} = L^{p_j}_{w_j}$ when $p_j=q_j$.

\subsubsection{{\bf Weighted Lorentz spaces}}
Let $\X_j = L^{p_j, q_j}_{w_j}$ with $r_j <p_j, q_j < \infty$ and $\frac1s < \frac1p = \sum_{j=1}^m \frac{1}{p_j}$. Then $\X_j^{\frac1m} = L^{mp_j, mq_j}(w_j^{p_j})$. Denote $w = \prod_{j=1}^m w_j$, 
\begin{align*}
\sigma := w^{\frac{1}{\frac1p - \frac1s}}, \quad
\sigma_j := w_j^{-\frac{1}{\frac{1}{r_j} - \frac{1}{p_j}}}, \quad \text{and} \quad
v := \sigma^{\frac{1}{(ms)'}} \prod_{j=1}^m \sigma_j^{\frac{1}{mr_j}}
= \sigma^{\frac{1}{(mp)'}} \prod_{j=1}^m \sigma_j^{\frac{1}{mp_j}}.
\end{align*}

\noindent $\bullet$ {\bf The sufficiency of $A_{\vec{p}, (\vec{r}, s)}$}. Now assume that $\vec{w} \in A_{\vec{p}, (\vec{r}, s)}$. By Theorem \ref{thm:MM}, it suffices to show
\begin{align}\label{eq:Lam}
\|\Lambda_{m \vec{r}, (ms)'}^*\|_{\X_1^{\frac1m} \times \cdots \times \X_m^{\frac1m} \times \Y' \to \R} < \infty.
\end{align}
Lemma \ref{lem:Aprs} gives
\begin{align}\label{eq:ssai}
\sigma, \sigma_j \in A_{\infty}, \quad j=1, \ldots, m.
\end{align}
By H\"{o}lder's inequality, \eqref{eq:ssai}, the definition of $A_{\infty}$, and Jensen's inequality, we have
\begin{align*}
\fint_Q v \, dx 
& \le \bigg( \fint_Q \sigma \, dx \bigg)^{\frac{1}{(mp)'}}
\prod_{j=1}^m \bigg( \fint_Q \sigma_j \, dx \bigg)^{\frac{1}{mp_j}}
\\
& \lesssim \exp \bigg(\frac{1}{(mp)'} \fint_Q \log \sigma \, dx \bigg)
\prod_{j=1}^m \exp \bigg(\frac{1}{mp_j} \fint_Q \log \sigma_j \, dx \bigg)
\\
& = \exp \bigg(\fint_Q \log v \, dx \bigg)
\le \bigg(\fint_Q v \, dx \bigg),
\end{align*}
which means 
\begin{align}\label{eq:vai-1}
v \in A_{\infty} \quad \text{ and } \quad 
\langle \sigma \rangle_Q^{\frac{1}{(mp)'}}
\prod_{j=1}^m \langle \sigma_j \rangle_Q^{\frac{1}{mp_j}}
\simeq \frac{v(Q)}{|Q|}.
\end{align}
Given a $\gamma$-sparse family $\S$, there exists a disjoint collection $\{E_Q\}_{Q \in \S}$ such that $E_Q \subset Q$ with $|E_Q| \ge \gamma |Q|$. Then by \eqref{eq:vai-1} and the property of $A_{\infty}$ weights, 
\begin{align}\label{eq:vai-2}
v(Q) \lesssim v(E_Q). 
\end{align}
Setting $g_j := f_j \sigma_j^{-\frac{1}{mr_j}}$ and $g_{m+1} = f_{m+1} \sigma^{-\frac{1}{(ms)'}}$, we use \eqref{eq:vai-1} and \eqref{eq:vai-2} to arrive at
\begin{align*}
\Lambda_{\S}
& := \sum_{Q \in \S} |Q| \prod_{j=1}^m \langle f_j \rangle_{Q, mr_j}
\langle f_{m+1} \rangle_{Q, {ms}'}
\\
& = \sum_{Q \in \S} |Q| \langle \sigma \rangle_Q^{\frac{1}{(mp)'}}
\prod_{j=1}^m \langle \sigma_j \rangle_Q^{\frac{1}{mp_j}}
\times \bigg[\langle \sigma \rangle_Q^{\frac1p - \frac1s}
\prod_{j=1}^m \langle \sigma_j \rangle_Q^{\frac{1}{r_j} - \frac{1}{p_j}} \bigg]^{\frac1m}
\\
&\qquad \times \prod_{j=1}^m \bigg(\fint_Q |g_j|^{mr_j} \, d\sigma_j \bigg)^{\frac{1}{mr_j}}
\bigg(\fint_Q |g_{m+1}|^{(ms)'} \, d\sigma \bigg)^{\frac{1}{(ms)'}}
\\
& \lesssim [\vec{w}]_{A_{\vec{p}, (\vec{r}, s)}}^{\frac1m} 
\sum_{Q \in \S} v(Q) \prod_{j=1}^m 
\inf_Q M_{mr_j}^{\sigma_j} g_j \, 
\inf_Q M_{(ms)'}^{\sigma} g_{m+1}
\\
& \lesssim [\vec{w}]_{A_{\vec{p}, (\vec{r}, s)}}^{\frac1m}
\sum_{Q \in \S} \int_{E_Q} \prod_{j=1}^m 
M_{mr_j}^{\sigma_j} g_j \,
M_{(ms)'}^{\sigma} g_{m+1} \, v(x) \, dx
\\
& \lesssim [\vec{w}]_{A_{\vec{p}, (\vec{r}, s)}}^{\frac1m}
\int_{\Rd} \prod_{j=1}^m 
M_{mr_j}^{\sigma_j} g_j \,
M_{(ms)'}^{\sigma} g_{m+1} \, v(x) \, dx
\\
& \le [\vec{w}]_{A_{\vec{p}, (\vec{r}, s)}}^{\frac1m}
\prod_{j=1}^m \big\|\sigma_j^{\frac{1}{mr_j}} 
M_{mr_j}^{\sigma_j} g_j \big\|_{\X_j^{\frac1m}} 
\big\|\sigma^{\frac{1}{(ms)'}} M_{(ms)'}^{\sigma} g_{m+1} \big\|_{\Y'}.
\end{align*}
It follows from \eqref{eq:ssai}, \eqref{eq:AiBMO}, and Lemma \ref{lem:Lorentz} that
\begin{align*}
\big\|\sigma_j^{\frac{1}{mr_j}} M_{mr_j}^{\sigma_j} g_j \big\|_{\X_j^{\frac1m}} 
& = \big\|\sigma_j^{\frac{1}{mr_j}} M_{mr_j}^{\sigma_j} g_j \big\|_{L^{mp_j, mq_j}(w_j^{p_j})} 
= \big\|M_{mr_j}^{\sigma_j} g_j \big\|_{\bL^{mp_j, mq_j}(\sigma_j^{-{p_j}/{r_j}}; \sigma_j)}
\\
& \lesssim \big\|g_j \big\|_{\bL^{mp_j, mq_j}(\sigma_j^{-{p_j}/{r_j}}; \sigma_j)}
= \|f_j\|_{L^{pm_j, mq_j}(w_j^{p_j})}
= \|f_j\|_{\X_j^{\frac1m}}.
\end{align*}
Thus, to achieve \eqref{eq:Lam}, it is enough to prove
\begin{align}\label{eq:sgf2}
\big\|\sigma^{\frac{1}{(ms)'}} M_{(ms)'}^{\sigma} g_{m+1} \big\|_{\Y'}
\lesssim \|f_{m+1}\|_{\Y'}.
\end{align}
Define
\begin{align*}
\Z := \prod_{j=1}^m \Z_j 
\quad \text{and } \quad
\Z_j := \bL^{p_j, q_j}(\sigma^{-1} w_j^{p_j}; \sigma) 
= (\sigma^{-1} w_j^{p_j})^{\frac{1}{p_j}} \X_j.
\end{align*}
Let $\Z'_j$ be the associated space of $\Z_j$ with respect to the measure $\sigma$. Then 
\begin{align}\label{eq:zz}
\Z = \sigma^{-\frac{1}{ms}} \Y
\quad \text{and } \quad
\Z' = \prod_{j=1}^m (\Z'_j)^{\frac1m}.
\end{align}
Pick $\frac{1}{t_j} := 1 - \frac{p}{sp_j}$ so that $1 \le t_j \le p'_j$ and $\sum_{j=1}^m \frac{1}{mt_j} = \frac{1}{(ms)'}$. By \eqref{eq:zz}, for every $f \in \Z'$, there exists nonnegative functions $f_j \in (\Z'_j)^{\frac1m}$ so that $|f| \le \prod_{j=1}^m |f_j|$ and $\prod_{j=1}^m \|f_j\|_{(\Z'_j)^{\frac1m}} \le 2 \|f\|_{\Z'}$. Thus, by H\"{o}lder's inequality and Lemma \ref{lem:Lorentz},
\begin{align*}
\|M_{(ms)'}^{\sigma} f\|_{\Z'}
& \le \bigg\|\prod_{j=1}^m M_{m t_j}^{\sigma} f_j \bigg\|_{\Z'}
\le \prod_{j=1}^m \|M_{m t_j}^{\sigma} f_j \|_{(\Z'_j)^{\frac1m}}
\\
& = \prod_{j=1}^m \|M_{t_j}^{\sigma} (|f_j|^m) \|_{\Z'_j}^{\frac1m}
\lesssim \prod_{j=1}^m \| |f_j|^m \|_{\Z'_j}^{\frac1m}
= \prod_{j=1}^m \|f_j\|_{(\Z'_j)^{\frac1m}}
\lesssim \|f\|_{\Z'},
\end{align*}
which implies 
\begin{align*}
\big\|\sigma^{\frac{1}{(ms)'}} M_{(ms)'}^{\sigma} g_{m+1} \big\|_{\Y'}
= \|M_{(ms)'}^{\sigma} g_{m+1}\|_{\Z'}
\lesssim \|g_{m+1}\|_{\Z'}
= \|f_{m+1}\|_{\Y'}.
\end{align*}
This verifies \eqref{eq:sgf2}.

\noindent $\bullet$ {\bf The necessity of $A_{\vec{p}, (\vec{r}, s)}$}. Now assume that $\mathcal{N}_{\vec{\X}} < \infty$. Then by Lemma \ref{lem:cc} below,
\begin{align}\label{eq:ccm-1}
\sigma_j \in A_{\infty} \quad \text{ and } \quad
\sigma_j(Q)^{\frac{1}{\s_j} - \frac{1}{\p_j}} 
\lesssim \sup_{\|f\|_{\Y_j} \le 1} \|f \mathbf{1}_Q\|_{L^{\s_j}}, \qquad j=1, \ldots, m.
\end{align}
It suffices to show
\begin{align}\label{eq:ccm-2}
\sigma(Q)^{\frac{1}{mp} - \frac{1}{ms}} 
\lesssim |Q|^{\frac{1}{\s_{m+1}} - \frac{1}{\s_{m+1}^*}}
\sup_{\|f\|_{\Y'} \le 1} \|f \mathbf{1}_Q\|_{L^{\s_{m+1}^*}},
\end{align}
for some $\s_{m+1}^* > \s_{m+1}$. Indeed, a combination of \eqref{eq:ccm-1} with \eqref{eq:ccm-2} yields 
\begin{align*}
\prod_{j=1}^{m+1} \langle \sigma_j \rangle_Q^{\frac{1}{\s_j} - \frac{1}{\p_j}} 
& \lesssim  |Q|^{1 - \sum_{j=1}^m \frac{1}{\s_j} - \frac{1}{\s_{m+1}^*}} 
\prod_{j=1}^m \sup_{\|f\|_{\Y_j} \le 1} \|f \mathbf{1}_Q\|_{L^{\s_j}}
\sup_{\|f\|_{\Y'} \le 1} \|f \mathbf{1}_Q\|_{L^{\s_{m+1}^*}}
\\
& \le \|\M_{\vec{\s}, \s_{m+1}^*}\|_{\Y_1 \times \cdots \times \Y_m \times \Y' \to L^1}
\lesssim 1,
\end{align*}
provided Lemma \ref{lem:Mimp}. This shows $\vec{w} \in A_{\vec{p}, (\vec{r}, s)}$. 

Let us next verify \eqref{eq:ccm-2}. In light of Lemma \ref{lem:Mimp}, choose $\s_{m+1}^* \in (\s_{m+1}, \s_{m+1} + \varepsilon_0)$ so that 
\begin{align*}
0 < \frac{1}{\s_{m+1}} - \frac{1}{\s_{m+1}^*} 
< \min \bigg\{\frac{1}{(mp)'}, \, \sum_{j=1}^{m+1} \frac{1}{\s_j} -1 \bigg\},
\end{align*}
and then set $\frac{1}{\s} := \frac{1}{mp} + \frac{1}{\s_{m+1}} - \frac{1}{\s_{m+1}^*}$, which satisfies  
\begin{align*}
1 < \s < mp \quad \text{ and } \quad
\frac{1}{\s_{m+1}^*} = \frac{1}{\s'} + \bigg(\frac{1}{mp} - \frac{1}{ms} \bigg).
\end{align*}
We claim that
\begin{align}\label{eq:qys}
\|h w^{\frac1m} \mathbf{1}_Q\|_{\Y'} 
\lesssim |Q|^{\frac{1}{\s_{m+1}} - \frac{1}{\s_{m+1}^*}} 
\|h \mathbf{1}_Q\|_{L^{\s'}}. 
\end{align}
For each $h \in \Y$, there exists nonnegative functions $h_j \in \Y_j$ such that $|h| \le \prod_{j=1}^m h_j$ and $\prod_{j=1}^m \|h_j\|_{\Y_j} \le 2 \|h\|_{\Y}$. Set $t_j := \s p_j/p < mp_j$ so that $\frac{1}{\s} = \sum_{j=1}^m \frac{1}{t_j}$. By H\"{o}lder's inequality, \eqref{eq:ccm-1}, and \eqref{eq:uupq-3}, 
\begin{align*}
\|h w^{\frac1m} \mathbf{1}_Q\|_{L^{\s}}
& \le \bigg\| \prod_{j=1}^m h_j w_j^{\frac1m} \mathbf{1}_Q \bigg\|_{L^{\s}}
\le \prod_{j=1}^m \| h_j w_j^{\frac1m} \mathbf{1}_Q \|_{L^{t_j}}
\\
& \lesssim \prod_{j=1}^m |Q|^{\frac{1}{t_j} - \frac{1}{mp_j}} 
\| h_j\|_{L^{m p_j, mq_j}(w_j^{p_j})}
\\
& = |Q|^{\frac{1}{\s_{m+1}} - \frac{1}{\s_{m+1}^*}} \prod_{j=1}^m \|h_j\|_{\Y_j}
\le 2 |Q|^{\frac{1}{\s_{m+1}} - \frac{1}{\s_{m+1}^*}} \|h\|_{\Y},
\end{align*}
which along with duality gives \eqref{eq:qys}.

Then using \eqref{eq:qys}, we obtain
\begin{align*}
& |Q|^{\frac{1}{\s_{m+1}^*} - \frac{1}{\s_{m+1}}} 
\sigma(Q)^{\frac{1}{mp} - \frac{1}{ms}} 
= |Q|^{\frac{1}{\s_{m+1}^*} - \frac{1}{\s_{m+1}}} 
\|w^{\frac1m} \mathbf{1}_Q\|_{L^{\frac{1}{\frac{1}{mp} - \frac{1}{ms}}}} 
\\
& = |Q|^{\frac{1}{\s_{m+1}^*} - \frac{1}{\s_{m+1}}} 
\sup_{\|h \mathbf{1}_Q\|_{L^{\s'} \le 1}} 
\|h w^{\frac1m} \mathbf{1}_Q\|_{L^{\s_{m+1}^*}} 
\lesssim \sup_{\|f\|_{\Y'}\ \le 1} \|f \mathbf{1}_Q\|_{L^{\s_{m+1}^*}},
\end{align*}
which shows \eqref{eq:ccm-2}.

\begin{lemma}\label{lem:cc}
If $\mathcal{N}_{\vec{\X}} < \infty$, then for each $j=1, \ldots, m$, 
\begin{align*}
\sigma_j \in A_{\infty} \quad \text{ and } \quad
\sigma_j(Q)^{\frac{1}{m r_j} - \frac{1}{m p_j}} 
\lesssim \sup_{\|f\|_{\Y_j} \le 1} \|f \mathbf{1}_Q\|_{L^{m r_j}}.
\end{align*}
\end{lemma}

\begin{proof}
Fix $j \in \{1, \ldots, m\}$. 
\begin{align*}
\prod_{j=1}^m \sup_{\|f\|_{\Y_j} \le 1} \|f \mathbf{1}_Q\|_{L^{m r_j}}
\sup_{\|f\|_{\Y'} \le 1} \|f \mathbf{1}_Q\|_{L^{(ms)'}}
\le \mathcal{N}_{\vec{\X}} < \infty. 
\end{align*}
Choose
\begin{align*}
0 < \alpha_j < \min \bigg\{1, r_j \bigg(\frac1r - \frac1s \bigg) \bigg\}
\quad \text{ and } \quad 
\frac{1}{\beta_j} := 1 + \bigg(\frac{1}{mr} - \frac{1}{ms} \bigg) - \frac{\alpha_j}{mr_j} > 1,
\end{align*}
which satisfies
\begin{align}\label{eq:bbj}
\frac{1}{\beta_j} 
= \frac{1-\alpha_j}{mr_j} + \frac{1}{(ms)'} + \sum_{i \ne j} \frac{1}{mr_j}.
\end{align}
To proceed, we claim that 
\begin{align}\label{eq:yyj}
\Y'_j = \Y' \prod_{i \ne j} \Y_i 
\quad \text{ and } \quad 
(\Y_j^{\alpha_j})' = \Y_j^{1-\alpha_j} \Y' \prod_{i \ne j} \Y_i.
\end{align}
Since $\Y_j$ and $\Y$ are Banach function spaces, all spaces below are also Banach function spaces:
\begin{align}\label{eq:ybfs}
\Y'_j, \, \Y', \, \Y_j^{\frac1m}, \, \Y_j^{\alpha_j}, \, 
\Y_j^{1-\alpha_j}, \, (\Y'_j)^{\frac1m}, \, 
(\Y_j^{\alpha_j})', \quad \text{and} \quad
\Big(\Y' \prod_{i \ne j} \Y_i \Big)^{\frac1m}.
\end{align}
Note that 
\begin{align*}
\Y_j^{\frac1m} \Big(\Y' \prod_{i \ne j} \Y_i \Big)^{\frac1m}
= L^m(\Rd) = \Y_j^{\frac1m} (\Y'_j)^{\frac1m},
\end{align*}
which together with \eqref{eq:ybfs} and \cite[Corollary~2.6]{Schep} gives $\Y'_j = \Y' \prod_{i \ne j} \Y_i $. In addition, 
\begin{align*}
\Y_j^{1-\alpha_j} \Y' \prod_{i \ne j} \Y_i
= \Y_j^{1-\alpha_j} \Y'_j
= (\Y_j \Y'_j)^{1-\alpha_j} (\Y'_j)^{\alpha_j}
= L^1(\Rd)^{1-\alpha_j} (\Y'_j)^{\alpha_j},
\end{align*}
which means $\Y_j^{1-\alpha_j} \Y' \prod_{i \ne j} \Y_i$ is a Banach function space. Moreover, 
\begin{align*}
\Y_j^{\alpha_j} \Big(\Y_j^{1-\alpha_j} \Y' \prod_{i \ne j} \Y_i \Big)
= \Y_j \Big( \Y' \prod_{i \ne j} \Y_i \Big)
= \Y_j \Y'_j
= L^1(\Rd)
= \Y_j^{\alpha_j} (\Y_j^{\alpha_j})',
\end{align*}
which along with \cite[Corollary~2.6]{Schep} yields $(\Y_j^{\alpha_j})' = \Y_j^{1-\alpha_j} \Y' \prod_{i \ne j} \Y_i$.

Given $f \in (\Y_j^{\alpha_j})'$ with $\|f\|_{(\Y_j^{\alpha_j})'} \le 1$, by \eqref{eq:yyj}, there exist nonnegative functions $f_j \in \Y_j^{1-\alpha_j}$, $h \in \Y'$, and $f_i \in \Y_i$ for each $i \ne j$, such that 
\begin{align}\label{eq:ffj}
|f| \le f_j h \prod_{i \ne j} f_i
\quad \text{ and } \quad 
A := \|f_j\|_{\Y_j^{1-\alpha_j}} \|h\|_{\Y'} \prod_{i \ne j} \|f_i\|_{\Y_i} \le 2.
\end{align} 
By definition, there holds
\begin{align}
\label{eq:fhf-1}
\|f_j \mathbf{1}_Q\|_{L^{\frac{mr_j}{1-\alpha_j}}}
& = \big\| |f_j|^{\frac{1}{1-\alpha_j}} \mathbf{1}_Q \big\|_{L^{mr_j}}^{1-\alpha_j}
\le \big\| |f_j|^{\frac{1}{1-\alpha_j}} \big\|_{\Y_j}^{1-\alpha_j} 
\sup_{\|g\|_{\Y_j} \le 1} \|g \mathbf{1}_Q\|_{L^{mr_j}}^{1-\alpha_j},
\\
\label{eq:fhf-2}
\|h \mathbf{1}_Q\|_{L^{(ms)'}}
& \le \|h\|_{\Y'} \sup_{\|g\|_{\Y'} \le 1} \|g \mathbf{1}_Q\|_{L^{(ms)'}},
\\
\label{eq:fhf-3}
\|f_i \mathbf{1}_Q\|_{L^{mr_i}}
& \le \|f_i\|_{\Y_i} \sup_{\|g\|_{\Y_i} \le 1} \|g \mathbf{1}_Q\|_{L^{mr_i}}.
\end{align} 
Then by \eqref{eq:bbj}, \eqref{eq:ffj}--\eqref{eq:fhf-3}, and H\"{o}lder's inequality, 
\begin{align*}
\|f \mathbf{1}_Q\|_{L^{\beta_j}}
& \le \|f_j \mathbf{1}_Q\|_{L^{\frac{mr_j}{1-\alpha_j}}} 
\|h \mathbf{1}_Q\|_{L^{(ms)'}} 
\prod_{i \ne j} \|f_i \mathbf{1}_Q\|_{L^{mr_i}} 
\\
& \le A \sup_{\|g\|_{\Y_j} \le 1} \|g \mathbf{1}_Q\|_{L^{mr_j}}^{-\alpha_j}
\prod_{i=1}^m \sup_{\|g\|_{\Y_i} \le 1} \|g \mathbf{1}_Q\|_{L^{mr_i}}
\sup_{\|g\|_{\Y'} \le 1} \|g \mathbf{1}_Q\|_{L^{(ms)'}}
\\
& \le 2 \mathcal{N}_{\vec{X}} 
\sup_{\|g\|_{\Y_j} \le 1} \|g \mathbf{1}_Q\|_{L^{mr_j}}^{-\alpha_j},
\end{align*}
which means 
\begin{align}\label{eq:fnx}
\sup_{\|g\|_{\Y_j} \le 1} \|g \mathbf{1}_Q\|_{L^{m r_j}}^{\alpha_j} 
\sup_{\|f\|_{(\Y_j^{\alpha_j})'} \le 1} \|f \mathbf{1}_Q\|_{L^{\beta_j}}
\le 2 \mathcal{N}_{\vec{\X}} < \infty. 
\end{align}

On the other hand, by the fact $\Y_j = L^{mp_j, mq_j}(w_j^{p_j})$ and \eqref{eq:uupq-2}, 
\begin{align}\label{eq:fnx-1}
\sup_{\|g\|_{\Y_j} \le 1} \|g \mathbf{1}_Q\|_{L^{m r_j}} 
\simeq \|\sigma_j \mathbf{1}_Q\|_{L^{1, \frac{(q_j/r_j)'}{(p_j/r_j)'}}}^{\frac{1}{mr_j} - \frac{1}{mp_j}}.
\end{align}
Denote $a_j := \frac{(mp_j/\alpha_j)'}{\beta_j}$, and $b_j := \frac{(mq_j/\alpha_j)'}{\beta_j}$. Observing that 
\begin{align*}
\Y_j^{\alpha_j} = L^{mp_j/\alpha_j, mq_j/\alpha_j}(w_j^{p_j})
\quad\text{ and } \quad 
\|f\|_{(\Y_j^{\alpha_j})'} = \|f w_j^{-p_j}\|_{L^{(mp_j/\alpha_j)', (mq_j/\alpha_j)'}(w_j^{p_j})},
\end{align*}
we invoke \eqref{eq:uupq-1} to conclude 
\begin{align}\label{eq:fnx-2}
\sup_{\|f\|_{(\Y_j^{\alpha_j})'} \le 1} \|f \mathbf{1}_Q\|_{L^{\beta_j}}
& = \sup_{\|g\|_{L^{(mp_j/\alpha_j)', (mq_j/\alpha_j)'}(w_j^{p_j})} \le 1} 
\|g w_j^{p_j} \mathbf{1}_Q\|_{L^{\beta_j}}
\\ \nonumber
& = \sup_{\|h\|_{L^{a_j, b_j}(w_j^{p_j})} \le 1} 
\big\|h w_j^{p_j (\beta_j-1)} \mathbf{1}_Q \big\|_{L^1(w_j^{p_j})}^{\frac{1}{\beta_j}}
\\ \nonumber
& = \big\|w_j^{p_j (\beta_j-1)} \mathbf{1}_Q \big\|_{L^{a'_j, b'_j}(w_j^{p_j})}^{\frac{1}{\beta_j}}
\simeq \big\|w_j^{a'_j p_j (\beta_j-1) + p_j} \mathbf{1}_Q \big\|_{L^{1, b'_j/a'_j}}^{\frac{1}{a'_j \beta_j}}. 
\end{align}
Note that
\begin{align*}
\frac{a'_j p_j (\beta_j-1) + p_j}{a'_j \beta_j}
= p_j - \frac{p_j}{\beta_j} \bigg[1 - \frac{1}{a'_j} \bigg]
= p_j \bigg[1 - \frac{1}{a_j \beta_j} \bigg]
= p_j \bigg[1 - \frac{1}{(mp_j/\alpha_j)'} \bigg]
= \frac{\alpha_j}{m},
\end{align*}
which implies
\begin{align}\label{eq:fnx-3}
\sigma_j^{- (\frac{\alpha_j}{mr_j} - \frac{\alpha_j}{mp_j})}
= w_j^{\frac{\alpha_j}{m}}
= w_j^{\frac{a'_j p_j (\beta_j-1) + p_j}{a'_j \beta_j}}.
\end{align}
Hence, the result $\sigma_j \in A_{\infty}$ follows from \eqref{eq:fnx}--\eqref{eq:fnx-3} and Lemma \ref{lem:ssth}.

By the definition of $A_{\infty}$ and \eqref{eq:fnx-1}, 
\begin{align*}
\sigma_j(Q) 
\le e^{\langle \log \sigma_j \rangle_Q} |Q| 
\lesssim \|\sigma_j \mathbf{1}_Q\|_{L^{1, \frac{(q_j/r_j)'}{(p_j/r_j)'}}}
\simeq \sup_{\|f\|_{\Y_j} \le 1} \|f \mathbf{1}_Q\|_{L^{m r_j}}^{\frac{1}{{\frac{1}{m r_j} - \frac{1}{m p_j}}}},
\end{align*}
which completes the proof. 
\end{proof}

\subsubsection{{\bf Weighted variable Lebesgue spaces}}
Let $\X_j = L^{p_j(\cdot)}_{w_j}$ with $p_j(\cdot) \in \LH$, $r_j < (p_j)_- \le (p_j)_+ < \infty$, $\frac{1}{p(\cdot)} = \sum_{j=1}^m \frac{1}{p_j(\cdot)}$, and $p_+ < s$. Denote $\frac1r = \sum_{j=1}^m \frac{1}{r_j}$, $w = \prod_{j=1}^m w_j$, and $w_{m+1} = w^{-1}$. Then 
\begin{align*}
\X_j^{\frac1m} = L^{mp_j(\cdot)}_{w_j^{1/m}}, \quad 
\Y = L^{m p(\cdot)}_{w^{1/m}}, \quad \text{and} \quad
\Y' = L^{(m p(\cdot))'}_{w^{-1/m}}.
\end{align*} 
For convenience, for each $j=1, \ldots, m$, write 
\begin{align*}
\s_{m+1} = (ms)', \quad 
\p_{m+1}(\cdot) = (mp(\cdot))', \quad
\s_j = m r_j, \quad \text{and} \quad
\p_j(\cdot) = m p_j(\cdot), 
\end{align*} 
Define $\frac{1}{\q_j(\cdot)} := \frac{1}{\s_j} - \frac{1}{\p_j(\cdot)}$, $j=1, \ldots, m+1$. Then 
\begin{align}\label{eq:Apcd}
[\vec{w}]_{A_{\vec{p}(\cdot), (\vec{r}, s)}}^{\frac1m}
= \sup_{Q \in \Q} |Q|^{-(\frac{1}{mr} - \frac{1}{ms})} 
\prod_{j=1}^{m+1} \|w_j^{-\frac1m} \mathbf{1}_Q\|_{L^{\q_j(\cdot)}}.
\end{align}

\noindent $\bullet$ {\bf The necessity of $A_{\vec{p}(\cdot), (\vec{r}, s)}$.} Assume that $\mathcal{N}_{\vec{\X}} < \infty$. Note that for all cubes $Q \subset \Rd$,
\begin{align*}
\big\|w_j^{-\frac1m} \mathbf{1}_Q \big\|_{L^{\q(\cdot)}}^{\s_j}
& = \big\|w_j^{-\frac{\s_j}{m}} \mathbf{1}_Q \big\|_{L^{(\p_j(\cdot)/\s_j)'}} 
\simeq \sup_{\|h\|_{L^{\p_j(\cdot)/\s_j}} \le 1} 
\big\|h w_j^{-\frac{\s_j}{m}} \mathbf{1}_Q \big\|_{L^1} 
\\
& 
= \sup_{\|g\|_{L^{\p_j(\cdot)}} \le 1} \big\| |g w_j^{-\frac1m}|^{\s_j} \mathbf{1}_Q \big\|_{L^1} 
= \sup_{\|f\|_{L^{\p_j(\cdot)}_{w_j^{1/m}}} \le 1} \|f \mathbf{1}_Q \|_{L^{\s_j}}^{\s_j},
\end{align*}
which implies
\begin{align*}
& |Q|^{-(\frac{1}{mr} - \frac{1}{ms})} 
\prod_{j=1}^{m+1} \|w_j^{-\frac1m} \mathbf{1}_Q\|_{L^{\q_j(\cdot)}}
\\
& \simeq  |Q|^{-(\frac{1}{mr} - \frac{1}{ms})}  
\prod_{j=1}^m \sup_{\|f_j\|_{\X_j^{\frac1m}} \le 1} \|f_j \mathbf{1}_Q\|_{L^{\s_j}}
\sup_{\|f\|_{\Y'} \le 1} \|f \mathbf{1}_Q\|_{L^{\s_{m+1}}}
\\
& \le \|\M_{m \vec{r}, (ms)'}\|_{\X_1^{\frac1m} \times \cdots \times \X_m^{\frac1m} \times \Y' \to L^1}
= \mathcal{N}_{\vec{X}} < \infty.
\end{align*}
In view of \eqref{eq:Apcd}, this indicates $\vec{w} \in A_{\vec{p}(\cdot), (\vec{r}, s)}$.

\noindent $\bullet$ {\bf The sufficiency of $A_{\vec{p}(\cdot), (\vec{r}, s)}$.}
Now assume that $\vec{w} \in A_{\vec{p}(\cdot), (\vec{r}, s)}$. Observe that 
\begin{align*}
\sum_{j=1}^{m+1} \frac{1}{\p_j(\cdot)} = 1 
\quad \text{ and } \quad 
\sum_{j=1}^{m+1} \frac{1}{\q_j(\cdot)} 
= \frac{1}{mr} - \frac{1}{ms} =: \eta > 0,
\end{align*}
which gives
\begin{align*}
\frac{1}{(\eta \q_j(\cdot))'} = \sum_{i \ne j} \frac{1}{\eta \q_j(\cdot)}
\quad \text{ and } \quad 
w_j^{\frac{1}{m \eta}} = \prod_{i \ne j} w_i^{- \frac{1}{m \eta}}.
\end{align*}
Then we use H\"{o}lder's inequality and \eqref{eq:Apcd} to obtain
\begin{align*}
& |Q|^{-1} \|w_j^{-\frac{1}{m \eta}} \mathbf{1}_Q\|_{L^{\eta \q_j(\cdot)}}
\|w_j^{\frac{1}{m \eta}} \mathbf{1}_Q\|_{L^{(\eta \q_j(\cdot))'}}
\\
& \le |Q|^{-1} \|w_j^{-\frac{1}{m \eta}} \mathbf{1}_Q\|_{L^{\eta \q_j(\cdot)}}
\prod_{i \ne j} \|w_i^{- \frac{1}{m \eta}} \mathbf{1}_Q\|_{L^{\eta \q_i(\cdot)}}
\\
& = \bigg[|Q|^{-\eta} \prod_{i=1}^{m+1} \|w_i^{- \frac1m} \mathbf{1}_Q\|_{L^{\q_i(\cdot)}}\bigg]^{\frac{1}{\eta}}
\le [\vec{w}]_{A_{\vec{p}(\cdot), (\vec{r}, s)}}^{\frac{1}{m \eta}},
\end{align*}
which means $w_j^{-\frac{1}{m \eta}} \in A_{\eta \q_j(\cdot)}$. By this and \cite[Theorem 1.1]{CP}, there exists  some $\tau > 1$ so that
\begin{align*}
\|w_j^{-\frac{1}{m \eta}} \mathbf{1}_Q\|_{L^{\tau \eta q_j(\cdot)}} 
\lesssim |Q|^{- \langle \frac{1}{\tau' \eta q_j} \rangle_Q} 
\|w_j^{-\frac{1}{m \eta}} \mathbf{1}_Q\|_{L^{\eta q_j(\cdot)}}, \quad \forall Q \in \Q,
\end{align*}
which is equivalent to 
\begin{align}\label{eq:met}
\|w_j^{-\frac1m} \mathbf{1}_Q\|_{L^{\tau q_j(\cdot)}} 
\lesssim |Q|^{- \langle \frac{1}{\tau' q_j} \rangle_Q} 
\|w_j^{-\frac1m} \mathbf{1}_Q\|_{L^{q_j(\cdot)}}, \quad \forall Q \in \Q.
\end{align}

Fo every $j=1, \ldots, m+1$, define
\begin{align}\label{eq:qspj}
\frac{1}{\t_j(\cdot)} 
:= \frac{1}{\s_j} - \frac{1}{\tau \q_j(\cdot)}
= \frac{1}{\p_j(\cdot)} + \frac{1}{\tau' \q_j(\cdot)},
\end{align}
which satisfies 
\begin{align}\label{eq:ssLH}
\t_j \in \mathrm{LH}, \quad 
1 < (\t_j)_- \le (\t_j)_+ < \infty, \quad \text{and} \quad
\esssup \frac{\t_j(\cdot)}{\p_j(\cdot)} < 1.
\end{align}
Then by \eqref{eq:ssLH} and \cite[Theorem 7.3.27]{DHHR},
\begin{align}\label{eq:Mtj}
\|M_{\t_j(\cdot)}\|_{L^{\p_j(\cdot)} \to L^{\p_j(\cdot)}}
\lesssim 1.
\end{align}
Additionally, it follows from \eqref{eq:qspj}, H\"{o}lder's inequality, and \eqref{eq:met} that
\begin{align*}
\|f_j \mathbf{1}_Q\|_{L^{\s_j}}
& \lesssim \|f_j w_j^{\frac1m} \mathbf{1}_Q\|_{L^{\t_j(\cdot)}}
\|w_j^{-\frac1m} \mathbf{1}_Q\|_{L^{\tau \q_j(\cdot)}}
\\
& \lesssim \|f_j w_j^{\frac1m} \mathbf{1}_Q\|_{L^{\t_j(\cdot)}}
|Q|^{- \langle \frac{1}{\tau' \q_j} \rangle_Q} 
\|w_j^{-\frac1m} \mathbf{1}_Q\|_{L^{\q_j(\cdot)}},
\end{align*}
which gives for every $\gamma$-sparse family $\S$,
\begin{align*}
\Lambda_{\S}(\vec{f}, f_{m+1}) 
& := \sum_{Q \in \S} |Q| \prod_{j=1}^m \langle f_j \rangle_{Q, mr_j}
\langle f_{m+1} \rangle_{Q, (ms)'}
\\
& = \sum_{Q \in \S} |Q|^{-(\frac{1}{mr} - \frac{1}{ms})} 
\prod_{j=1}^{m+1} \|f_j \mathbf{1}_Q\|_{L^{\s_j}}
\\
& \lesssim \sum_{Q \in \S} |Q|^{1 -(\frac{1}{mr} - \frac{1}{ms})}
\prod_{j=1}^{m+1}
\frac{\|f_j w_j^{\frac1m} \mathbf{1}_Q\|_{L^{\t_j(\cdot)}}}{\|\mathbf{1}_Q\|_{L^{\t_j(\cdot)}}}
\|w_j^{-\frac1m} \mathbf{1}_Q\|_{L^{\q_j(\cdot)}}
\\
&\lesssim [\vec{w}]_{A_{\vec{p}(\cdot), (\vec{r}, s)}}^{\frac1m} 
\sum_{Q \in \S} \int_{E_Q} \prod_{j=1}^{m+1} M_{\t_j(\cdot)}(f_j w_j^{\frac1m})(x) \, dx
\\
&\lesssim [\vec{w}]_{A_{\vec{p}(\cdot), (\vec{r}, s)}}^{\frac1m} 
\prod_{j=1}^{m+1} \|M_{\t_j(\cdot)}(f_j w_j^{\frac1m})\|_{L^{\p_j(\cdot)}}
\\
&\lesssim [\vec{w}]_{A_{\vec{p}(\cdot), (\vec{r}, s)}}^{\frac1m} 
\prod_{j=1}^{m+1} \|f_j w_j^{\frac1m}\|_{L^{\p_j(\cdot)}},
\\
& = [\vec{w}]_{A_{\vec{p}(\cdot), (\vec{r}, s)}}^{\frac1m} 
\prod_{j=1}^m \|f_j\|_{\X_j^{\frac1m}} \|f_{m+1}\|_{\Y'},
\end{align*}
provided that \eqref{eq:qspj}, $\sum_{j=1}^{m+1} \frac{1}{\p_j(\cdot)} = 1$, $\|\mathbf{1}_Q\|_{L^{\t_j(\cdot)}} \simeq |Q|^{\langle \t_j^{-1} \rangle_Q}$ (cf. \cite[Lemma 4.5.3]{DHHR}), \eqref{eq:Apcd}, H\"{o}lder's inequality, and \eqref{eq:Mtj}. Hence, by Theorem \ref{thm:MM},
\begin{align*}
\|\M_{m \vec{r}, (ms)'}\|_{\X_1^{\frac1m} \times \cdots \times \X_m^{\frac1m} \times \Y' \to L^1}
\simeq 
\|\Lambda_{m \vec{r}, (ms)'}^*\|_{\X_1^{\frac1m} \times \cdots \times \X_m^{\frac1m} \times \Y' \to \R} < \infty.
\end{align*}

\subsubsection{{\bf Weighted Morrey spaces}}
Let $w_j(x) = |x|^{d \alpha_j}$ and $\X_j = \L^{p_j, q_j}_{w_j}$ with $\alpha_j \in \R$, $r_j < p_j \le q_j < \infty$, and $\frac1s < \frac1q = \sum_{j=1}^m \frac{1}{q_j}$. Denote $\alpha = \sum_{j=1}^m \alpha_j$ and $w = \prod_{j=1}^m w_j = |x|^{d \alpha}$. Then 
\begin{align*}
\X_j^{\frac1m} = \L^{mp_j, mq_j}_{w_j^{1/m}}, \quad
\Y = \L^{mp, mq}_{w^{1/m}}, \quad \text{ and } \quad 
\Y' = \B_{w^{-1/m}}^{(mp)', (mq)'}.
\end{align*} 
In this current setting, one has
\begin{align}\label{eq:waqr}
[\vec{w}]_{A_{\vec{q}, (\vec{r}, s)}}
= \sup_{Q \in \Q} \bigg \langle |x|^{\frac{d \alpha}{\frac1q - \frac1s}} \bigg \rangle_Q^{\frac1q - \frac1s}
\prod_{j=1}^m \bigg \langle |x|^{- \frac{d \alpha_j}{\frac{1}{r_j} - \frac{1}{q_j}}} \bigg \rangle_Q^{\frac{1}{r_j} - \frac{1}{q_j}}.
\end{align}

First, assume that $\vec{w} \in A_{\vec{q}, (\vec{r}, s)}$. Since for the cube $Q_0 = [-1, 1)^d$,
\begin{align*}
\beta > -1 \quad \iff \quad 
\langle |x|^{d \beta} \rangle_{Q_0}
\simeq \int_0^1 t^{d(1+\beta)} \frac{dt}{t} < \infty,
\end{align*}
the finiteness of $[\vec{w}]_{A_{\vec{q}, (\vec{r}, s)}}$ implies $\frac{\alpha}{\frac1q - \frac1s} > -1$ and $- \frac{\alpha_j}{\frac{1}{r_j} - \frac{1}{q_j}} > -1$, $j=1, \ldots, m$,
which is equivalent to 
\begin{align}\label{eq:jsq}
\alpha > \frac1s - \frac1q 
\quad \text{ and } \quad
\alpha_j < \frac{1}{r_j} - \frac{1}{q_j}, \quad j=1, \ldots, m.
\end{align}

Next, assume \eqref{eq:jsq} holds. It was shown in \cite[eq. (4.3)]{MV} that for every $\alpha > -1$,
\begin{align}\label{eq:aaw}
\langle |x|^{d \alpha} \rangle_Q 
\simeq [\ell(Q) + \dist(0, Q)]^{d \alpha}, \quad \forall Q \in \Q,
\end{align}
which together with \eqref{eq:waqr} gives $\vec{w} \in A_{\vec{q}, (\vec{r}, s)} < \infty$.

To prove that $\mathcal{N}_{\vec{\X}} < \infty$ implies \eqref{eq:jsq}, we argue by contradiction. Assume that $\mathcal{N}_{\vec{\X}} < \infty$ but \eqref{eq:jsq} fails. Let us first handle the case $\alpha_j \ge \frac{1}{r_j} - \frac{1}{q_j}$ for some $j \in \{1, \ldots, m\}$. Choose
\begin{align*}
f_j(x) := |x|^{-\frac{d}{m} (\alpha_j + \frac{1}{q_j})}
[\log (e |x|^{-1})]^{- \beta_j} \mathbf{1}_{B(0, 1)}
\quad \text{ and } \quad
f_i = f_{m+1} = \mathbf{1}_{B(0, 2) \setminus B(0, 1)},
\end{align*}
for each $i \ne j$, where $\beta_j = 0$ if $p_j < q_j$, and $\beta_j \in (\frac{1}{mp_j}, \frac{1}{mr_j})$ if $p_j = q_j$. Let $Q = [-2, 2)^d$. In the scenario $p_j < q_j$, by \eqref{eq:aaw},
\begin{align*}
\|f_j\|_{\X_j^{\frac1m}} 
& \le \big\| |x|^{-\frac{d}{m} (\alpha_j + \frac{1}{q_j})} w_j^{\frac1m} \big\|_{\L^{mp_j, mq_j}}
= \big\| |x|^{-\frac{d}{q_j}} \big\|_{\L^{p_j, q_j}}^{\frac1m}
\\
& \simeq \sup_Q \bigg[\frac{\ell(Q)}{\ell(Q) + \dist(0, Q)}\bigg]^{\frac{d}{m q_j}} 
\le 1;
\end{align*}
while in the scenario $p_j = q_j$, by change of variables, 
\begin{align*}
\|f_j\|_{\X_j^{\frac1m}} 
& = \big\| |x|^{-d} [\log (e |x|^{-1})]^{-mp_j \beta_j} \mathbf{1}_{B(0, 1)}\big\|_{L^1}^{\frac{1}{mp_j}}
\\
& \simeq \int_0^1 \frac{dt}{t [\log (e t^{-1})]^{m p_j \beta_j}}
= \int_1^{\infty} \frac{ds}{s^{mp_j \beta_j}} 
\simeq 1.
\end{align*}
Note that in both scenarios, 
\begin{align*}
\|f_i\|_{\X_i^{\frac1m}} 
& = \big\|\mathbf{1}_{\{1 \le |x| \le 2\}} |x|^{d \alpha_i} \big\|_{\L^{p_i, q_i}}^{\frac1m}
\simeq 1,
\\
\|f_{m+1}\|_{\Y'}
& \le |Q|^{\frac{1}{mp} - \frac{1}{mq}} \|f_{m+1}\|_{L^{(mp)', (mq)'}}
\simeq 1,
\\
\langle f_i \rangle_{Q, mr_i}^{mr_i}
& = |B(0, 2) \setminus B(0, 1)|/|Q|
= \langle f_{m+1} \rangle_{Q, (m s)'}^{(ms)'},
\\
\langle f_j \rangle_{Q, mr_j}^{mr_j}
& \simeq \int_0^1 \frac{t^{-d r_j(\alpha_j - \frac{1}{r_j} + \frac{1}{q_j})}}{[\log (e t^{-1})]^{m r_j \beta_j}} \, \frac{dt}{t}
= \int_1^{\infty} \frac{e^{d r_j(\alpha_j - \frac{1}{r_j} + \frac{1}{q_j})(s-1)}}{s^{m r_j \beta_j}} \, ds
= \infty.
\end{align*}
The above means $f_i \in \Y_i^{\frac1m}$ and $f_{m+1} \in \Y'$ for all $i=1, \ldots, m$, but
\begin{align*}
\|\M_{m \vec{r}, (ms)'}\|_{L^1}
\ge |Q| \prod_{i=1}^m \langle f_i \rangle_{Q, mr_i}
\langle f_{m+1} \rangle_{Q, (m s)'}
= \infty,
\end{align*}
which contradicts the condition $\mathcal{N}_{\vec{\X}} < \infty$. 

We proceed to deal with the case $\sum_{j=1}^m \alpha_j \le \frac1s - \frac1q $. In this setting, choose
\begin{align*}
f_j(y) := |y|^{-\frac{d}{m} (\alpha_j + \frac{1}{q_j})}
[\log (e + |y|)]^{- \beta_j} \mathbf{1}_{B(0, 2)^c}
\quad \text{ and } \quad
f_{m+1} = \mathbf{1}_{B(0, 2) \setminus B(0, 1)},
\end{align*}
for each $j=1, \ldots, m$, where $\beta_j = 0$ if $p_j < q_j$, and $\beta_j > \frac{1}{mp_j}$ if $p_j = q_j$ so that $\frac{1}{mp} < \beta := \sum_{j=1}^m \beta_j < 1$. Fix $|x| \ge 2$ and let $Q_x = [-2|x|, 2|x|)^d$. In the scenario $p_j < q_j$, 
\begin{align*}
\|f_j\|_{\X_j^{\frac1m}}^m 
\le \big\| |y|^{-\frac{d}{q_j}} \big\|_{\L^{p_j, q_j}}
\lesssim \sup_{Q \in Q} |Q|^{\frac{1}{q_j} - \frac{1}{p_j}} 
\int_0^{\infty} \min\{|Q|, t^{- \frac{q_j}{p_j}}\} \, dt
\lesssim 1;
\end{align*}
while in the scenario $p_j = q_j$, by change of variables, 
\begin{align*}
\|f_j\|_{\X_j^{\frac1m}} 
& = \big\| |y|^{-d} [\log (e + |y|)]^{-mp_j \beta_j} \mathbf{1}_{B(0, 2)^c}\big\|_{L^1}^{\frac{1}{mp_j}}
\\
& \simeq \int_2^{\infty} \frac{dt}{t [\log (e + t)]^{m p_j \beta_j}}
\le \int_1^{\infty} \frac{ds}{s^{mp_j \beta_j}} 
\simeq 1.
\end{align*}
Note that in both scenarios, 
\begin{align*}
& \|f_{m+1}\|_{\Y'} \simeq 1, \qquad 
\langle f_{m+1} \rangle_{Q_x, (m s)'}^{(ms)'}
= |B(0, 2) \setminus B(0, 1)|/|Q_x| 
\simeq |x|^{-d},
\\
& \langle f_j \rangle_{Q_x, mr_j}^{mr_j}
\ge \frac{1}{|Q_x|} \int_{\{|x| \le |y| \le 2|x|\}} 
\frac{|y|^{-d r_j(\alpha_j + \frac{1}{q_j})}}{[\log (e + |y|)]^{m r_j \beta_j}} \, dy
\simeq \frac{|x|^{-d r_j(\alpha_j + \frac{1}{q_j})}}{[\log (e + |x|)]^{m r_j \beta_j}}. 
\end{align*}
Consequently, 
\begin{align*}
\|\M_{m \vec{r}, (ms)'}\|_{L^1}
& \ge \int_{|x| \ge 2} \prod_{j=1}^m \langle f_j \rangle_{Q_x, mr_j}
\langle f_{m+1} \rangle_{Q_x, (m s)'} \, dx
\\
& \gtrsim \int_{|x| \ge 2} |x|^{-\frac{d}{(ms)'}} \prod_{j=1}^m 
\frac{|x|^{- \frac{d}{m}(\alpha_j + \frac{1}{q_j})}}{[\log (e + |x|)]^{\beta_j}} \, dx
\\
& = \int_{|x| \ge 2} \frac{|x|^{-d + \frac{d}{m}(\frac1s - \frac1q - \alpha)}}{[\log (e + |x|)]^{\beta}} \, dx
\simeq \int_2^{\infty} \frac{t^{\frac{d}{m}(\frac1s - \frac1q - \alpha)}}{[\log (e + t)]^{\beta}} \, \frac{dt}{t}
= \infty,
\end{align*}
which contradicts the condition $\mathcal{N}_{\vec{\X}} < \infty$. 

Finally, let us justify \eqref{eq:jsq} implies $\mathcal{N}_{\vec{\X}} < \infty$. We may assume $p<q$ because in the case $p=q$, $\L^{p_j, q_j}_{w_j} = L^{p_j}_{w_j}$, $A_{\vec{q}, (\vec{r}, s)} = A_{\vec{p}, (\vec{r}, s)}$, we have proved $\eqref{eq:jsq} \iff  \vec{w} \in A_{\vec{p}, (\vec{r}, s)}$, and the latter is equivalent to $\mathcal{N}_{\vec{\X}} < \infty$ (cf. \cite[Proposition 2.7]{Nie19}). Observe that
\begin{align}\label{eq:MTT}
\M_{m \vec{r}, (ms)'}(\vec{f}, f)(x)
\lesssim \mathcal{T}_1(\vec{f}, f)(x) + \mathcal{T}_2(\vec{f}, f)(x),
\end{align}
where
\begin{align*}
\mathcal{T}_1(\vec{f}, f)(x)
& := \sup_{k \in \Z} \mathbf{1}_{B_k}(x)
\prod_{j=1}^m  \langle f_j \rangle_{B_k, mr_j}\, 
\langle f \rangle_{B_k, (ms)'},
\\
\mathcal{T}_2(\vec{f}, f)(x)
& := \prod_{j=1}^m M_{m r_j} (f_j w_j^{\frac1m}) (x) \, 
M_{(ms)'} (f w^{\frac1m}) (x).
\end{align*}
Here $B_k := B(0, 2^k)$. Indeed, if $\dist(0, Q) < 2 \sqrt{d} \ell(Q)$, then $Q \subset B(0, 3 \sqrt{d} \ell(Q)) =: B_Q$, 
\begin{align}\label{eq:BQm-1}
\langle f_j \rangle_{Q, mr_j} \lesssim \langle f_j \rangle_{B_Q, mr_j}, 
\quad \text{ and } \quad 
\langle f \rangle_{Q, (ms)'} \lesssim \langle f \rangle_{B_Q, (ms)'}. 
\end{align}
If $\dist(0, Q) \ge 2 \sqrt{d} \ell(Q)$, then for all $x, y \in Q$, we have $|x| \simeq |y|$, $w_j(x) \simeq w_j(y)$, and $w(x) \simeq w(y)$, which gives
\begin{align}
\label{eq:BQm-2}
\langle f_j \rangle_{Q, mr_j} 
& \lesssim w_j(x)^{-\frac1m} M_{mr_j} (f_j w_j^{\frac1m})(x), 
\\
\label{eq:BQm-3}
\langle f \rangle_{Q, (ms)'} 
& \lesssim w(x)^{-\frac1m} M_{(ms)'} (f w^{\frac1m})(x). 
\end{align}
Thus, \eqref{eq:MTT} is a consequence of \eqref{eq:BQm-1}--\eqref{eq:BQm-3}.

In order to bound $\mathcal{T}_1$, by H\"{o}lder's inequality, 
\begin{align*}
\langle f_j \rangle_{B_k, mr_j}
& \le \sum_{i \le k} |B_k|^{-\frac{1}{mr_j}} 
\big\| \mathbf{1}_{B_i \setminus B_{i-1}} w_j^{-\frac1m} 
f_j w_j^{\frac1m} \big\|_{L^{mr_j}}
\\
& \lesssim \sum_{i \le k} 2^{- \frac{kd}{mr_j}} 
2^{-i \frac{d}{m} \alpha_j} 2^{i \frac{d}{m} (\frac{1}{r_j} - \frac{1}{p_j})}
\big\| \mathbf{1}_{B_i} f_j w_j^{\frac1m} \big\|_{L^{mp_j}}
\\
& \le \sum_{i \le k} 2^{- \frac{kd}{mr_j}} 
2^{-i \frac{d}{m} \alpha_j} 2^{i \frac{d}{m} (\frac{1}{r_j} - \frac{1}{p_j})}
2^{id (\frac{1}{mp_j} - \frac{1}{mq_j})}
\|f_j w_j^{\frac1m}\|_{\L^{mp_j, mq_j}}
\\
& = \sum_{i \le k} 2^{- \frac{kd}{mr_j}} 
2^{i \frac{d}{m}(\frac{1}{r_j} - \frac{1}{q_j} - \alpha_j)}
\|f_j\|_{\X_j^{\frac1m}}
\lesssim 2^{- k \frac{d}{m} (\alpha_j + \frac{1}{q_j})}
\|f_j\|_{\X_j^{\frac1m}},
\end{align*}
and by Minkowski's inequality,
\begin{align*}
& \sum_{k \in \Z} 2^{-k \frac{d}{m} (\alpha - \frac1s + \frac1q)}
\|f \mathbf{1}_{B_k}\|_{L^{(ms)'}}
\\
& \le \sum_{k \in \Z} 2^{-k \frac{d}{m} (\alpha - \frac1s + \frac1q)}
\sum_{i \le k} \big\| \mathbf{1}_{B_i \setminus B_{i-1}} w^{\frac1m} 
f w^{-\frac1m} \big\|_{L^{(ms)'}}
\\
& \lesssim \sum_{i \in \Z} 2^{i d \frac{\alpha}{m}} 
\big\| \mathbf{1}_{B_i \setminus B_{i-1}} f w^{-\frac1m} \big\|_{L^{(ms)'}} 
\sum_{k \ge i} 2^{-k \frac{d}{m} (\alpha - \frac1s + \frac1q)}
\\
& = \sum_{i \in \Z} 2^{-i \frac{d}{m} (\frac1q - \frac1s)} 
\bigg\| \int_0^{\infty} \mathbf{1}_{\{x \in B_i \setminus B_{i-1}: \, |f w^{-\frac1m}| > t \}} \, dt \bigg\|_{L^{(ms)'}} 
\\
& \le \sum_{i \in \Z} 2^{-i \frac{d}{m} (\frac1q - \frac1s)} 
\int_0^{\infty} \bigg\| \mathbf{1}_{\{x \in B_i \setminus B_{i-1}: \, |f w^{-\frac1m}| > t \}} \bigg\|_{L^{(ms)'}} \, dt 
\\
& \lesssim \int_0^{\infty} \sum_{i \in \Z} 2^{-i \frac{d}{m} (\frac1q - \frac1s)} 
\min \big\{2^{id}, \, |\{|f w^{-\frac1m}| > t \}| \big\}^{\frac{1}{(ms)'}} \, dt 
\\
& \lesssim \int_0^{\infty} |\{|f w^{-\frac1m}| > t \}|^{\frac{1}{(mq)'}} \, dt 
\simeq \|f w^{-\frac1m}\|_{L^{(mq)', 1}}
\\
& \lesssim \|f w^{-\frac1m}\|_{\B^{(mp)', (mq)'}}
= \|f\|_{\Y'},
\end{align*}
where we have used $\B^{(mp)', (mq)'} = (\L^{mp, mq})' \hookrightarrow (L^{mq, \infty})' = L^{(mq)', 1}$ due to $L^{mq, \infty} \hookrightarrow \L^{mp, mq}$ (see \cite{Tsu}). Collecting the estimates above, we arrive at
\begin{align}\label{eq:TT1}
\|\mathcal{T}_1(\vec{f}, f)\|_{L^1}
& \le \sum_{k \in \Z} |B_k| \prod_{j=1}^m 
\langle f_j \rangle_{B_k, mr_j}\, 
\langle f \rangle_{B_k, (ms)'}
\\ \nonumber
& \lesssim \sum_{k \in \Z} 2^{-k \frac{d}{m} (\alpha - \frac1s + \frac1q)}
\|f \mathbf{1}_{B_k}\|_{L^{(ms)'}}
\prod_{j=1}^m \|f_j\|_{\X_j^{\frac1m}}
\\ \nonumber
& \lesssim \prod_{j=1}^m \|f_j\|_{\X_j^{\frac1m}} \|f\|_{\Y'}.
\end{align}
In addition, by \cite[Theorem 1]{CF} and \cite[Theorem 3.7]{Nie23}, for all $1 < \p \le \q < \infty$, 
\begin{align*}
\|M_{\s}\|_{\L^{\p, \q} \to \L^{\p, \q}} & \lesssim 1, \quad 1 \le \s < \p, 
\\
\|M_{\s}\|_{\L^{\p', \q'} \to \L^{\p', \q'}} & \lesssim 1, \quad 1 \le \s < \q',
\end{align*}
which imply 
\begin{align}\label{eq:TT2}
\|\mathcal{T}_2(\vec{f}, f)\|_{L^1}
& \lesssim \bigg\| \prod_{j=1}^m M_{m r_j} (f_j w_j^{\frac1m}) \bigg\|_{\L^{mp, mq}}
\|M_{(ms)'} (f w^{\frac1m})\|_{\B^{(mp)', (mq)'}}
\\ \nonumber
& \le \prod_{j=1}^m \| M_{m r_j} (f_j w_j^{\frac1m}) \|_{\L^{mp_j, mq_j}}
\|M_{(ms)'} (f w^{\frac1m})\|_{\B^{(mp)', (mq)'}}
\\ \nonumber
& \lesssim \prod_{j=1}^m \|f_j w_j^{\frac1m} \|_{\L^{mp_j, mq_j}}
\|f w^{\frac1m}\|_{\B^{(mp)', (mq)'}}
= \prod_{j=1}^m \|f_j\|_{\X_j^{\frac1m}} \|f\|_{\Y'}.
\end{align}
Therefore, the conclusion $\mathcal{N}_{\vec{\X}} < \infty$ follows from \eqref{eq:MTT}, \eqref{eq:TT1}, and \eqref{eq:TT2}.

\section{Applications}\label{sec:app}
In this section, we present some applications of multilinear extrapolation theorems obtained above. First, Theorem \ref{thm:Tbbdd} follows from Theorem \ref{thm:BFS} and \cite[Theorem 2.22]{LMO}. Note that by the hypotheses of Theorem \ref{thm:Tbcpt} and the diagonal case in \cite[Theorem 1.3]{COY}, we obtain 
\begin{align*}
[T, \b]_{\a}: L^{p_1}_{w_1} \times \cdots \times L^{p_m}_{w_m} \to L^p_w \, \text{ is compact}
\end{align*} 
for all $\vec{p} = (p_1, \ldots, p_m)$ with $(\vec{r}, s) \prec \vec{p}$, where $\frac1p = \sum_{j=1}^m \frac{1}{p_j}$ and $w = \prod_{j=1}^m w_j$. Hence, Theorem \ref{thm:Tbcpt} is a consequence of this and Theorem \ref{thm:BFScpt}.

\subsection{Boundedness of multilinear operators}\label{sec:app-bdd}
We are going to show Theorem \ref{thm:spBFS}. The boundedness $T: \X_1 \times \cdots \times \X_m \to \X$ is a direct consequence of Theorems \ref{thm:BFS} and \ref{thm:form}, while the boundedness of commutators $[T, \b]_{\a}$ follows from Theorems \ref{thm:Tbbdd} and \ref{thm:form}. As for the specific operators appearing in Theorem \ref{thm:spBFS}, we shall treat them one by one in what follows. 

\subsubsection{\bf{Multilinear $\omega$-Calder\'{o}n--Zygmund operators}}
Let $\omega: [0, \infty) \to [0, \infty)$ be a modulus of continuity, which means that $\omega$ is increasing, subadditive and $\omega(0)=0$. Given a modulus of continuity $\omega$, we say that $\omega \in \text{Dini}$ if  
\begin{align*}
\|\omega\|_{{\rm Dini}} := \int_{0}^{1}\omega(t) \frac{dt}{t} < \infty.
\end{align*} 
An example of Dini condition is $\omega(t)=t^{\delta}$ with $\delta>0$.

We say that a function $K:\R^{d(m+1)} \setminus \{x=y_1=\cdots=y_m\} \to \C$ is an $\omega$-{\tt Calder\'{o}n}--{\tt Zygmund kernel}, if there exists a constant $A>0$ such that 
\begin{align*}
|K(x,\vec{y})| & \le  \frac{A}{\big(\sum_{j=1}^{m}|x-y_j|\big)^{mn}},
\\ 
|K(x,\vec{y}) - K(x',\vec{y})| & \le \frac{A}{\big(\sum_{j=1}^{m}|x-y_j|\big)^{mn}} 
\omega\bigg(\frac{|x-x'|}{\sum_{j=1}^{m}|x-y_j|}\bigg),
\end{align*}
whenever $|x-x'| \leq \frac12 \max\limits_{1\leq j \leq m}|x-y_j|$, and for each $i=1,\ldots,m$, 
\begin{align*}
|K(x,\vec{y}) - K(x,y_1,\ldots,y'_i,\ldots,y_m)| & \le \frac{A}{\big(\sum_{j=1}^{m}|x-y_j|\big)^{mn}} 
\omega\bigg(\frac{|y_i-y'_i|}{\sum_{j=1}^{m}|x-y_j|}\bigg),
\end{align*}
whenever
$|y_i-y_i'| \leq \frac12 \max\limits_{1\leq j \leq m}|x-y_j|$.

An $m$-linear operator $T: \mathscr{S}(\Rd) \times \cdots \times \mathscr{S}(\Rd) \to \mathscr{S}'(\Rd)$ is called an $\omega$-{\tt Calder\'{o}n}--{\tt Zygmund operator} if there exists an $\omega$-Calder\'{o}n--Zygmund kernel $K$ such that 
\begin{align*}
T(\vec{f})(x) =\int_{(\Rd)^m} K(x, \vec{y}) f_1(y_1)\cdots f_m(y_m) d\vec{y},  
\end{align*}
whenever $x \not\in \bigcap_{i=1}^m \supp(f_i)$ and $\vec{f}=(f_1,\ldots,f_m) \in \mathscr{C}_c^{\infty}(\Rd) \times \cdots \times \mathscr{C}_c^{\infty}(\Rd)$, and $T$ can be boundedly extended from $L^{q_1}(\Rd) \times \cdots \times L^{q_m}(\Rd)$ to $L^q(\Rd)$ for some $\frac1q=\frac{1}{q_1}+\cdots+\frac{1}{q_m}$ with $1<q_1,\ldots,q_m<\infty$. 

If $\omega(t)=t^{\delta}$ with $\delta>0$, an $\omega$-Calder\'{o}n-Zygmund operator $T$ is called a standard Calder\'{o}n--Zygmund operator, which was studied by Grafakos and Torres \cite{GT}. For the general $\omega$, the linear $\omega$-CZO was introduced by Yabuta in \cite{Y}, while it was extended by Maldonado and Naibo \cite{MN} to the bilinear case. 

Invoking the pointwise sparse domination \cite[Theorem 3.6]{DHL} and Theorem \ref{thm:spBFS} for abstract operators, we obtain the result below.

\begin{theorem}\label{thm:CZO}
Let $T$ be an m-linear $\omega$-Calder\'{o}n--Zygmund operator with $\omega \in \text{Dini}$. Let $\X_1, \ldots, \X_m$ be Banach function spaces such that 
\begin{align*}
\M_{(m, \ldots, m), 1}: \X_1^{\frac1m} \times \cdots \times \X_m^{\frac1m}\times \Y' \to L^1 \text{ is bounded},
\end{align*} 
where $\Y = \X^{\frac1m}$ and $\X = \prod_{j=1}^m \X_j$. Then both $T$ and $[T, \b]_{\a}$ are bounded from $\X_1 \times \cdots \times \X_m$ to $\X$ for all $\a \in \N^m$ and $\b \in \BMO^m$.
\end{theorem}

\subsubsection{\bf{Multilinear pseudo-differential operators}}
Given a function $\sigma$ on $\Rd \times \R^{dm}$, the $m$-linear pseudo-differential operators $T_{\sigma}$ is defined by
\begin{align*}
T_a(\vec{f})(x) := \int_{\R^{dm}} a(x, \vec{\xi})
e^{2 \pi i x \cdot (\xi_1 + \cdots + \xi_m)} \widehat{f}_1(\xi_1) \cdots \widehat{f}_m(\xi_m) \, d\vec{\xi},
\end{align*}
for all $f_j \in \mathscr{S}(\Rd)$, $j=1,\ldots,m$, where $\widehat{f}$ is the Fourier transform of $f$. Given $\tau \in \R$ and $\rho, \delta \in [0, 1]$, we say $\sigma \in S_{\rho,\delta}^{\tau}$ if 
\begin{align*}
\big| \partial_{x}^{\alpha} \partial_{\xi_1}^{\beta_1} \cdots
\partial_{\xi_m}^{\beta_m}  a(x,\vec{\xi}) \big|
\leq C_{\alpha, \vec{\beta}} (1+ |\vec{\xi}|)^{\tau - \rho \sum_{j=1}^m |\beta_j| + \delta|\alpha|}, 
\qquad x \in \Rd.
\end{align*}
for all multi-indices $\alpha, \beta_1, \ldots, \beta_m$.

Then the following result is a consequence of \cite[Proposition 4.1]{CXY} and Theorem \ref{thm:spBFS} for abstract operators. 

\begin{theorem}\label{thm:pse}
Let $0 \le \rho, \delta \le 1$, $\tau < - md(1-\rho)$, and $\sigma \in S_{\rho, \delta}^{\tau}$. Let $\X_1, \ldots, \X_m$ be Banach function spaces such that 
\begin{align*}
\M_{(m, \ldots, m), 1}: \X_1^{\frac1m} \times \cdots \times \X_m^{\frac1m}\times \Y' \to L^1 \text{ is bounded},
\end{align*} 
where $\Y = \X^{\frac1m}$ and $\X = \prod_{j=1}^m \X_j$. Then both $T_{\sigma}$ and $[T_{\sigma}, \b]_{\a}$ are bounded from $\X_1 \times \cdots \times \X_m$ to $\X$ for all $\a \in \N^m$ and $\b \in \BMO^m$.
\end{theorem}

On the other hand, we can enlarge the range of $\tau$, but at the cost of strengthening the bounded of multilinear maximal operators. 
Set $\Gamma(r_1, r_2, s) := \frac{1}{\min\{r_1, 2\}} + \frac{1}{\min\{r_2, 2\}} - \frac{1}{\max\{s, 2\}}$ for $r_1, r_2, s \in (0, \infty)$. Combining \cite[Theorem 1.4]{CDY} with Theorem \ref{thm:spBFS} for abstract operators, we derive the result given below.

\begin{theorem}\label{thm:psee}
Let $0 < \rho \le 1$, $0 \le \delta \le \rho$, and $\delta < 1$. Let $1 < r_1, r_2, s < \infty$ with $\frac1s < \frac{1}{r_1} + \frac{1}{r_2}$, and $\tau \le - d(1-\rho) \Gamma(r_1, r_2, s)$. Let $\X_1, \ldots, \X_m$ be Banach function spaces such that 
\begin{align*}
\M_{2r_1, 2r_2, (2s)'}: \X_1^{\frac1m} \times \X_2^{\frac1m} \times \Y' \to L^1 \text{ is bounded},
\end{align*} 
where $\Y = \X^{\frac12}$ and $\X = \X_1 \X_2$. Then for every $\sigma \in S_{\rho, \delta}^{\tau}$, both $T_{\sigma}$ and $[T_{\sigma}, \b]_{\a}$ are bounded from $\X_1 \times \X_2$ to $\X$ for all $\a \in \N^2$ and $\b \in \BMO^2$.
\end{theorem}

\subsubsection{\bf{Bilinear Hilbert transform}}
The {\tt bilinear Hilbert transform} is defined as
\begin{align*}
BH(f_1, f_2)(x) := \mathrm{p.v.} \, \int_{\R} f_1(x-t) f_2(x+t) \, \frac{dt}{t}. 
\end{align*}
It is an integral operator falling outside of the theory of bilinear Calder\'{o}n--Zygmund operators. Lacey and Thiele \cite{LT97, LT99} showed that $BH$ is bounded $L^{p_1}(\Rd) \times L^{p_2}(\Rd) \to L^p(\Rd)$ for all $p_1, p_2 \in (1, \infty]$ with $\frac1p = \frac{1}{p_1} + \frac{1}{p_2} < \frac32$. The sparse domination for $BH$ was first established in \cite{CDO}: 
\begin{align}\label{eq:BHsp}
|\langle BH(f_1, f_2), \, f_3 \rangle|
\lesssim \Lambda_{r_1, r_2, s'}^*(f_1, f_2, f_3),
\end{align}
for all $f_1, f_2, f_3 \in L_c^{\infty}(\Rd)$, where $r_1, r_2, s \in (1, \infty)$ with 
\begin{align}\label{eq:BHrr}
\frac{1}{\min\{r_1, 2\}} + \frac{1}{\min\{r_2, 2\}} + \frac{1}{\min\{s', 2\}} < 2.
\end{align}
More general vector-valued sparse domination for $BH$ was given in \cite{BM}.

A direct consequence of Theorem \ref{thm:spBFS} and \eqref{eq:BHsp} is the following result.

\begin{theorem}\label{thm:BH}
Let $r_1, r_2, s \in (1, \infty)$ satisfy \eqref{eq:BHrr}. Let $\X_1, \ldots, \X_m$ be Banach function spaces such that 
\begin{align*}
\M_{2r_1, 2r_2, (2s)'}: \X_1^{\frac12} \times \X_2^{\frac12}\times \Y' \to L^1 \text{ is bounded},
\end{align*} 
where $\Y = \X^{\frac12}$ and $\X = \X_1 \X_2$. Then both $BH$ and $[BH, \b]_{\a}$ are bounded from $\X_1 \times \X_m$ to $\X$ for all $\a \in \N^2$ and $\b \in \BMO^2$.
\end{theorem}

\subsubsection{\bf{Multilinear rough singular integrals}}
Given $\Omega \in L^q(\bS^{md-1})$ with $q \in (1, \infty]$ and $\int_{\bS^{md-1}} \Omega\, d\sigma=0$, we define the {\tt multilinear rough singular integral operator} $T_{\Omega}$ by 
\begin{equation*}
T_{\Omega}(\vec{f})(x) 
:= \mathrm{p.v.} \int_{\R^{md}} K_{\Omega}(x-\vec{y}) \prod_{i=1}^m f_i(y_i) \, d\vec{y},  
\end{equation*}
where the rough kernel $K_{\Omega}$ is given by 
\begin{equation*}
K_{\Omega}(\vec{y}) 
:= \frac{\Omega(\vec{y}/|\vec{y}|)}{|\vec{y}|^{md}} 
\quad \text{with} \quad \vec{y} = (y_1, \ldots, y_m). 
\end{equation*}
Given $N \ge 2$ and $q \in (1, \infty]$, define the regions
\begin{align*}
\H_N^1(q) 
& := \bigg\{\vec{a} \in (0, 1)^N: 
\sum_{i \in I} a_i + \frac{|I| - 1}{q} < |I|, \, \, 
\forall \varnothing \ne I \subsetneq \{1, \ldots, N \}\bigg\},
\\
\H_N^d(q) 
& := \bigg\{\vec{a} \in (0, 1)^N: 
\sum_{i \in I} a_i + \frac{|I| - 1}{q} < |I|, \, \, 
\forall \varnothing \ne I \subseteq \{1, \ldots, N \}\bigg\},
\end{align*}
for every $d \ge 2$. Recent years have witnessed significant progress in multilinear rough singular integrals, see \cite{DS, GHH, GHHP23, GHHP24}.

By means of \cite[Theorem 1.2]{CDY} and Theorem \ref{thm:spBFS} for abstract operators, we establish boundedness of $T_{\Omega}$  and $[T_{\Omega}, \b]_{\a}$ on Banach function spaces.

\begin{theorem}\label{thm:rough}
Let $q \in (1, \infty]$ and $\Omega \in L^q(\bS^{md-1})$ with $\int_{\bS^{md-1}} \Omega\, d\sigma=0$. Let $(\vec{r}, s) \in (1, \infty)^{m+1}$ be such that $\frac1s < \sum_{j=1}^m \frac{1}{r_j}$ and $\big(\frac{1}{r_1}, \ldots, \frac{1}{r_m}, \frac{1}{s'} \big) \in \H_{m+1}^d(q)$. Let $\X_1, \ldots, \X_m$ be Banach function spaces such that 
\begin{align*}
\M_{m \vec{r}, (ms)'}: \X_1^{\frac1m} \times \cdots \times \X_m^{\frac1m}\times \Y' \to L^1 \text{ is bounded},
\end{align*} 
where $\Y = \X^{\frac1m}$ and $\X = \prod_{j=1}^m \X_j$. Then both $T_{\Omega}$ and $[T_{\Omega}, \b]_{\a}$ are bounded from $\X_1 \times \cdots \times \X_m$ to $\X$ for all $\a \in \N^m$ and $\b \in \BMO^m$.
\end{theorem}

\subsubsection{\bf{Bilinear oscillatory Fourier multipliers}}
Let $\phi \in \mathscr{C}^{\infty}(\R^{2d})$ be a radial function satisfying $\phi(\vec{\xi}) = 0$ for $|\vec{\xi}| \le 1$, and $\phi(\vec{\xi}) = 1$ for $|\vec{\xi}| \ge 2$. Given $a > 0$, $a \ne 1$, and $b > 0$, denote
\begin{equation*}
m_{a, b}(\vec{\xi})
:= \phi(\vec{\xi}) |\vec{\xi}|^{-b} e^{i |\vec{\xi}|^a}, \qquad \vec{\xi} = (\xi_1, \xi_2) \in \R^{2d}.
\end{equation*}
For a sequence $c = (c_k)_{k \in \Z}$ with $\sup_k |c_k| \le 1$, the {\tt bilinear oscillatory Fourier multiplier} is defined by
\begin{align*}
T_{m_{a, b}^c}(\vec{f})(x)
&:= \int_{\R^{2d}} m^{\c}_{a, b}(\vec{\xi}) 
e^{2 \pi i x \cdot (\xi_1 + \xi_2)} 
\widehat{f}_1(\xi_1) \, \widehat{f}_2(\xi_2) \, d\vec{\xi},
\end{align*}
where $m_{a, b}^c(\vec{\xi}) := \sum_{k \in \Z} c_k \, m_{a, b} (2^k \vec{\xi})$.

We achieve boundedness of bilinear oscillatory Fourier multipliers by \cite[Theorem 1.5]{CDY} and Theorem \ref{thm:spBFS} for abstract operators.
\begin{theorem}\label{thm:osc}
Let $a>0$, $a \ne 1$, $1 < r < 2$, and $b > a \, d/r$. Let $\X_1, \X_2$ be Banach function spaces such that 
\begin{align*}
\M_{2r, 2r, (2r)'}: \X_1^{\frac12} \times \X_2^{\frac12} \times \Y' \to L^1 \text{ is bounded},
\end{align*} 
where $\Y = \X^{\frac12}$ and $\X = \X_1 \X_2$. Then both $T_{m_{a, b}^c}$ and $[T_{m_{a, b}^c}, \b]_{\a}$ are bounded from $\X_1 \times  \X_2$ to $\X$ for all $\a \in \N^2$ and $\b \in \BMO^2$, uniformly in $c = (c_k)_{k \in \Z}$.
\end{theorem}

\subsection{Compactness of multilinear commutators}\label{sec:app-Tbcpt}
To prove Theorem \ref{thm:TbBTS}, we first treat the case $\vec{r} = (1, \ldots, 1)$ and $s=\infty$. Given $\delta > 0$, the bilinear Bochner--Riesz means is defined by 
\begin{align*}
\B^{\delta}(f, g)(x) := \int_{\R^{2d}} (1-|\xi|^2-|\eta|^2)^{\delta}_{+} \, 
\widehat{f}(\xi) \, \widehat{g}(\eta) e^{2\pi i x \cdot(\xi+\eta)} \, d\xi \, d\eta. 
\end{align*}

Recall the weighted boundedness of multilinear $\w$-Calder\'{o}n--Zygmund operators $T$ \cite[eq. (5.1)]{COY}, 
multilinear pseudo-differential operator $T_{\sigma}$ \cite[Theorem 2.22]{CIRXY}, and bilinear Bochner--Riesz means \cite[eq. (5.61)]{COY}. Besides, the compactness of their commutators are respectively obtained in 
\cite[Theorem 5.1]{COY}, \cite[Theorem 2.23]{CIRXY}, and \cite[Theorem 5.15]{COY}. We then utilize Theorem \ref{thm:Tbcpt} to conclude the following. 

\begin{theorem}\label{thm:CZOcpt}
Let $\b \in \CMO^m$ and $\X_1, \ldots, \X_m$ be Banach function spaces such that 
\begin{align*}
\M_{(m, \ldots, m), 1}: \X_1^{\frac1m} \times \cdots \times \X_m^{\frac1m} \times \Y' \to L^1 \text{ is bounded},
\end{align*} 
where $\Y = \X^{\frac1m}$ and $\X = \prod_{j=1}^m \X_j$. Then for each $j=1, \ldots, m$, both $[T, \b]_{e_j}$ and $[T_{\sigma}, \b]_{e_j}$ are compact from $\X_1 \times \cdots \times \X_m$ to $\X$. Moreover, in the case $m=2$ and $d \ge 2$, the same assertion holds for $[\B^{d-1/2}, \b]_{e_j}$.
\end{theorem}

In the limited range case, we establish the compactness of commutators of $T_{\Omega}$ on Banach function spaces.

\begin{theorem}\label{thm:roughcpt}
Let $q \in (1, \infty]$ and $\Omega \in L^q(\bS^{md-1})$ with $\int_{\bS^{md-1}} \Omega\, d\sigma=0$. Let $(\vec{r}, s) \in (1, \infty)^{m+1}$ be such that $\frac1s < \sum_{j=1}^m \frac{1}{r_j}$ and $\big(\frac{1}{r_1}, \ldots, \frac{1}{r_m}, \frac{1}{s'} \big) \in \H_{m+1}^d(q)$. Let $\b \in \CMO^m$ and $\X_1, \ldots, \X_m$ be Banach function spaces such that 
\begin{align*}
\M_{m \vec{r}, (ms)'}: \X_1^{\frac1m} \times \cdots \times \X_m^{\frac1m}\times \Y' \to L^1 \text{ is bounded},
\end{align*} 
where $\Y = \X^{\frac1m}$ and $\X = \prod_{j=1}^m \X_j$. Then for each $k=1, \ldots, m$, $[T_{\Omega}, \b]_{e_k}$ is compact from $\X_1 \times \cdots \times \X_m$ to $\X$.
\end{theorem}

\begin{proof}
In view of Theorem \ref{thm:Tbcpt} and \ref{thm:rough}, it suffices to 
\begin{align}\label{eq:TOb}
[T_{\Omega}, \b]_{e_k} \text{ is compact from $L^{2m}(\Rd) \times \cdots \times L^{2m}(\Rd)$ to $L^2(\Rd)$}. 
\end{align}
Indeed, similar to that of \cite[Theorem 5.13]{COY}, the proof of \eqref{eq:TOb} is mainly based on the frequency decay estimates of \cite{DS, GHH}. It also needs the bounded of $[T_{\Omega}, \b]_{e_k}$ (see Theorem \ref{thm:rough}) and the compactness of multilinear Calder\'{o}n--Zygmund operators (see Theorem \ref{thm:CZO}) on Lebesgue spaces. Moreover, \cite[Lemma 2.11]{COY} provides a way to transfer the compactness of single-scale operators to the original operator. We omit the details. 
\end{proof}

Next, we proceed to multilinear Fourier multipliers, which is defined by
\begin{align*}
T_{\m}(\vec{f})(x) := \int_{\R^{md}} \m(\vec{\xi})
e^{2\pi i x \cdot (\xi_1 + \cdots + \xi_m)} \widehat{f}_1(\xi_1)
\cdots \widehat{f}_m(\xi_m)  \, d\vec{\xi},
\end{align*}
for all $f_i \in \S(\Rd)$, $i=1,\ldots,m$. Given $s \in \R$, the Sobolev space $W^s(\R^{md})$ is defined by the norm
\begin{align*} 
\|f\|_{W^s(\R^{md})} 
:=\bigg(\int_{\R^{md}} (1+|\vec{\xi}|^2)^{s}|\widehat{f}(\vec{\xi})|^2 d\vec{\xi}\bigg)^{\frac12}.
\end{align*}
Let $\Phi \in \S(\R^{md})$ satisfy $\supp \Phi \subset \{\frac12 \le |\vec{\xi}| \le 2\}$ and $\sum_{j \in \Z} \Phi(2^{-j} \vec{\xi}) = 1$ for each $\vec{\xi} \in \R^{md} \setminus\{0\}$. Denote $\m_j(\vec{\xi}) :=\Phi(\vec{\xi}) \m(2^j \vec{\xi})$ for each $j \in \Z$. Denote  
\begin{align*}
\mathcal{W}^s(\R^{md}) 
:= \big\{\m \in L^{\infty}(\R^{md}): \sup_{j \in \Z} \|\m_j\|_{W^s(\R^{md})}<\infty\big\}. 
\end{align*}
The boundedness and compactness of $T_{\m}$ can be found in \cite{COY, FT12, FT14} and \cite{COY, Hu14, Hu17} respectively. In particular, it was shown in \cite[eq. (5.19) and Theorem 5.9]{COY} that for every $\m \in \mathcal{W}^s(\R^{md})$ with $s \in (md/2, md]$ and $\vec{r} \in (1, 2)$ with $\frac{s}{n} = \sum_{j=1}^m \frac{1}{r_j}$, 
\begin{align*}
T_{\m} : & \, L^{p_1}_{w_1} \times \cdots \times L^{p_m}_{w_m} \to L^p_w \, \text{ is bounded},
\\
[T_{\m}, \b]_{e_j}: & \, L^{p_1}_{w_1} \times \cdots \times L^{p_m}_{w_m} \to L^p_w \, \text{ is compact},
\end{align*}
for all $\b \in \CMO^m$, $\frac1p = \sum_{j=1}^m \frac{1}{p_j}$ with $\vec{r} \prec \vec{p}$, and $\vec{w} \in A_{\vec{p}, (\vec{r}, 1)}$, where $w=\prod_{j=1}^m w_j$. Hence, by Theorem \ref{thm:Tbcpt}, the above result allows us to obtain boundedness and compactness on Banach function spaces. 

\begin{theorem}\label{thm:FIO}
Let $\m \in \mathcal{W}^s(\R^{md})$ with $s \in (md/2, md]$ and $\vec{r} \in (1, 2)$ with $\frac{s}{n} = \sum_{j=1}^m \frac{1}{r_j}$. Let $\b \in \CMO^m$ and $\X_1, \ldots, \X_m$ be Banach function spaces such that 
\begin{align*}
\M_{m \vec{r}, m'}: \X_1^{\frac1m} \times \cdots \times \X_m^{\frac1m}\times \Y' \to L^1(\Rd) \text{ is bounded},
\end{align*} 
where $\Y = \X^{\frac1m}$ and $\X = \prod_{j=1}^m \X_j$. Then for each $j=1, \ldots, m$, $[T, \b]_{e_j}: \X_1 \times \cdots \times \X_m \longrightarrow \X$ is compact.
\end{theorem}

\subsection{Compactness of bilinear Calder\'{o}n--Zygmund operators}\label{sec:app-CZOcpt}
This section is devoted to proving the compactness of a class of bilinear Calder\'{o}n--Zygmund operators, which satisfy the following hypotheses: 
\begin{list}{(\theenumi)}{\usecounter{enumi}\leftmargin=1cm \labelwidth=1cm \itemsep=0.1cm \topsep=0.2cm \renewcommand{\theenumi}{H\arabic{enumi}}}
 
\item\label{H1} $T$ is associated with a compact bilinear Calder\'{o}n--Zygmund kernel, 

\item\label{H2} $T$ satisfies the weak compactness property, 

\item\label{H3} $T(1,1), T^{*1}(1,1), T^{*2}(1,1) \in \CMO(\Rd)$. 
\end{list}
The reader is referred to \cite{CLSY} for a precise definition.

Let us present some examples of bilinear Calder\'{o}n--Zygmund operators which satisfy the hypotheses \eqref{H1}--\eqref{H3}. Let $\Phi, \Psi \in \mathscr{C}_c^{\infty}(\Rd)$ be radial functions such that $\supp (\Phi) \subset B(0, 1)$, $\int_{\Rd} \Phi (x) \, dx = 1$, $\int_{\Rd} \Psi(x) \, dx = 0$, and $\int_0^{\infty} |\widehat{\Psi}(t e_1)|^2 \frac{dt}{t} =1$, where $e_1 := (1, 0, \ldots, 0)$. For any $t>0$, define convolution operators 
\begin{align*}
P_t f := \Phi_t *f 
\quad\text{ and }\quad 
Q_t f := \Psi_t * f, 
\end{align*} 
where $\Phi_t(x) := t^{-n} \Phi(t^{-1} x)$ and $\Psi_t(x) := t^{-n} \Psi(t^{-1} x)$. Given $b \in \BMO(\Rd)$, the {\tt bilinear continuous paraproduct} $\pi_b$ is defined by 
\begin{align*}
\pi_b (f, g) 
:= \int_0^{\infty} Q_t \big( (Q_t b) (P_t f) (P_t g) \big) \, \frac{dt}{t}. 
\end{align*}
For studies on paraproducts, see \cite{Bony, CM78, DJ, Yab}.

Let us introduce bilinear dyadic paraproducts, which have developed into a fundamental tool in bilinear dyadic analysis, see \cite{CDLSY, CLSY, CY}. Fix a dyadic grid $\D$ and let $\{\psi_I\}_{I \in \D}$ be a system of wavelets. Given a function $b \in \BMO(\Rd)$, the {\tt bilinear dyadic paraproduct} are defined by
\begin{align*}
\Pi_b(f_1, f_2) 
:= \sum_{I \in \D} \langle b, \psi_I \rangle \langle f_1 \rangle_I \langle f_2 \rangle_I \,  \psi_I. 
\end{align*}
In addition, to investigate the compactness of bilinear pseudo-differential operators $T_{\sigma}$, we recall a new class of symbols used in \cite{CLSY}. Given $\tau \in \R$ and $\rho, \delta \in [0, 1]$, we say that $\sigma \in \mathcal{K}_{\rho, \delta}^{\tau}$ if it satisfies 
\begin{align*}
\big|\partial_{x}^{\alpha} \partial_{\xi}^{\beta} \partial_{\eta}^{\gamma} \sigma(x, \xi, \eta) \big|
\leq C_{\alpha, \beta, \gamma}(x, \xi, \eta) (1+ |\xi| + |\eta|)^{\tau + \delta|\alpha| - \rho (|\beta| + |\gamma|)}, 
\end{align*}
for all multi-indices $\alpha$, $\beta$ and $\gamma$, where $C_{\alpha, \beta, \gamma}$ is a bounded function satisfying 
\begin{align*}
\lim_{|x| + |\xi| + |\eta| \to \infty} C_{\alpha,\beta, \gamma}(x, \xi, \eta)
=0.
\end{align*}

It was shown in \cite{CLSY} that
\begin{equation}\label{eq:clsy-1}
\begin{aligned}
& \text{bilinear Calder\'{o}n--Zygmund operators satisfying \eqref{H1}--\eqref{H3} are compact}
\\
& \quad \text{ from $L^{p_1}_{w_1} \times \cdots \times L^{p_m}_{w_m}$ to $L^p_w$ for all $\vec{p} \in (1, \infty)^m$ and $\vec{w} \in A_{\vec{p}}$},
\end{aligned}
\end{equation}
and 
\begin{align}\label{eq:clsy-2}
\text{$\mathbb{T}$ is a bilinear Calder\'{o}n--Zygmund operators and satisfies \eqref{H1}--\eqref{H3}},
\end{align}
for every $\mathbb{T} \in \big\{\pi_b, \Pi_b, T_{\sigma}, [\mathfrak{b}, T_a]_1, [\mathfrak{b}, T_a]_2 \big\}$, where $\sigma \in \mathcal{K}_{1, \delta}^0$ with $\delta \in [0, 1)$, $a \in \mathcal{K}_{1, 0}^1$, and $\nabla \mathfrak{b} \in L^{\infty}(\Rd)$.

Then the following result is a consequence of \eqref{eq:clsy-1}--\eqref{eq:clsy-2} and Theorem \ref{thm:BFScpt}.

\begin{theorem}\label{thm:CZOcpt}
Let $b \in \CMO$, $\sigma \in \mathcal{K}_{1, \delta}^0$ with $\delta \in [0, 1)$, $a \in \mathcal{K}_{1, 0}^1$, and $\mathfrak{b}$ be a Lipschitz function with $\nabla \mathfrak{b} \in L^{\infty}(\Rd)$. Let $\X_1, \ldots, \X_m$ be Banach function spaces such that 
\begin{align*}
\M_{(m, \ldots, m), 1}: \X_1^{\frac1m} \times \cdots \times \X_m^{\frac1m}\times \Y' \to L^1(\Rd) \text{ is bounded},
\end{align*} 
where $\Y = \X^{\frac1m}$ and $\X = \prod_{j=1}^m \X_j$. Then $\mathbb{T}: \X_1 \times \cdots \times \X_m \longrightarrow \X$ is compact for every $\mathbb{T} \in \big\{\pi_b, \Pi_b, T_{\sigma}, [\mathfrak{b}, T_a]_1, [\mathfrak{b}, T_a]_2 \big\}$.
\end{theorem}

\end{document}